\documentclass{amsart}

\usepackage{phaistos}
\usepackage{amssymb}
\usepackage{enumerate, xspace}
\usepackage{enumitem}
\usepackage{mathrsfs}
\usepackage[normalem]{ulem}
\usepackage{bbm}       
\usepackage[backgroundcolor=blue!20!white, linecolor=blue!20!white,
textsize=footnotesize]{todonotes}
\usepackage{verbatim}
\usepackage[colorlinks=true,linkcolor=blue,citecolor=magenta]{hyperref}
\numberwithin{equation}{section}
\usepackage{stmaryrd}

\usepackage{enumitem}
\usepackage{hyperref}

\usepackage{crossreftools}

\makeindex
\newtheorem{theorem}{Theorem}[section]
\newtheorem{lemma}[theorem]{Lemma}
\newtheorem{proposition}[theorem]{Proposition}

\newtheorem{definition}[theorem]{Definition}
\newtheorem{remark}[theorem]{Remark}

\newcommand{\mylabel}[2]{#2\def\@currentlabel{#2}\label{#1}}

\newcommand\ve{\varepsilon}

\newcommand\vf{\varphi}

\renewcommand{\ge}{\geqslant}
\renewcommand{\le}{\leqslant}

\renewcommand{\tilde}{\widetilde}

\newcommand\cA{{\mathcal A}}
\newcommand\cB{{\mathcal B}}
\newcommand\cC{{\mathcal C}}
\newcommand\cD{{\mathcal D}}

\newcommand\cF{{\mathcal F}}
\newcommand\cG{{\mathcal G}}
\newcommand\cH{{\mathcal H}}

\newcommand\cI{{\mathcal I}}
\newcommand\cJ{{\mathcal J}}
\newcommand\cL{{\mathcal L}}
\newcommand\cO{{\mathcal O}}
\newcommand\cM{{\mathcal M}}

\newcommand\cQ{{\mathcal Q}}
\newcommand\cR{{\mathcal R}}
\newcommand\cS{{\mathcal S}}

\newcommand\cW{{\mathcal W}}
\newcommand\cZ{{\mathcal Z}}

\newcommand\bN{{\mathbb N}}

\newcommand\bQ{{\mathbb Q}}
\newcommand\bR{{\mathbb R}}

\DeclareMathAlphabet{\mymathbb}{U}{BOONDOX-ds}{m}{n}

\newcommand{\Deriv}{\mathcal D}

\newcommand{\sing}{\nu^{\text{sing}}}

\newcommand{\catt}{\cF_{A,t}}
\newcommand{\cato}{\cF_{A,0}}
\newcommand{\catjt}{F_{A,j,t}}

\newcommand{\cov}{ \text{Cor} }
\newcommand{\Center}{\text{Center}}
\newcommand{\Coup}{\text{Coup}}
\newcommand{\diam}{\operatorname{diam}}
\newcommand{\Int}{\operatorname{Int}}

\begin{document}

\title{On linear response for discontinuous perturbations of smooth endomorphisms}

\author{Giovanni Canestrari}
\address{Department of Mathematics, University of Toronto, 
40 St.~George Street, Toronto, ON M5S 2E4, Canada. Email:
\texttt{giovanni.canestrari@utoronto.ca}.}

\date{\today}

\begin{abstract}
We consider discontinuous perturbations of smooth endomorphisms and show that if the perturbed family satisfies uniform mixing assumptions on standard pairs the physical measure is Lipschitz in the parameter defining the perturbation. We also study the problem of linear response for this class of perturbations. Finally we discuss the applicability of the abstract assumptions proving linear response for a concrete example.
\end{abstract}

\maketitle
\tableofcontents

\section{Introduction}

The dependence of the SRB measure on the parameters of the system is a natural question in smooth ergodic theory which has seen wide interest in recent years. While on physical grounds it is reasonable to expect that the averages of smooth observables change in a smooth way when varying the parameters smoothly, counterexamples show that this is not always the case (a famous one being the quadratic family \cite{BBS15}).\\
Let \(\{\cF_{t}: \cM \to \cM\}_t\) be a family of dynamical systems such that for each element of the family there exists a unique SRB measure \(\mu_t\) and let \(\vf: \cM \to \bR\) be a smooth observable. Linear response means that the function \(t \to \mu_{t}(\vf)\) is differentiable.\\
It is common knowledge that proving linear response, whenever it holds, is significantly more challenging than showing continuity. However, this is not just an interesting mathematical problem: there are cases in which linear response is really the property that lies at the foundation of a rigorous mathematical justification of a physical theory (see e.g. \cite{CELS93, BDL00} for a derivation of the Ohm law in a periodic Lorentz gas. In this case the conductivity is expressed in terms of microscopical quantities via linear response).\\
Nowadays, linear response has been proved in many situations and the \textit{smooth} uniformly hyperbolic setting seems to be well understood. Ruelle has proved differentiability for smooth hyperbolic diffeomorphisms \cite{Ruelle97, Ruelle04} and Anosov flows \cite{Ruelle08}. Similar results can also be obtained using transfer operators acting on Banach spaces, as has been done in \cite{BL07}. Beyond uniform hyperbolicity, there have recently been studies in many directions: e.g., partially hyperbolic systems \cite{Dolgopyat04}, intermittent maps \cite{BT16, Korepanov}, random systems \cite{BRS20}. For a broader and more complete review of the state of the art, we refer to \cite{Baladi14,Ruelle09} and the references therein.

The situation appears to be less neat when it comes to systems with discontinuities. In the one-dimensional case, even uniformly expanding systems may lack statistical stability or linear response. Results in the literature \cite{Baladi07,BS08,LS18} suggest that the main obstruction to linear response occurs when the perturbation changes the topological class of the map at first order (i.e., when the perturbation is transversal).\footnote{However, there are situations in which linear response may hold also for transversal perturbations. In \cite{BG24} it is shown how unbounded expansion can compensate for the lack of smoothness, providing yet another mechanism for linear response.} In the case of transversal families of piecewise \(\cC^2\) one-dimensional uniformly expanding maps, it is shown that the SRB measure has a dependence not worse than \(t|\ln t|\). This follows from perturbation theory \cite{Keller82,KL99}. If optimal, this would prevent linear response, due to the presence of the logarithm. In \cite{LS18}, it is indeed proved that for a wide family of transversal perturbations, the modulus of continuity of \(\mu_t\) is of the order of \(t\sqrt{|\ln t|}\), demonstrating incompatibility with linear response. Recently, there has been much interest also in the case of discontinuous systems in higher dimensions (see e.g. \cite{Wormell22, Wormell24}), and although linear response beyond the regime of smooth uniform hyperbolicity is somehow believed to be true, it remains an open and challenging problem. Such conviction has driven the search for new mechanisms for linear response. This task has been addressed for instance in \cite{Ruelle18}, where it is argued that a large enough stable dimension yields linear response for smooth non-uniformly hyperbolic systems with tangencies.

With the present work, we show that discontinuous uniformly hyperbolic systems can exhibit linear response, also for families not tangent to a conjugacy class, contrary to the one-dimensional case. This shows the existence of a new mechanism yielding linear response. While in 1D the discontinuities generate delta functions, which do not converge to the absolutely continuous invariant measure when pushed forward by the dynamics, in D\(>1\) (with a contracting direction transversal to the image of the discontinuity set) the discontinuities give rise to measures supported on submanifolds (e.g. standard pairs) that have a chance to converge to the SRB measure.

The idea of using standard pairs to study linear response is not new. In \cite{Dolgopyat04} it has been successfully carried out to study perturbations of partially hyperbolic systems. Other works related to the present one include \cite{CELS93}, which is an example of linear response for a perturbation of a discontinuous system. More specifically, \cite{CELS93} studies billiards in the presence of a small electric field. However, \cite{CELS93} uses some special properties of the system, most notably that the perturbed flow remains 1-1 and onto. Here, we consider the more general situation in which the perturbation can change the number of preimages of points of a measure proportional to the parameter defining the perturbation. In other words, linear response can hold also for perturbations that change at first order the topological class of the unperturbed dynamics (this is the case of the example in Section \ref{sec:examples}).

The present analysis is restricted to discontinuous perturbations around a map that have a smooth invariant density. Since SRB measures of hyperbolic maps are in general distributions (see e.g. \cite{DKL21}), this is a rather restrictive assumption. To check if the strategy we propose can be extended to more general cases requires further work. Unfortunately, the existing literature does not provide the needed technical tools even for studying a simple example (see the discussion in Section \ref{sec:examples} before Proposition \ref{prop:fundamental-example}). To remedy this, in Section \ref{sec:coupling}, we provide the results on the behavior of standard pairs strictly needed to treat the example discussed here.

To test our ideas, we consider the following discontinuous perturbation of the famous Arnold cat map on the unit square
\[
    \begin{pmatrix}x\\y 
\end{pmatrix} \mapsto  \begin{pmatrix} \{x + y\}\\
\{x+(2-t)y\}\end{pmatrix},
\]
where \(t \in [0, 1/8]\) and \(\{a\}\) is the fractional part of \(a\). The map \(\cato\) models a conservative dynamical system (e.g., a Hamiltonian one), while \(\catt\) models the effect of some non-conservative perturbation: as \(t\) increases, a strip-shaped ``hole'' of size \(t\) opens in the image.


\begin{tikzpicture}[scale=2.5]
{\crtcrossreflabel{(Fig.1)}[Fig:1]}
\hspace{0.8cm}


 \begin{scope}[xshift = -0.3cm]
  \coordinate (a3) at (0,0);
      \coordinate (b3) at (1,0);
      \coordinate (c3) at (0,1);
      \coordinate (d3) at (1,1);
      
      \coordinate (ac) at (0,8/15);
      \coordinate (bd) at (1,8/15);
      \coordinate (cd) at (1/8, 1);

      \fill[red] (a3) -- (ac)  -- (b3) -- cycle;
      \fill[blue] (ac) -- (c3)  -- (b3) -- cycle;
      \fill[green] (c3) -- (cd)  -- (bd) -- (b3)-- cycle;
      \fill[brown] (cd) -- (d3)  -- (bd) -- cycle;

      \draw [ultra thick] (a3) -- (b3) -- (d3) --(c3)--(a3);
      \draw [very thick] (b3)--(ac);
      \draw [very thick] (b3)--(c3);
      \draw [very thick] (bd)--(cd);

  \begin{scope}[xshift=3cm]

      \coordinate (a3) at (0,0);
      \coordinate (b3) at (1,0);
      \coordinate (c3) at (0,1);
      \coordinate (d3) at (1,1);
      
      \coordinate (cd1) at (1/8, 1);
      \coordinate (ac) at (0, 1-1/8);
      \coordinate (ab) at (8/15, 0);
      \coordinate (cd2) at (8/15, 1);
      \coordinate (ab1) at (1/8, 0);
      \coordinate (bd1) at (1,1-1/8);

      \fill[red] (a3) -- (cd2)  -- (d3) -- cycle;
      \fill[brown] (ab1) -- (bd1)  -- (ab) -- cycle;
      \fill[green] (a3) -- (ac)  -- (cd1) -- (cd2) -- cycle;
      \fill[blue] (ab) -- (b3)  -- (bd1)  -- cycle;

      \draw [ultra thick] (a3) -- (b3) -- (d3) --(c3)--(a3);
      \draw [very thick] (a3)--(d3);
      \draw [very thick] (cd1)--(ac);
      \draw [very thick] (ab1)--(bd1);
      \draw [very thick] (a3)--(cd2);
      \draw [very thick] (bd1)--(ab);

  \end{scope}

 ;

  \draw[<-] (2.6,1) to [bend right] node [midway, above] { $ \catt $ } (1.4,1);
\end{scope}
  \node at (1.8,-0.3) [align = left] {Figure 1: a perturbation of the Arnold cat map.}; 
\end{tikzpicture}

Denote by \(1\) the constant probability density corresponding to the SRB measure for the unperturbed cat map and by \(\cL_t^A\) the transfer operator associated to the perturbed dynamics, which evolves the densities. This perturbation is pathological in the sense that the expression \(( \cL^A_t 1 - 1)/t\) has a \(\sim 1/t\) blow-up in the uncolored portion of the square in the right hand side. This happens because mass does not reach those regions. On the other hand, since the perturbation moves the discontinuities in a direction transversal to the unstable cone, the resulting inhomogeneities in the mass distribution are smooth along \textit{some} direction in the unstable cone and are therefore “regular” from the viewpoint of a hyperbolic system. To make this precise, we cover the hole with unstable curves (segments in this case), whose lengths are of order one, uniformly in the size of the perturbation. We formalize this in Assumption~\ref{ass:foliation}, where we also allow for a small portion of the hole to be covered by short curves (see Assumptions~\ref{ass:foliation}–\ref{ass:good-initial-cond1} and \ref{ass:foliation}–\ref{ass:good-initial-cond2}). This phenomenon also occurs for the family \(\catt\) due to the little white triangle inside the square on the right of \ref{Fig:1}. The fact that most curves are uniformly long in $t$ yields the uniform decay of correlations needed to establish linear response. Denote by \(m\) the Lebesgue measure and by \(\mu_t\) the SRB measure of \(\catt\). In Section \ref{sec:examples} we prove that, \mbox{for any smooth observable \(\vf\),} 
 \[
    \lim_{t \to 0}\frac{\mu_{t}(\vf) - m(\vf)}{t} = \sum_{k =0}^{\infty}\left[m(\vf) - \int_0^1 (\vf \circ \cato^k)(s,s)ds \right].
 \]
The rightmost term in the equation above is the push-forward of the measure obtained by integration along the diagonal of the square. The diagonal appears because it is where the inhomogeneities of mass in \(( \cL^A_t 1 - 1)/t\) concentrate.

In order to treat the general case, a complete theory for using standard pairs in the setting of hyperbolic piecewise smooth endomorphisms must be developed first. Toward this goal, we introduce an alternative coupling method with respect to the literature, which simplifies the standard arguments by avoiding the use of stable/unstable holonomies on Cantor sets (e.g. \cite{C99}). 

If such a theory is completed, it should be possible to apply it to manifold new cases. This may even allow investigating the situation in which the unperturbed SRB measure is not absolutely continuous or studying Lipschitz dependence of the conditional survival probability for small holes in hyperbolic systems.

The paper is structured as follows: Section \ref{sec:setting} is devoted to presenting the class of maps and perturbations that we care about and the results. In Section \ref{sec:Lip} we prove Theorem \ref{thm:Lip} about Lipschitz dependence. Under additional assumptions, in Section \ref{sec:linear-response} we prove Proposition \ref{prop:linear-resp-accumulation} needed to prove Theorem \ref{thm:linear-resp} about linear response. In Section \ref{sec:examples} we provide a linear response formula for a concrete example and in Section \ref{sec:coupling} we prove Proposition \ref{prop:fundamental-example} that provides the needed tools on standard pairs for the example.

Throughout the paper \(\bN = \{1,2,...\}\), \(\bN_0 = \bN \cup \{0\}\) and \(C \in \bR^+\) is a constant that depends only on the family \(\{(\cF_{j,t}, \cM_{j,t}^{+})\}_{t \in [0,\ve]}\). For \(x,y \in \cM\), \(d(x,y)\) is the Euclidean distance between \(x\) and \(y\). \(m\) is the Lebesgue measure. For \(A, B \subseteq \bR^2\), \(s \in \bR^+\), \([A]_{s} = \{x \in \cM: d(x,A) \le s\}\), \(A \triangle B = (A \setminus B) \cup (B \setminus A)\) and \(d(A,B)\) is the Hausdorff distance between \(A\) and \(B\). Given a disjoint union of curves \(\Gamma\), we let \(\left|\Gamma\right|\) be the sum of their lengths. For \(\cM = (0,1)^2\), we will denote by \(\cC^k(\cM)\) the set of differentiable functions such that \(f^{(p)}\), \(p \in \{0,...,k\}\), admit a continuous extension up to the boundary of \(\cM\), making \(\cC^k(\cM)\) Banach spaces. Finally, by measure we mean signed measure unless otherwise specified.\\

\textbf{Acknowledgments:} I am grateful to Dmitry Dolgopyat for suggesting the question that inspired the present work and sharing with me his insights on the topic on many occasions. I am also grateful to Carlangelo Liverani for patiently explaining coupling to me and for key discussions at the beginning and throughout the whole project. I wish to thank P\'eter B\'alint, Shuang Chen and Nicholas Fleming-Vázquez for stimulating discussions. Finally, I wish to thank the referees for their thoughtful comments, which improved the presentation of this work. This work was supported by the PRIN Grants ``Regular and stochastic behavior in dynamical systems" (PRIN 2017S35EHN), ``Stochastic properties of dynamical systems" (PRIN 2022NTKXCX), and by the MIUR Excellence Department Projects MatMod@TOV and Math@TOV awarded to the Department of Mathematics, University of Rome Tor Vergata. The author acknowledges membership to the GNFM/INDAM and the hospitality of Maryland University and Simons Center where part of this work was done. This work is part of the author's PhD project at Roma Tor Vergata and his activity within the DinAmicI group.

\section{Setting and results}
\label{sec:setting}
Let \(\cM = (0,1)^2\) and \(\cF\) be a \textit{piecewise smooth endomorphism} as explained below.
Let \(\cJ \subset \bN\) be a finite set of indices and \(\cM_j^+ \subseteq \cM\), \(j \in \cJ\), be open disjoint path connected sets such that \(\partial \cM_{j}^+\) is a closed curve and is a finite union of \(\cC^1\) curves of finite length that intersect at their endpoints. Moreover, \(\overline \cM = \bigcup_{j \in \cJ}\overline \cM_j^{+}\). For \(j \in \cJ\), we consider diffeomorphisms 
\[
\cF_j: \cM_j^{+} \to \cM_j^- = \cF_j(\cM^+_j)
\]
such that \(\bigcup_{j \in \cJ} \cM_j^{-} \subseteq \cM\). We also introduce the sets of singularities 
\[
\cS^{+} = \bigcup_{j \in \cJ}\partial \cM_j^{+}, \quad  \quad \cS^{-} =  \bigcup_{ j\in \cJ} \partial \cM_j^{-}.
\]
The maps \(\cF_j\) define a map \(\cF\) from \(\cM \setminus \cS^+\) to its image by applying \(\cF_j\) for \(x \in \cM_j^+\). If the sets \(\cM_j^{-}\) are pairwise disjoint (that guarantees injectivity) and their union is equal to \(\cM \setminus \cS^{-}\) (that guarantees surjectivity), the map \(\cF : \cM \setminus \cS^+ \to \cM \setminus \cS^-\) is a diffeomorphism but we are considering also more general situations. Any piecewise endomorphism is uniquely determined by the pairs \(\{(\cF_j, \cM_j^+)\}_{j \in \cJ}\). We now define a subclass of maps with the needed uniform properties.
\begin{definition}\label{def:uniform}
Let \(\mathfrak S = \mathfrak S_{ D}\), \(D \in \bR^+\), be the set of piecewise endomorphisms such that for all \(\{(\cF_j, \cM_j^+)\}_{j \in \cJ} \in \mathfrak S\):
\begin{enumerate}
\item  [{\crtcrossreflabel{(U1)}[ass:unif-extension0]}] The cardinality of \(\cJ\) is constant and the lengths of \(\partial \cM^+_{j}\) are uniformly bounded.
\item [{\crtcrossreflabel{(U2)}[ass:unif-extension]}] For any \(j \in \cJ\), the map \(\cF_j: \cM_{j}^{+} \to \cM_{j}^{-}\) admits a \(\cC^2\) invertible extension \(F_j\) to a ball \(U_{\cM}\) containing \(\cM\). Moreover, 
\[
\max_{j \in \cJ}\max\left\{\|F_j\|_{\cC^2}, \|F_j^{-1}\|_{\cC^2 } \right\} \le D.
\]
\end{enumerate}
\end{definition}
The sets \(\cS^{+}\) and \(\cS^{-}\) consist of a finite union of \(\cC^1\) curves of finite length and, for each \(k \in \bN\), define
\begin{equation}\label{eq:future-disc}
    \cS^{+,k} = \bigcup_{p=0}^{k-1}\cF^{-p}(\cS^+)  \setminus \partial \cM.
\end{equation} 
Note that \(\cS^{+} \setminus \partial \cM = \cS^{+,1}\) and \(\cF^k\) is discontinuous at most on \(\cS^{+,k}\). We assume that  the sets \(\cS^{+,k}\) and \( \cS^-\) intersect in finitely many points and are transversal. I.e., for any two \(\cC^1\) curves \(\gamma \subseteq \cS^{+,k} \) and \(\widetilde \gamma \subseteq \cS^{-}\), \(\gamma\) and \(\widetilde \gamma\) are transversal at \(\gamma \cap \widetilde \gamma\). The map \(\cF\) is defined on \(\cM \setminus \cS^{+}\), but it is possible to extend the action of \(\cF\) to \(\overline{\cM \setminus \cS^{+}} = [0,1]^2\) by defining \(\cF(\cS^+)\) via limits of images of points in \(\cM \setminus \cS^+\). In general, the resulting map depends on the choice of the particular limiting sequence (e.g. approximating the value of the map from the left or the right), but this ambiguity will not play any role in the sequel. For \(\ve \in \bR^+\), let \(\left\{(\cF_{j,t}, \cM^{+}_{j, t})\right\}_{t \in [0,\ve]}\) be a family of elements of \(\mathfrak S\) indexed by the real parameter \(t\) and let \(\cS^{\pm}_t\), \(F_{j,t}\) and \(\cF_t\) be the corresponding objects. 

\begin{definition}\label{def:admissible-family} We say that \(\left\{(\cF_{j,t}, \cM_{j,t}^{+})\right\}_{t \in [0,\ve]}\) is \textit{an admissible family} if conditions (H1)-(H4) are verified.
\end{definition}
\textit{\begin{enumerate}
\item[{\crtcrossreflabel{(H1)}[ass:sing-near]}] There exists \(Q \in \bR^+\) such that, for all \(t \in [0, \ve]\) and \(j \in \cJ\),  
\[
\begin{split}
&\hspace{1.8cm}\left\|F_{j,t} - F_{j,0}\right\|_{\cC^2}\le Q t, \\
& \cM_{j,0}^{+} \triangle \cM_{j,t}^{+} \subset [\partial \cM_{j,0}^{+} \setminus \partial \cM]_{Q t} \cap [\partial \cM_{j,t}^{+}\setminus \partial \cM]_{Q t}.
\end{split}
\]
\item[{\crtcrossreflabel{(H2)}[ass:smooth-0]}] We assume that \(\cF_0\) preserves a measure \(\mu_0 \ll m\) and \(h_0:=\frac{d\mu_0}{dm} \in \cC^2(\cM, \bR)\).\\
\item [{\crtcrossreflabel{(H3)}[ass:srb]}] For all \(t \in [0, \ve]\), there exists a measure \(\mu_t\) such that, for any \(\vf \in \cC^0(\cM)\), 
\[
\mu_t(\vf) = \lim_{n \to \infty} \mu_0 (\vf \circ \cF_t^n).
\] 
\end{enumerate}
}

The next assumption requires a couple of preliminaries. Let \(W\) be a \(\cC^1\) curve and \(\rho \in \cC^1(W, \bR)\). We set, for \(t \in [0, \ve]\), \(\vf \in \cC^0(\cM)\),  \(n \in \bN_0\),
\begin{equation}\label{eq:cov-def-general}
    \cov(\vf, n, t; \rho) = \left|\int_{W} \left(\vf \circ \cF_t^n\right) \rho - \mu_t(\vf) \int_{W}\rho\right|.
\end{equation}
Let \(\mathscr P_t\) be a measurable partition of \(\cM \setminus (\cS_0^- \cup \cS_{t}^{-})\) in \(\cC^1\) curves. There exist conditional probability densities \(p_{W}: W \to \bR^+\), \(W \in \mathscr P_t\), and a disintegration\footnote{See \cite{Rohlin52} for the existence of disintegrations for measurable partitions.} \(m_{\mathscr P_t}\) of the Lebesgue measure such that for any bounded \mbox{measurable function \(g: \cM \to \bR\),}
\begin{equation}\label{eq:disintegration}
\int_{\cM}gdm = \int_{\mathscr P_t} \left(\int_{W}g p_W \right)dm_{\mathscr P_t}(W).
\end{equation}

\textit{
\begin{enumerate}
\item[{\crtcrossreflabel{(H4)}[ass:foliation]}]
There exists a sequence of positive real numbers \(\{\Theta_n\}_n\), \(\sum_{n=0}^{\infty}\Theta_n < \infty\) and, for every \(t \in [0,\ve]\), a measurable partition \(\mathscr P_t\) of \(\cM \setminus (\cS_0^- \cup \cS_{t}^{-})\) with the following property. We assume that \(\mathscr P_t\) can be partitioned into two \(m_{\mathscr P_t}-\)measurable sets \(\mathscr P^1_t\) and \(\mathscr P^2_t\) such that: For all \(t \in (0,\ve]\) and \(\vf \in \cC^1(\cM)\), \(n \in \bN_0\),\\
\begin{enumerate}
  \item[{\crtcrossreflabel{(a)}[ass:good-initial-cond1]}] \!\!\(\cov\left(\vf, n , t; \rho_W p_{W}\right) \le \Theta_n \|\vf\|_{\cC^1}\|\rho_W\|_{\cC^1}, \text{  } \forall W \in \mathscr P_t^{1}, \text{ } \forall \rho_W \!\in\! \cC^1(W, \bR);\)\\
  \item[{\crtcrossreflabel{(b)}[ass:good-initial-cond2]}] \(t^{-1}\int_{\mathscr P^2_t} \cov\left(\vf, n , t; g p_{W}\right) dm_{\mathscr P_t}(W) \le \Theta_n\|\vf\|_{\cC^1}\|g\|_{\cC^1}\), \text{  }\(\forall g \in \cC^1(\cM)\).\\
\end{enumerate} 
\end{enumerate}
}

We refer to the end of this section and to the introduction for a discussion of the assumptions. From now on, let \(\left\{(\cF_{j,t}, \cM_{j,t}^+)\right\}\) be an admissible family and \(\mu_t\) the corresponding measures given by \ref{ass:srb}. The next two theorems are the main result of this work.
\begin{theorem}[Lipschitz continuity]\label{thm:Lip}
 For all \(t \in  [0, \ve]\), and each \(\vf \in \cC^1(\cM)\),
  \[
    \left| \mu_t(\vf) - \mu_0(\vf) \right| \le C\|\vf\|_{\cC^1} t.
  \]
\end{theorem}
The proof of Theorem \ref{thm:Lip} is presented in Section \ref{sec:Lip}. Recall that \(h_0\) was introduced in \ref{ass:smooth-0}. For \(t \in [0,\ve]\), \(\vf \in \cC^0(\cM)\), set
\[
   \nu_t (\vf) = \int_{\cM} \left(\vf \circ \cF_t - \vf\right) h_0dm.
\]
\begin{theorem}[Linear Response] \label{thm:linear-resp}
 If there exists a measure \(\widetilde \nu\) such that
\begin{equation}\label{eq:limit-at-beginning}
  \lim_{t \to 0} \frac{\nu_t (\vf) }{t} = \widetilde \nu(\vf), \quad \forall \vf \in \cC^0(\cM),
 \end{equation}
then linear response holds and, for any \(\vf \in \cC^1(\cM)\),
\[
\lim_{t \to 0} \frac{\mu_t(\vf) - \mu_0(\vf)}{t} = \sum_{k \in \bN_0} \widetilde\nu \left(\vf \circ \cF_0^k\right).
\]
\end{theorem}
The condition \eqref{eq:limit-at-beginning} states that the inhomogeneities in the mass distribution produced by the perturbation converge to an initial profile \(\tilde \nu\). For the concrete example of the Arnold cat map, this condition is verified in Lemma~\ref{lem:formula-lin-resp}. Let \(\widetilde \nu\) satisfy \eqref{eq:limit-at-beginning}. The following proposition asserts that \(\widetilde \nu\) has a special form which guarantees that \(\tilde \nu\) is a measure with finite total variation (not positive since \(\tilde \nu (1) =0\)).
\begin{proposition}\label{prop:form-measure}
There exists a measure \(\sing\) supported on \(\cS^{-}_0\) of finite total variation and a function \(\mathcal R \in Lip(\cM)\) such that, for any \(\vf \in \cC^0(\cM)\),
\[
  \widetilde \nu(\vf) = \int_{\cM}\mathcal R \vf dm + \sing(\vf).
\]
\end{proposition}
The proof of Proposition \ref{prop:form-measure} is postponed to Section \ref{sec:linear-response}. Theorem \ref{thm:linear-resp} is a consequence of a more general result. Let \(\{t_n\}_{n \in \bN} \subset (0,\ve]\), \(\lim_{n \to \infty}t_n = 0\) and consider
\begin{equation}\label{eq:A-Bsets}
  \begin{split}
  &  \mathscr A = \left\{\frac{\nu_{t_n}}{t_n}\right\}_{n \in \bN}, \quad \mathscr B = \left\{\frac{\mu_{t_n} - \mu_0}{t_n}\right\}_{n \in \bN}.
  \end{split}
\end{equation}
By Lemma \ref{lem:total-variation} (see \eqref{eq:deriv-def} for the definition of  \(\Deriv_t\)),
\begin{equation}\label{eq:compactness-initial}
\begin{split}
\frac{|\nu_t(\vf)|}{t} =&  \left|\int_{\cM} \vf \Deriv_t dm\right| \le  \|\vf\|_{\cC^0}\|\Deriv_t\|_{L^1} \le C\|\vf\|_{\cC^0}.
\end{split}
\end{equation}
Therefore, by \eqref{eq:compactness-initial} and Theorem \ref{thm:Lip} respectively,
\[
\begin{split}
  &\mathscr A \subseteq \left\{\rho \in {\cC^0(\cM)}^*: \sup_{\vf \in \cC^0(\cM), \text{ }\|\vf\|_{\cC^0} = 1}|\rho (\vf)| \le C\right\},\\
  &\mathscr B \subseteq \left\{\rho \in {\cC^1(\cM)}^*: \sup_{\vf \in \cC^1(\cM), \text{ }\|\vf\|_{\cC^1} = 1}|\rho (\vf)| \le C\right\}.
\end{split}
\]
Hence, by Banach-Alaoglu theorem both \(\mathscr A\) and \(\mathscr B\) admit accumulation points in the correspondent weak-\(\star\) topologies. Denote by \(\partial \mathscr A \subseteq \mathscr A\) and \(\partial \mathscr B \subseteq \mathscr B\) the set of such accumulation points.
\begin{proposition}
  \label{prop:linear-resp-accumulation}
  There exists a surjective map \(I: \partial \mathscr A \to \partial \mathscr B\) such that, for any \(\widetilde \nu \in \partial \mathscr A\), \(\vf \in \cC^1(\cM)\), 
 \begin{equation}
  \label{eq:linear-resp-formulas}
  I(\widetilde \nu)(\vf) = \sum_{k \in \bN_0} \widetilde \nu \left(\vf \circ \cF_0^k\right).
\end{equation}
\end{proposition}
The proof of Proposition \ref{prop:linear-resp-accumulation} is presented in section \ref{sec:linear-response}. We end this section by showing how Theorem \ref{thm:linear-resp} is a consequence of Proposition \ref{prop:linear-resp-accumulation} and with some remarks.
\begin{proof}[\textbf{Proof of Theorem \ref{thm:linear-resp}}]
  Let \(\vf \in \cC^1(\cM)\) and consider any sequence \(\{t_n\} \subset (0, \ve]\), \(\lim_{n \to \infty}t_n = 0\) together with \(\mathscr A\) and \(\mathscr B\).
  By \eqref{eq:limit-at-beginning}, \(\lim_{n \to \infty} \nu_{t_n} / t_n = \widetilde \nu\) weakly. Hence, \(\# \partial \mathscr A = 1\) and by Proposition \ref{prop:linear-resp-accumulation} also \(\# \partial \mathscr B = 1\). Therefore,
\[
    \begin{split}
    &\lim_{n \to \infty}\frac{\mu_{t_{n}}(\vf) - \mu_0(\vf)}{t_{n}} = I(\widetilde \nu) (\vf) = \sum_{k \in \bN_0}\widetilde \nu\left(\vf \circ \cF_0^k\right).
    \end{split} 
\]
The claim follows from the fact that the sequence \(\{t_n\}\) is arbitrary.
\end{proof}

We conclude this section with a discussion of the assumptions. Assumption \ref{ass:sing-near} describes perturbations of piece-wise smooth maps in which the continuity domains and the branches of the map are allowed to vary with the perturbation. In Assumption \ref{ass:smooth-0} we restrict to the case in which the unperturbed dynamics is conservative, and, as discussed in the introduction, this is an important limitation that we are not able to relax in this work. Assumption \ref{ass:srb} requires that the perturbed maps have a SRB measure (or, more appropriately, physical measure) and it is a natural assumption in the context of linear response. Assumption \ref{ass:foliation} is the most technical one. It requires that \(\cM \setminus (\cS_t^{-} \cup \cS_0^{-})\) can be foliated with \(\cC^1\) curves which support probability measures whose correlations decay fast enough. This assumption is closely related with the notion of standard pair, that is, loosely speaking, a measure supported on a unstable curve equipped with a smooth density. In concrete situations from hyperbolic dynamics, Assumption \ref{ass:foliation} could be `unpacked' in two parts: standard pairs mix rapidly (this typically happens in systems with some hyperbolicity and a contracting direction, see e.g., \cite{Dolgopyat04, CM06, DSL16, C99, CD09}) and the geometry of the discontinuity set \(\cS_t^{-} \cup \cS_0^{-}\) allows \(\cM \setminus (\cS_t^{-} \cup \cS_0^{-})\) to be foliated by long unstable curves, except possibly for a set whose measure is \(\cO(t^2)\) which is accounted for in \ref{ass:foliation}–\ref{ass:good-initial-cond2}. In the example of the perturbed Arnold cat map of Section \ref{sec:examples}, this assumption holds because the singularities \(\cS_t^{-} \cup \cS_0^{-}\) are aligned with the unstable cone and so they do not `cut' unstable curves efficiently, see \ref{Fig:2}. We look for long unstable curves because in hyperbolic dynamics the inverse length of an unstable curve determines the time needed for mixing as the curve has to grow in size in order to equidistribute. The fact that these curves are long uniformly in \(t\) guarantees the uniform decay of correlations needed for linear response. We believe/hope that the range of applicability of these assumptions is vast but, as of now, the verification of \ref{ass:foliation} requires a considerable amount of work even for the simple example of the perturbed cat map.

\begin{remark}\label{rm:discuss-3}
In Lemma \ref{lem:decay-tot} we show that under the previous assumptions it holds:
\textit{
\begin{enumerate}
\item[{\crtcrossreflabel{(H4')}[ass:decay-n]}]  
For any \(t \in (0, \ve]\), \(n \in \bN_0\), \(\frac{\left|\int_{\cM} \left(\vf \circ \cF_t^{n+1} - \vf \circ \cF_t^n\right) h_0 dm\right| }{t}\le C\Theta_n \|\vf\|_{\cC^1}\).\\
\end{enumerate}
}
While assumption \ref{ass:decay-n} yields Theorem \ref{thm:Lip} straightforwardly, it is also possible to assume \ref{ass:decay-n} instead of \ref{ass:foliation} in Theorem \ref{thm:linear-resp}.
\end{remark}

\section{Lipschitz continuity}\label{sec:Lip}
In this section we prove Theorem \ref{thm:Lip}. We start by defining the transfer operator point-wise, for any \(g: \cM \to \bR\), and \(x \in \cM \setminus \cS^{-}_{t}\), as
\begin{equation}
\label{eq:transf-op}
   \cL_t g(x) = \sum_{y \in \cF_t^{-1}(x)}\frac{g(y)}{\left| \det D\cF_t(y) \right|}. 
\end{equation}
As it is well known, for every \(h \in L^{\infty}(\cM)\) and \(g \in L^{1}(\cM)\),
\begin{equation}
  \label{eq:change-of-var}
  \int_{\cM} (h \circ \cF_t) g dm = \int_{\cM} h  (\cL_t g) dm.
\end{equation}
By the following Lemma, equation \eqref{eq:transf-op} has a uniform version on sufficiently small open sets. For each \(x \in \cM \setminus \cS^{-}_t\) and \(t \in [0, \ve]\), set
\begin{equation}\label{eq:tagged-preimages}
    n_t(x) = \{j \in \cJ: x \in \cM_{j,t}^{-}\}.
\end{equation}

\begin{lemma}
  \label{lem:tr-op}
  For any \(t \in [0,\ve]\), \(x \in \cM \setminus \cS_t^{-}\), there exists a neighborhood \(U \ni x\), such that, on \(U\), for any \(g: \cM \to \bR\),
  \[
  \cL_t g = \sum_{j \in n_t(x)} \frac{g }{\left|\det D F_{j,t}\right|}\circ F_{j,t}^{-1}.
  \]
  \end{lemma}
  \begin{proof}
  The set \(\cM \setminus \cS^{-}_t\) is open. Hence, there exists a neighborhood \(U\) of \(x\) such that \(U \cap \cS_t^{-} = \emptyset\). In particular, for all \(j \in \cJ\), either \(U \subseteq \cM_{j,t}^{-}\) or \(U \cap \cM_{j,t}^{-} = \emptyset\). Since \(x \in U\), one has that \(U \subseteq \cM_{j,t}^{-}\) if and only if \(j \in n_t(x)\). Hence, for any \(z \in U\), \(\cF_t^{-1}(z) = \left\{F_{j,t}^{-1}(z) : j \in n_t(x)\right\}\) and the claim follows from the formula \eqref{eq:transf-op} for the transfer operator.
  \end{proof}
  
The next Lemma is nothing more than a simple observation that allows us to rephrase the problem of linear response in terms of the convergence of a series.

\begin{lemma}
  \label{lem:telescopic}
  For any \(t \in [0, \ve]\), \(\vf \in \cC^0(\cM)\),
  \[
    \mu_t(\vf) - \mu_0(\vf) = \sum_{k \in \bN_0}\int_{\cM}\vf \circ \cF_t^k \left(\cL_t h_0 - h_0 \right)dm.
  \]
\end{lemma}
\begin{proof}
Applying \ref{ass:srb} and \ref{ass:smooth-0}, and by means of a telescopic sum,
  \[
    \begin{split}
    \mu_t(\vf) &= \lim_{n \to \infty} \mu_{0}(\vf \circ \cF_t^n) = \lim_{n \to \infty} \int_{\cM} (\vf \circ \cF_t^n)h_0 dm \\
    &= \lim_{n \to \infty} \sum_{k=0}^n \int_{\cM}\left(\vf \circ \cF_t^{k+1}- \vf \circ \cF_t^k\right) h_0 dm  + \int_{\cM} \vf h_0 dm.
    \end{split} 
  \]
Finally, the claim is proved by the following identity derived by \eqref{eq:change-of-var},
\[
  \int_{\cM}\left( \vf \circ \cF_t^{k+1}- \vf \circ \cF_t^k\right) h_0 dm = \int_{\cM}\vf \circ \cF_t^k \left(\cL_t h_0 - h_0 \right)dm.
\]
\end{proof}
We introduce sets of good and bad points. Recalling \eqref{eq:tagged-preimages}, define
\begin{equation}\label{eq:goodbad}
  \begin{split}
    &\cB_t = \biggl\{ x \in \cM \setminus \left( \cS^{-}_t \cup \cS^{-}_0\right) :  n_t(x) \neq n_0(x)\biggr\} \\
&\hspace{0.35cm}= \bigcup_{j \in \cJ}\cM_{j,0}^{-} \triangle \cM_{j,t}^{-} \setminus \left( \cS^{-}_t \cup \cS^{-}_0\right),\\
& \cG_t = \cM \setminus \left(\cB_t \cup \cS^{-}_t \cup \cS^{-}_0\right).
  \end{split}
\end{equation}
In simple terms, \(\cB_t\) are the points having a different number of preimages with respect to \(\cF_t\) and \(\cF_0\) and/or the preimages come from a different \(\cM_j^+\). Since \(n_t\) and \(n_0\) are locally constant functions \(\cM \setminus (\cS_t^{-}\cup \cS_0^{-}) \to \cJ\), one has that
\begin{equation}\label{eq:boundariesB-G}
\partial \cB_t \cup \partial \cG_t \subseteq \cS_{t}^{-}\cup \cS_{0}^{-}.
\end{equation}
By Lemma \ref{lem:telescopic}, we understand that the function \(t^{-1}(\cL_t h_0 -h_0)\) is particularly important for our analysis. We split it into a good and a bad part:
\begin{equation}\label{eq:deriv-def}
\Deriv_t = \Deriv_t^\cG + \Deriv_t^\cB; \quad   \Deriv_t^\cG = \frac{\cL_{t} h_0 - h_0}{t}\mathbbm 1_{\cG_t}, \text{ }\text{ } \Deriv_t^\cB = \frac{\cL_{t} h_0 - h_0}{t}\mathbbm 1_{\cB_t}.
\end{equation}
In the next Lemmata, we argue that the series \(\sum_{k \in \bN_0} \left|\int_{\cM} \vf \circ \cF_t^k \Deriv_t dm \right|\) converges uniformly in \(t\). To this aim, it is crucial to notice that \(\Deriv_t\) has zero average,
\begin{equation}\label{eq:zero-avg-D}
  \int_{\cM} \Deriv_t dm = 0.
\end{equation} 
We first consider the good part \(\Deriv_t^{\cG}\) and show that it has bounded \(\cC^1\) norm, uniformly in \(t\). We need the following preliminary Lemma. Recall that \(D\) was introduced in \ref{ass:unif-extension} and \(Q\) in \ref{ass:sing-near}. We will denote finite dimensional norms by \(|\cdot|\).
\begin{lemma}
\label{lem:bounding-inverse}
For all \(t \in [0,\ve]\) small enough, \(j \in \cJ\) and \(x \in \cM_{j,t}^{-} \cap \cM_{j,0}^{-}\),
\[
\begin{split}
  &\left|F_{j,t}^{-1}(x) - F_{j,0}^{-1}(x) \right| \le QD t,\\
  & \left| D(F_{j,t}^{-1})(x) - D(F_{j,0}^{-1})(x)\right| \le D^2(D+2)Qt.
\end{split}
\]
\end{lemma}
\begin{proof}
By \ref{ass:unif-extension} and \ref{ass:sing-near}, and recalling that \(F_{j}\) are defined on \(U_{\cM} \supset \cM\), 
\[
\begin{split}
0 &= \left|F_{j,0}(F_{j,0}^{-1}(x)) - F_{j,t}(F_{j,t}^{-1}(x)) \right| \\
&\ge  \left|F_{j,0}(F_{j,0}^{-1}(x)) - F_{j,0}(F_{j,t}^{-1}(x))\right| - \left|F_{j,0}(F_{j,t}^{-1}(x))- F_{j,t}(F_{j,t}^{-1}(x)) \right| \\
&\ge  D^{-1} \left|F_{j,0}^{-1}(x)-F_{j,t}^{-1}(x) \right| - Qt,
\end{split}
\]
proving the first part of the Lemma. Using what we just proved,
\begin{equation}\label{eq:bound-inverse-matrix}
\begin{split}
\bigl| D(F_{j,t}^{-1})(x) - D(F_{j,0}^{-1})(x)\bigr| &\le \left| (DF_{j,t})^{-1} \circ F_{j,t}^{-1}(x) - (DF_{j,t})^{-1} \circ F_{j,0}^{-1}(x) \right|\\
&+  \left| (DF_{j,t})^{-1} \circ F_{j,0}^{-1}(x) - (DF_{j,0})^{-1} \circ F_{j,0}^{-1}(x)\right|\\
&\le \sup_{\cM}\left| D((DF_{j,t})^{-1})\right| QDt \\
&+ \sup_{\cM}\left|(DF_{j,t})^{-1} -(DF_{j,0})^{-1}\right|.
\end{split}
\end{equation}
To move forward, \( D((DF_{j,t})^{-1}) = [D^2(F_{j,t}^{-1})\circ F_{j,t}] DF_{j,t} \) and, for all \(t\) small enough,
\[
\left|(DF_{j,0})^{-1}- (DF_{j,t})^{-1} \right| \le \left|(DF_{j,0})^{-1}\right|^2 \frac{\left|DF_{j,0} - DF_{j,t}\right|}{1- \left| (DF_{j,0})^{-1}\right|\left|DF_{j,0} - DF_{j,t}\right|}.
\]
By the above relations and \ref{ass:unif-extension}, \ref{ass:sing-near}, we have that the quantity in the right side of \eqref{eq:bound-inverse-matrix} is not bigger than
\[
D^2 QD t+ D^2\frac{ Qt}{1-DQt} \le D^2(D+2)Qt,
\]
for \(t\) small enough. This concludes the proof of the Lemma.
\end{proof}
We will assume that \(\ve\) is small enough such that the above Lemma is satisfied for any \(t \in [0,\ve]\) and we will sometimes denote with the apostrophe derivatives (or differentials in general) of functions whose codomains are the real numbers.
\begin{lemma}\label{lem:single-perturbed-preimage-invariant-measure}
 For any \(t \in (0, \ve]\), 
  \[
    \begin{split}
  &\sup_{\cG_t}\max\left\{\bigl| \Deriv_t\bigr|, \bigl|\Deriv_t'\bigr|  \right\}  \le C.
    \end{split}
  \]
\end{lemma}
\begin{proof}
Let \(x \in \cG_t \subseteq \cM \setminus \left(\cS_t^{-} \cup \cS_0^{-}\right)\). By the fact that \(h_0\) is invariant and by Lemma \ref{lem:tr-op}, there is a neighbourhood \(U \ni x\) such that, for all points in \(U\),
\begin{equation}\label{eq:comparison-invariant}
\begin{split}
\cL_t h_0 - h_0 &= \cL_t h_0 - \cL_{0}h_0 \\
&= \sum_{j \in n_t(x)} \frac{h_0}{\left|\det DF_{j,t}\right|} \circ F_{j,t}^{-1} - \sum_{j \in n_0(x)} \frac{h_0}{\left|\det DF_{j,0}\right|} \circ F_{j,0}^{-1}\\
&=\sum_{j \in n_0(x)} \frac{h_0}{\left|\det DF_{j,t}\right|} \circ F_{j,t}^{-1} - \sum_{j \in n_0(x)} \frac{h_0}{\left|\det DF_{j,0}\right|} \circ F_{j,0}^{-1},
\end{split}
\end{equation}
where in the last equality we have used that  \(n_t(x) = n_0(x)\) for \(x \in \cG_t\). By adding and subtracting,
\begin{equation}\label{eq:pre-divisionbyt}
\begin{split}
|\cL_t h_0 - h_0| &\le \sum_{j \in n_0(x)} \biggl|\frac{h_0}{|\det DF_{j,t}|} \circ F_{j,t}^{-1} - \frac{h_0}{|\det DF_{j,t}|}\circ F_{j,0}^{-1} \biggr| \\
&+ \sum_{j \in n_0(x)}\biggl|\frac{h_0}{|\det DF_{j,t}|} - \frac{h_0}{|\det DF_{j,0}|} \biggr| \circ F_{j,0}^{-1}\\
& \le \#\cJ \sup_{j \in \cJ}\sup_{\cM} \left|\left( \frac{h_0}{\det DF_{j,t}}\right)' \right|\sup_{\cM_{j,0}^{-}\cap \cM_{j,t}^{-}}\left|F_{j,t}^{-1}- F_{j,0}^{-1}\right|\\
& + \#\cJ\sup_{j \in \cJ}\sup_{\cM}\left|\frac{h_0}{|\det DF_{j,t}|} - \frac{h_0}{|\det DF_{j,0}|}\right|\\
&\le \#\cJ\|h_0\|_{\cC^1} \sup_{j \in \cJ}\sup_{\cM} \biggl[\left|\frac{(\det DF_{j,t})'}{(\det DF_{j,t})^2}\right| + \left|\frac{1}{\det DF_{j,t}}\right|\biggr] QD t\\
& + \#\cJ \|h_0\|_{\cC^0}\sup_{j \in \cJ}\sup_{\cM}\left|\frac{|\det DF_{j,0}| - |\det DF_{j,t}|}{\det DF_{j,t}\det DF_{j,0}}\right|.
\end{split}
\end{equation}
where in the last inequality we have used Lemma  \ref{lem:bounding-inverse}. By \ref{ass:unif-extension} the determinant of \(F_{j}\) is bounded away from zero. Hence, using \ref{ass:sing-near} for the last summand, we have that the previous quantity is bounded by \(Ct\). Taking the derivative of \eqref{eq:comparison-invariant}, we get
\[
\begin{split}
\sum_{j \in n_0(x)}  \left(\frac{h_0}{|\det DF_{j,t}|}\right)' \circ F_{j,t}^{-1}&  D(F^{-1}_{j,t}) - \sum_{j \in n_0(x)}  \left(\frac{h_0}{|\det DF_{j,t}|}\right)' \circ F_{j,0}^{-1}  D(F^{-1}_{j,0}).
\end{split}
\]
By adding and subtracting,
\begin{equation}\label{eq:bound-deriv}
    \begin{split}
&|(\cL_t h_0  - h_0 )'| \le \sum_{j \in n_0(x)} \biggl| \left(\frac{h_0}{|\det DF_{j,t}|}\right)' \circ F_{j,t}^{-1} \bigl[  D(F^{-1}_{j,t}) - D(F^{-1}_{j,0}) \bigr]\biggr|\\
& + \sum_{j \in n_0(x)} \biggl|\biggl[  \left(\frac{h_0}{|\det DF_{j,t}|}\right)' \circ F_{j,t}^{-1}  -  \left(\frac{h_0}{|\det DF_{j,t}|}\right)' \circ F_{j,0}^{-1}  \biggr] D(F_{j,0}^{-1})\biggr|.
\end{split}
\end{equation}
The first term is less than \(Ct\) because of the second part of Lemma \ref{lem:bounding-inverse} and the uniformity assumptions. As for the second term, we notice that it has the same form as the one in \eqref{eq:comparison-invariant} and we have estimates analogous to \eqref{eq:pre-divisionbyt}. Since both bounds \eqref{eq:pre-divisionbyt} and \eqref{eq:bound-deriv} do not depend on \(t\) and \(U\), we have proved the statement of the Lemma.
\end{proof}
In the following lemmata we study the regularity of \(\Deriv_t^{\cB}\).
\begin{lemma}
\label{lem:regularity}
\(\cL_t h_0\) is \(\cC^1\) on \(\cM \setminus \cS_t^{-}\), uniformly in  \(x \in \cM \setminus \cS_t^{-}\) and \(t \in [0,\ve]\).
\end{lemma}
\begin{proof}
Fix any \(t \in [0,\ve]\) and let \(x \in \cM \setminus \cS_t^{-}\). By Lemma \ref{lem:tr-op}, there exists a neighborhood \(U \ni x\), such that, for all points in \(U\),
\begin{equation}
  \label{eq:trop-U}
\cL_t h_0 = \sum_{j \in n_t(x)} \frac{h_0}{\left|\det DF_{j,t} \right|}\circ F_{j,t}^{-1}.
\end{equation}
 By \ref{ass:unif-extension}, \(F_{j,t}^{-1}\) is \(\cC^1\) uniformly in \(U\). Moreover, \(\left|\det DF_{j,t}\right|\) is \(\cC^1\) and bounded away from zero, also uniformly. By \ref{ass:smooth-0}, \(h_0 \in \cC^1 (\cM)\). Hence, the sum in \eqref{eq:trop-U} consists of finitely many terms that are \(\cC^1\) uniformly in \(U\). Since the estimates on the \(\cC^1\) norm do not depend on \(U\) and \(t\) this proves the general statement as well.
\end{proof}

\begin{lemma}
  \label{lem:regularity-norms}
For each \(t \in (0, \ve]\), \(\Deriv_t^{\cB}\) is \(\cC^1\) on (every connected subset of) \(\cB_t\). Moreover,
\[
  \sup_{\cM \setminus \cS_t^{-}} \max\left\{|\Deriv_t|, |\Deriv_t'|\right\} \le \frac C t.
\]
\end{lemma}
\begin{proof}
The first part of the statement follows by Lemma \ref{lem:regularity} and the fact that \(\cB_t \cap \cS_t^{-} = \emptyset\) by definition. As for the second part, by Lemma \ref{lem:regularity} and \ref{ass:smooth-0},
\[
\begin{split}
 & |\Deriv_t| = \left|\frac{\cL_t h_0 - h_0}{t}\right | \le \frac{|\cL_t h_0| + |h_0|}{t} \le \frac{C}{t},\\
&  |\Deriv_t '| = \frac{|(\cL_t h_0 )' - h_0 '|}{t} \le \frac{|(\cL_t h_0) '| + |h_0 '|}{t} \le \frac{C}{t}.
\end{split}
\]
\end{proof}
Lemma \ref{lem:regularity-norms} shows that the function \(\Deriv^{\cB}_t\) may have a \(1/t\) blow-up in some parts of the phase space as \(t\) tends to zero. However, as we will see in Lemma \ref{lem:meas-bad-set}, the place where this blow-up occurs is localized in a neighborhood of the image of the singularity lines and has a measure proportional to \(t\). We start with the following fact about the discontinuity lines that will come in handy later on.
\begin{lemma}\label{lem:neigh}
For any \(t \in [0,\ve]\), \(m \bigl( [\cS_0^{\pm}]_{t}\bigr) \le Ct\) and for each \(k \in \bN\) there exists \(C_k \in \bR^+\) such that \(m \bigl( [\cS_0^{+,k}]_{t} \bigr)\le C_k t\).
\end{lemma}
\begin{proof}
By assumption, both \(\cS_0^{+}\) and \(\cS_0^{-}\) are a finite union of \(\cC^1\) curves of finite length and the statement follows by covering the neighborhood of each curve with a number of balls proportional to the length of the curve each one of radius \(2t\). For the second part of the statement, recalling definition \eqref{eq:future-disc} and \ref{ass:unif-extension}, one has that \(\cS_0^{+,k}\) is a finite union of \(\cC^1\) curves of finite length as well, concluding the proof.
\end{proof}
Recall that \(d(A,B) = \inf \{s \ge 0: A \subseteq [B]_s \text{ and }B \subseteq [A]_s\}\).
\begin{lemma}\label{lem:aux-aux-sing}
For all \(t \in [0, \ve]\) and all \(j \in \cJ\), \(d \left( \partial \cM_{j,0}^{+}, \partial \cM_{j,t}^{+}\right) \le Qt\).
\end{lemma}
\begin{proof}
We first prove that \(\partial \cM_{j,t}^{+}\subseteq [\partial \cM_{j,0}^{+} ]_{Qt}\). If not, there exists \(x \in \partial \cM_{j,t}^{+}\) such that \(d(x, \partial \cM_{j,0}^{+}) > Qt\). There are two cases, either \(x \in \cM_{j,0}^{+}\) or \(x \notin \cM_{j,0}^{+}\). In the case \(x \in \cM_{j,0}^{+}\cap \partial \cM_{j,t}^{+}\), because \(\cM_{j,0}^+\) is open and by definition of boundary, there exists a ball \(B \ni x\) of radius \(\delta\) arbitrarily small, such that \(B \subseteq \cM_{j,0}^{+}\) and \(B \cap (\cM \setminus \cM_{j,t}^{+}) \neq \emptyset\). Hence, there exists \(y \in B\) such that \(y \in (\cM \setminus \cM_{j,t}^{+}) \cap \cM_{j,0}^{+} \subseteq \cM_{j,0}^{+} \triangle \cM_{j,t}^{+}\). On the other hand, 
\[
d(y, \partial \cM_{j,0}^{+}) \ge d(x, \partial \cM_{j,0}^{+} ) -d(y,x) > Qt - \delta,
\]
and since \(\delta\) is arbitrary we have a contradiction with the second equation in \ref{ass:sing-near}. In the case \(x \in (\cM \setminus\cM_{j,0}^{+})\cap \partial \cM_{j,t}^{+}\), either \(x \in \partial \cM_{j,0}^{+}\) (and we have immediately a contradiction) or, for each \(\delta \in \bR^+\), there exists a ball \(B \ni x\) of radius \(\delta\), such that \(B \subseteq \cM \setminus \cM_{j,0}^{+}\) and \(B \cap  \cM_{j,t}^{+} \neq \emptyset\). Analogously, this implies that there exists \(y \in (\cM_{j,0}^{+} \triangle \cM_{j,t}^{+} )\cap B\) such that \(d(y, \partial \cM_{j,0}^{+}) > Qt - \delta\), yielding a contradiction with \ref{ass:sing-near} and showing that \(\partial \cM_{j,t}^{+}\subseteq [\partial \cM_{j,0}^{+} ]_{Qt}\). Since \ref{ass:sing-near} is symmetric with respect to \(\cM_{j,t}^{+}\) and \(\cM_{j,0}^{+}\), by the same argument as above it follows that  \(\partial \cM_{j,0}^{+}\subseteq [\partial \cM_{j,t}^{+} ]_{Qt}\) and hence the statement.
\end{proof}
\begin{lemma}\label{lem:aux-sing}
For all \(t \in [0,\ve]\): If \(y \in \cM_{j,0}^{+}\) and \(F_{j,t}(y) \notin \cM_{j,0}^{-}\), then \(d \left(y, \partial \cM_{j,0}^{+}\right) \le 2 D Q t\). Conversely, if \(y \in \cM_{j,t}^{+}\) and \(F_{j,0}(y) \notin \cM_{j,t}^{-}\), then \(d \left(y, \partial \cM_{j,0}^{+}\right) \le  Q(2D+1) t\).
\end{lemma}
\begin{proof}
We consider the first part of the statement and we argue by contradiction. If the statement is false, by \ref{ass:unif-extension} and  \ref{ass:sing-near}, using that \(\partial \cM_{j,0}^{-} = F_{j,0}(\partial \cM_{j,0}^+)\),
\begin{equation}\label{eq:jt-distance}
\begin{split}
    d(F_{j,t}(y),& \partial \cM_{j,0}^{-}) \ge d(F_{j,0}(y), F_{j,0}(\partial \cM_{j,0}^{+})) - d(F_{j,t}(y), F_{j,0}(y)) \\
&\ge D^{-1} d (y, \partial \cM_{j,0}^{+}) - d(F_{j,t}(y), F_{j,0}(y)) > D^{-1}2DQt - Qt > Qt.  
\end{split}
\end{equation}
Since \(F_{j,t}(y) \not \in \cM_{j,0}^{-}\) and \(F_{j,0}(y) \in \cM_{j,0}^{-}\), by \eqref{eq:jt-distance}, 
\[
d(F_{j,t}(y), F_{j,0}(y)) \ge d(F_{j,t}(y), \partial \cM_{j,0}^{-})  > Qt,
\] 
contradicting the first part of \ref{ass:sing-near}. Continuing with the second statement, if this is false, by Lemma \ref{lem:aux-aux-sing},
\[
d(y, \partial \cM_{j,t}^{+}) \ge d(y, \partial \cM_{j,0}^{+}) - d(\partial \cM_{j,0}^{+}, \partial \cM_{j,t}^{+}) > Q(2D +1)t - Qt = 2QDt,
\]
and we apply the same argument as above, exchanging \(t\) and \(0\) in the subscripts.
\end{proof}
As anticipated, the set \(\cB_t\) is near the image of the singularity lines.
\begin{lemma}\label{lem:meas-bad-set}
There exist \(H, C \in \bR^+\) such that, for any \(t \in [0, \ve]\), \(\cB_t \cup \cS_t^{-}\subseteq [\cS_0^{-}]_{Ht}\) and \(
m \left( \cB_t \right) \le C t\).
\end{lemma}
\begin{proof}
Let \(x \in \cM_{j,0}^{-} \triangle \cM_{j,t}^{-}\), \(j \in \cJ\). Then either (a) \(x \in \cM_{j,0}^{-} \setminus \cM_{j,t}^{-}\) or (b) \(x \in \cM_{j,t}^{-} \setminus \cM_{j,0}^{-}\). Let us consider the first case. Then \(x = F_{j,0}(y)\) and either (a1) \(y \in \cM_{j,0}^{+} \setminus \cM_{j,t}^{+}\) or (a2) \(y \in \cM_{j,0}^+ \cap \cM_{j,t}^{+} \). In case (a1), by \ref{ass:unif-extension} and \ref{ass:sing-near},
\[
 d(x, \partial \cM_{j,0}^{-}) = d(F_{j,0}(y),  F_{j,0}(\partial \cM_{j,0}^{+})) \le D d(y, \partial \cM_{j,0}^{+}) \le DQt.
\]
In case (a2), by the second part of Lemma \ref{lem:aux-sing},
\[
 d(x, \partial \cM_{j,0}^{-}) = d(F_{j,0}(y),  F_{j,0}(\partial \cM_{j,0}^{+}))  \le Dd(y, \partial \cM_{j,0}^{+}) \le DQ(2D+1)t.
\]
It follows that \( \cM_{j,0}^{-} \setminus \cM_{j,t}^{-} \subseteq [\partial \cM_{j,0}^{-}]_{DQ(2D+1)t}\). In case (b), one has that \(x = F_{j,t}(y)\) and either (b1) \(y \in \cM_{j,t}^{+} \setminus \cM_{j,0}^{+}\) or (b2) \(y \in \cM_{j,t}^{+} \cap \cM_{j,0}^{+} \). In case (b1),
\[
\begin{split}
d\left(x, \partial \cM_{j,0}^{-} \right) &= d\left(F_{j,t}(y), F_{j,0}\left(\partial \cM_{j,0}^{+}\right) \right) \le d\left(F_{j,0}(y), F_{j,0}\left(\partial \cM_{j,0}^{+}\right) \right) \\
&+ d\left(F_{j,t}(y), F_{j,0}(y) \right) \le DQt + Qt = Q(D+1)t.
\end{split}
\]
Finally, in case (b2), by the first part of Lemma \ref{lem:aux-sing}, 
\[
\begin{split}
d\left(x, \partial \cM_{j,0}^{-} \right) &= d\left(F_{j,t}(y), F_{j,0}\left(\partial \cM_{j,0}^{+}\right) \right) \le d\left(F_{j,0}(y), F_{j,0}\left(\partial \cM_{j,0}^{+}\right) \right) \\
&+ d\left(F_{j,t}(y), F_{j,0}(y) \right) \le 2D^2 Qt + Qt = Q(2D^2+1)t.
\end{split}
\]
Hence, for some \(H \in \bR^{+}\) and for all \(j \in \cJ\),
\begin{equation}\label{eq:minus-symm-diff}
     \cM_{j,0}^{-} \triangle \cM_{j,t}^{-} \subseteq [\partial \cM_{j,0}^{-}]_{Ht},
\end{equation}
so that \(\cB_{t} \subseteq \bigcup_{j \in \cJ}\cM_{j,0}^{-}\triangle \cM_{j,t}^{-} \subseteq \bigcup_{j \in \cJ} [\partial \cM_{j,0}^-]_{Ht} = [\cS_0^{-}]_{Ht}\). According to the same argument of Lemma \ref{lem:aux-aux-sing}, equation \eqref{eq:minus-symm-diff} implies that \(\partial \cM_{j,t}^{-} \subseteq [\partial \cM_{j,0}^{-}]_{Ht}\), concluding the proof of the first part. The second part of the statement is a consequence of the first part and Lemma \ref{lem:neigh}.
\end{proof}

We now take advantage of assumption \ref{ass:foliation} about summable decay of correlations on curves and the fact that \(\cB_t\) has small measure to prove decay of correlations uniform in the perturbation parameter.

\begin{lemma}
  \label{lem:decay-tot}
  For all \(t \in (0,\ve]\), \(\vf \in \cC^1(\cM)\), \(n \in \bN_{0}\),
  \[
  \left|\int_{\cM}\bigl(\vf\circ \cF_t^n \bigr)\Deriv_t  dm  \right| \le C\Theta_n\|\vf\|_{\cC^1}.
  \]
\end{lemma}
\begin{proof}
Recalling \eqref{eq:zero-avg-D}, we have, 
\begin{equation}\label{eq:goodandbad-decay-cor}
\begin{split}
     \left|\int_{\cM}\left(\vf \circ \cF_t^n\right) \Deriv_t dm\right| &= \left|\int_{\cM}\left(\vf \circ \cF_t^n\right) \Deriv_t dm - \mu_t(\vf) \int_{\cM}\Deriv_t dm\right| \\
&\le \left| \int_{\cG_t} \left(\vf\circ \cF_t^n\right) \Deriv_t dm - \mu_t(\vf)  \int_{\cG_t}\Deriv_t dm\right|\\
&+  \left|\int_{\cB_t} \left(\vf\circ \cF_t^n\right) \Deriv_t dm- \mu_t(\vf) \int_{\cB_t}\Deriv_t dm\right|.
\end{split}
\end{equation}
Let \(\mathscr P_t\), \(\mathscr P_t^{1}\), \(\mathscr P_t^2\) be given by \ref{ass:foliation}. By \eqref{eq:boundariesB-G} and since \(\mathscr P_t\) is a partition of \(\cM \setminus (\cS_0^{-} \cup \cS_t^{-})\), for any curve \(W \in \mathscr P_t\), either \(W\subseteq \cB_t\) or \(W \subseteq \cG_t\). Hence, by multiplying and dividing by \(t\) only the terms integrated over \(\cB_t\), we have that \eqref{eq:goodandbad-decay-cor} is not bigger than, (recall \eqref{eq:cov-def-general} for the definition of correlations and \eqref{eq:disintegration} for the definition of the conditional density \(p_W\)),
\[
\begin{split}
   & \int_{\{W \subseteq \cG_t\}}\!\!\!\!\! \!\!\!\! \cov(\vf, n, t; p_W  \Deriv_t^{\cG})dm_{\mathscr P_t}(W)  +t^{-1} \int_{\{W \subseteq \cB_t\}}\!\!\!\!\!\!\!\!\!\!\!\!\!\!\!\!\! \cov(\vf, n, t;  p_W t \Deriv_t^{\cB}) dm_{\mathscr P_t}(W)\\
  &= \int_{\mathscr P_t^{1} \cap \{W \subseteq \cG_t\}}\!\!\!\!\!\!\!\!\!\!\!\!\!\!\!\!\!\!\!\! \!\!\!\! \cov(\vf, n, t;  p_W \Deriv_t^{\cG})dm_{\mathscr P_t}(W)  +t^{-1} \int_{\mathscr P_t^{1} \cap \{W \subseteq \cB_t\}}\!\!\!\!\!\!\!\!\!\!\!\!\!\!\!\!\!\!\!\!\!\!\!\! \cov(\vf, n, t; p_W t \Deriv_t^{\cB}) dm_{\mathscr P_t}(W)   \\
 &+\int_{\mathscr P_t^{2} \cap \{W \subseteq \cG_t\}}\!\!\!\!\!\!\!\!\!\!\!\!\!\!\!\! \!\!\!\!\cov(\vf, n, t; p_W  \Deriv_t^{\cG})dm_{\mathscr P_t}(W)  +t^{-1} \int_{\mathscr P_t^{2} \cap \{W \subseteq \cB_t\}}\!\!\!\!\!\!\!\!\!\!\!\!\!\!\!\!\!\!\!\!\!\!\!\! \cov(\vf, n, t; p_W t \Deriv_t^{\cB}) dm_{\mathscr P_t}(W).
\end{split}
\]
Notice the additional factor \(t\) in the correlations integrated over \(\cB_t\). We now show that every term in the above equation is bounded by \(C \Theta_n\|\vf\|_{\cC^1}\). By Lemma \ref{lem:single-perturbed-preimage-invariant-measure} and Lemma \ref{lem:regularity-norms}, respectively, we have that,
\begin{equation}\label{eq:reg-of-restriction}
\begin{split}
&\bigl\|{\cD_{t}^{\cG}}_{|W}\bigr\|_{\cC^1(W)}\le \left\|{\cD_{t}^{\cG}}\right\|_{\cC^1(\cG_t)}\le C, \quad \forall W \subseteq \cG_t, \\
& \bigl\|t{\cD_{t}^{\cB}}_{|W}\bigr\|_{\cC^1(W)} \le \left\| {t \cD_{t}^{\cB}} \right\|_{\cC^1(\cB_t)} \le C, \quad \forall W \subseteq \cB_t.
\end{split}
\end{equation}
Hence, for the first term we apply \ref{ass:foliation}-\ref{ass:good-initial-cond1} with \(\rho_W \!\!=\!\! {\Deriv_t^{\cG}}_{|W}\), \mbox{and by \eqref{eq:reg-of-restriction} we have,}
\[
\begin{split}
 \int_{\mathscr P_t^{1} \cap \{W \subseteq \cG_t\}}&\!\!\!\!\!\!\!\!\!\!\!\!\!\!\!\!\!\!\!\! \!\!\!\! \cov(\vf, n, t; p_W  \Deriv_t^{\cG}) dm_{\mathscr P_t}(W) \le \Theta_n \|\vf\|_{\cC^1} \left\|\Deriv_t^{\cG}\right\|_{\cC^1} \int_{\mathscr P_t^{1}} dm_{\mathscr P_t}(W) \le C \Theta_n\|\vf\|_{\cC^1}.
\end{split}
\]
For the second term, we apply \ref{ass:foliation}-\ref{ass:good-initial-cond1} with \(\rho_W \!\!=\!\! t{\Deriv_t^{\cB}}_{|W}\), Lemma \ref{lem:meas-bad-set} and \eqref{eq:reg-of-restriction},
\[
\begin{split}
   t^{-1} \int_{\mathscr P_t^{1} \cap \{W \subseteq \cB_t\}}\!\!\!\!\!\!\!\!\!\!\!\!\!\!\!\!\!\!\!\!\!\!\!\! \cov(\vf, n, t;  p_W t \Deriv_t^{\cB}) dm_{\mathscr P_t}(W) & \le t^{-1} \Theta_n \|\vf\|_{\cC^1} \left\|t\Deriv_t^{\cB}\right\|_{\cC^1}  \int_{\{W \subseteq \cB_t\}}dm_{\mathscr P_t}(W) \\
&\le C \Theta_n \|\vf\|_{\cC^1} \frac{m \left( \cB_t\right)}{t} \le C \Theta_n \|\vf\|_{\cC^1}.
\end{split}
\]
The final result follows by applying \ref{ass:foliation}-\ref{ass:good-initial-cond2} to the last two terms integrated over \(\mathscr P_t^2\), with \(g\) being either \(\Deriv_t^{\cG}\) or \(t\Deriv_t^{\cB}\), and by \eqref{eq:reg-of-restriction}.
\end{proof}

\begin{proof}[\textbf{Proof of Theorem \ref{thm:Lip}}]
  By Lemma \ref{lem:telescopic}, for any \(t \in (0, \ve]\), dividing and multiplying by \(t\),
\[
    |\mu_t(\vf) - \mu_0(\vf)| \le \left(\sum_{k \in \bN_0} \left|\int_{\cM} (\vf \circ \cF_t^k) \Deriv_t dm\right|\right) t,
\]
and the statement of the Theorem follows by Lemma \ref{lem:decay-tot} and the fact that \(\sum_{n \in \bN_0}\Theta_n\) is convergent by \ref{ass:foliation}.
\end{proof}

\section{Linear Response} \label{sec:linear-response}

In this section we prove Propositions \ref{prop:linear-resp-accumulation} and \ref{prop:form-measure}. We start with the latter. Let \(\{s_m\} \subseteq (0, \ve]\), \(\lim_{m \to \infty}s_m = 0\), be such that \( \nu_{s_m}/s_m \to \widetilde \nu\) weakly for \(m \to \infty\). I.e., for every \(\vf \in \cC^0(\cM)\),
\begin{equation}\label{eq:weak-star-limit}
    \lim_{m \to \infty} \frac{\nu_{s_m}}{s_m}(\vf) = \lim_{m \to \infty} \int_{\cM} \Deriv_{s_m} \vf dm = \widetilde \nu (\vf).
\end{equation}
Let \(H \in \bR^+\) be given by Lemma \ref{lem:meas-bad-set}.
\begin{lemma}\label{lem:support2}
  For any \(x \in \cM \setminus \cS_0^-\),
  \[
    \limsup_{m \to \infty} \Deriv_{s_m}\mathbbm{1}_{\cM \setminus [\cS_0^-]_{Hs_m}}(x) =  \liminf_{m \to \infty} \Deriv_{s_m}\mathbbm{1}_{\cM \setminus [\cS_0^-]_{Hs_m}}(x).
  \]
\end{lemma}
\begin{proof}
If not, there exists \(x \in \cM \setminus \cS_{0}^{-}\) and two subsequences \(s_{m_{j_1}}\) and \(s_{m_{j_2}}\), \(j_1, j_2 \in \bN\) such that, for some \(a, b \in \bR\),
\begin{equation}\label{eq:limsup-inf}
\Deriv_{s_{m_{j_1}}}\mathbbm{1}_{[\cS_0^-]_{\cM \setminus Hs_{m_{j_1}}}}(x) > a >b > \Deriv_{s_{m_{j_2}}}\mathbbm{1}_{[\cS_0^-]_{\cM \setminus Hs_{m_{j_2}}}}(x), \quad \forall j_1, j_2 \in \bN.
\end{equation}
Let \(m_{*} \in \bN\) be such that \(x \in \cM \setminus [\cS_0^-]_{Hs_{m_{*}}}\) and let \(U\) be any open ball containing \(x\) such that \(U\subseteq \cM \setminus [\cS_0^{-}]_{Hs_{m_{*}}}\). By Lemma \ref{lem:meas-bad-set}, \(\cG_t = \cM \setminus (\cB_t \cup \cS_0^{-} \cup \cS_t^{-})  \supseteq \cM \setminus [\cS_0^{-}]_{Ht}\). Therefore, by Lemma \ref{lem:single-perturbed-preimage-invariant-measure}, there exists \(\overline C \in \bR^+\) such that \(\|\Deriv_{s_m| U}\|_{\cC^1} \le \overline C\) for any \(U\) as before and \(m \ge m_{*}\). We choose \(U\) such that \(\diam U \le (a-b)/(2\overline C)\). Finally, let \(\vf_0 \in \cC^0(\cM)\) supported on \(U\) such that \(\int_{U}\vf_0 dm = 1\). Then, for any \(m \ge m_{*}\),
\begin{equation}\label{eq:limit-fail}
\begin{split}
    \biggl|\int_{\cM} \Deriv_{s_{m}} \vf_0 dm -& \Deriv_{s_m}(x)\biggr| \le  \|{\Deriv_{s_m}}_{|U}\|_{\cC^1} \left(\int_{U} \vf_0 dm\right) \diam U \le \frac{a-b}{2}.
\end{split}
\end{equation}
By \eqref{eq:limsup-inf} and \eqref{eq:limit-fail}, we have that the sequence \(\int_{\cM} \Deriv_{s_{m}} \vf_0  dm\) cannot have a limit for \(m\) tending to infinity, contradicting \eqref{eq:weak-star-limit} and proving the statement.
\end{proof}

By Lemma \ref{lem:support2}, the function \(\mathcal R : \cM \setminus \cS_0^- \to \bR\), 
\begin{equation}\label{eq:R-def}
  \mathcal R=  \lim_{m \to \infty} \Deriv_{s_m}\mathbbm{1}_{\cM \setminus [\cS_0^-]_{Hs_m}},
\end{equation}
is well-defined. Moreover, the restriction of \(\cR\) to any open ball \(U \subset \cM \setminus \cS_0^{-}\) is equal to the point-wise limit of functions with bounded \(\cC^1(\overline U)\) norm. Indeed, this follows from Lemma \ref{lem:single-perturbed-preimage-invariant-measure} and the fact that, for any such \(U\), there exists an \(\overline m\) big enough such that \(U \subseteq \cM \setminus [\cS_0^-]_{Ht_m}\) for any \(m \ge \overline m\). Hence, \(\cR_{|\overline U}\) is a Lipschitz function. Since the estimate on the \(\cC^1\) norm of \({\Deriv_{s_m}}_{|U}\) does not depend on \(U\), it follows that \(\mathcal R\) can be extended to a Lipschitz function to \(\overline{\cM \setminus \cS_{0}^{-}} = \overline{\cM}\). We will call such extension \(\mathcal R\) as well. We now define, for any \(\vf \in \cC^0(\cM)\),
\begin{equation}\label{eq:sing-distribution-def}
  \sing (\vf) = \widetilde \nu(\vf) - \int_{\cM} \mathcal R \vf dm.
\end{equation}

\begin{lemma}\label{lem:sing}
The support of \(\sing\) is contained in \(\cS_0^{-}\).
\end{lemma}
\begin{proof}
Let \(\vf \in \cC^0(\cM \setminus \cS_0^-)\) with compact support \(K \subset \cM \setminus \cS_0^-\). Since \(\cS_0^{-}\) and \(K\) are compact sets, there exists an \(m_{\star} \in \bN\) such that, for \(m \ge m_{\star}\), \([\cS_0^-]_{Hs_m} \subset \cM \setminus K\). Therefore,
\begin{equation}\label{eq:limit-shrinking-support}
     \lim_{m \to \infty} \int_{[\cS_0^-]_{Hs_m}} \Deriv_{s_m} \vf dm = 0.  
\end{equation}
By \eqref{eq:sing-distribution-def}, \eqref{eq:weak-star-limit} and \eqref{eq:limit-shrinking-support},
\[
\begin{split}
    \sing(\vf) = \widetilde \nu(\vf) &- \int_{\cM}\mathcal R \vf dm = \lim_{m \to \infty} \int_{\cM} \Deriv_{s_m} \vf dm - \int_{\cM} \mathcal R \vf dm \\
&= \lim_{m \to \infty} \int_{\cM  \setminus [\cS_{0}^{-}]_{Hs_m}} \Deriv_{s_m} \vf dm - \int_{\cM}\mathcal R \vf dm
\end{split}
\]
By Lemma \ref{lem:single-perturbed-preimage-invariant-measure}, the functions \(\vf \Deriv_{t_m}\mathbbm{1}_{[\cS_0^-]_{Ht_{m}}}\) are uniformly integrable and we can apply Lebesgue dominated convergence theorem to bring the limit inside the integral. Recalling the definition \eqref{eq:R-def} of \(\cR\), the equation above shows that \(\sing(\vf) = 0\), for any \(\vf \in \cC^{0}(\cM \setminus \cS_{0}^-)\) with compact support, yielding the claim.
\end{proof}
\begin{lemma}\label{lem:total-variation}
For all \(t \in (0,\ve]\), \(\|\Deriv_t\|_{L^1} \le C\).
\end{lemma}
\begin{proof}
It follows by writing 
\[
\int_{\cM}|\Deriv_t|dm = \int_{\cG_t}|\Deriv_t| dm + \int_{\cB_t}|\Deriv_t|dm \le \sup_{\cG_t}|\Deriv_t| + m(\cB_t) \sup_{\cM}|\Deriv_t|,
\]
and by Lemmata \ref{lem:single-perturbed-preimage-invariant-measure}, \ref{lem:regularity-norms} and \ref{lem:meas-bad-set}.
\end{proof}
\begin{proof}[\textbf{Proof of Proposition \ref{prop:form-measure}}] We consider the limit \eqref{eq:limit-at-beginning} along the sequence \(\{s_m\}\). The decomposition into singular part and Lipschitz density is a consequence of equation \eqref{eq:sing-distribution-def}, Lemma \ref{lem:sing} and the discussion after equation \eqref{eq:R-def}. By Riesz representation theorem, the total variation of the measure \(\widetilde \nu\) is equal to the norm of \(\widetilde \nu\) as a linear functional on \(\cC^0\). Hence, denoting by \(|\widetilde \nu|\) the variation of \(\widetilde \nu\), by Lemma \ref{lem:total-variation}, we have for the total variation
\[
   |\widetilde \nu|(\cM) = \sup_{\|\vf\|_{\cC^0} =1} \left|\lim_{m \to \infty}\int_{\cM} \vf \Deriv_{s_m}dm \right| \le C.
\]
 Since \(\sing\) is singular w.r.t. \(m\), one has that \(\sing\) has finite total \mbox{variation as well}.
\end{proof}
We now turn to the proof of Proposition \ref{prop:linear-resp-accumulation}. The first part of the argument is controlling the effect of the perturbation on higher iterates of the map. This is obtained in Lemma \ref{lem:n-map-sing} after some preliminary results.

\begin{lemma}\label{lem:path-connected}
Let \(S \subseteq \cM \setminus [\cS_0^{+,1}]_{Qt}\) be path connected. Then, there exists \(j \in \cJ\) such that \(S \subseteq \cM_{j,0}^+ \cap \cM_{j,t}^+\).
\end{lemma}
\begin{proof}
By \ref{ass:sing-near}, \( [\cS_{0}^{+,1} ]_{Q t} = \bigcup_{j \in \cJ} [\partial \cM_{j,0}^+ \setminus \partial \cM]_{Q t} \supset \bigcup_{j \in \cJ} \cM_{j,0}^+ \triangle \cM_{j,t}^+\). Taking the complement and using that \(\cM_{j,0}^{+}\) are disjoint,
\[
\begin{split}
 \cM \setminus [\cS_{0}^{+,1}]_{Q t} &\subseteq \bigcap_{j\in \cJ} (\cM \setminus ( \cM_{j,0}^+ \cup \cM_{j,t}^+)) \cup \left( \cM_{j,0}^+ \cap \cM_{j,t}^+\right) \\
&\subseteq \bigcap_{j\in \cJ} (\cM \setminus\cM_{j,0}^+ )\cup \left( \cM_{j,0}^+ \cap \cM_{j,t}^+\right) = \bigcup_{j \in \cJ} \cM_{j,0}^+ \cap \cM_{j,t}^+.
\end{split}
\]
The statement follows by noticing thet any path in \(\cM\) intersecting \(\cM_{j,0}^+\) and \(\cM_{k,0}^+\), \(j \neq k\), must intersect also \(\cS_0^{+,1}\).
\end{proof}

\begin{lemma}\label{lem:aux-zero}
Let \(S_k \subseteq \cM \setminus  [\cS_0^{+,k+1}]_{D^{k}\!(k+1)Qt}\), \(k \in \bN_0\), be path connected. Then, for any \(j \in \{0,...,k\}\), the set \(\cF^{j}_t\left( S_k\right)\) is a path connected subset of \(\cM \setminus [\cS_0^{+,1}]_{Qt}\).
\end{lemma}
\begin{proof}
By Lemma \ref{lem:path-connected} and \ref{ass:sing-near}, for any \(x \in \cM \setminus [\cS_{0}^{+,1}]_{Q t}\), there is \(j \in \cJ\) such that
\begin{equation}\label{eq:bound-simple}
d \left(\cF_t(x), \cF_0(x) \right) = d\left(F_{j,t}(x), F_{j,0}(x) \right) \le Qt.
\end{equation}
We now prove the statement by induction. In the case \(k = 0\) there is nothing to prove. Let us assume that the statement holds for \(k \in \bN_0\) and let \( S_{k+1} \subseteq  \cM \setminus  [\cS_0^{+,k+2}]_{D^{k+1}\!(k+2)Qt}\) be path connected. Since \(S_{k+1} \subseteq  \cM \setminus  [\cS_0^{+,k+1}]_{D^{k}\!(k+1)Qt}\), the fact that the statement holds for \(j \in \{0,...,k\}\) is a consequence of the inductive assumption, and we need to consider only the case \(j = k+1\). By the last sentence and Lemma \ref{lem:path-connected}, there exist \(j_1,..., j_{k}, j_{k+1} \in \cJ\), such that \(\cF_t^{k+1}\left(S_{k+1}\right) = F_{j_{k+1},t} \circ F_{j_{k},t} \circ ....\circ F_{j_1,t}\left(S_{k+1}\right)\). Because \(F_{j_{k+1},t} \circ F_{j_{k},t} \circ ....\circ F_{j_1,t}\) is a diffeomorphism, this implies that \(\cF_t^{k+1}\left(S_{k+1}\right)\) is path connected. It remains to prove that each point is far enough from \(\cS_0^{+,1}\). Using \ref{ass:unif-extension} and triangular inequalities, for all \(x \in S_{k+1}\),
\begin{equation}\label{eq:inductive-step-00}
\begin{split}
&d\left( \!\cF_t^{k+1}(x), \cS_0^{+,1} \!\right)  \!\ge \! d\left( \!\cF_0\circ \cF_t^{k}(x), \cF_0\bigl(\cS_0^{+,2}\bigr) \!\right)\!-\! d\left(\!\cF_t\circ \cF_t^{k}(x), \cF_0 \circ \cF_t^{k}(x)\!\right) \\
&\ge D^{-1}d\left( \cF_t^{k}(x), \cS_0^{+,2}\right) - d\left(\cF_t\circ \cF_t^{k}(x), \cF_0 \circ \cF_t^{k}(x)\right) \ge ...\\
& \ge D^{-k-1}d\left(x, \cS_0^{+,k+2}\right) - \sum_{j=0}^{k} D^{-j}d\left(\cF_t\circ \cF_t^{k-j}(x), \cF_0 \circ \cF_t^{k-j}(x)\right),
\end{split}
\end{equation}
where we have also used that \(\cF_0\bigl(\cS_{0}^{+,k+1} \bigr) \supseteq \cS_{0}^{+,k} \setminus \cF_{0}\bigl( \partial \cM \bigr) \supseteq \cS_{0}^{+,k} \setminus \cS_0^{-}\) and by the transverasality assumption the set in the last inclusion coincides with \(\cS_{0}^{+,k}\) except for finitely many points. On the other hand, by the inductive assumption, \(\cF_t^{k-j}(x) \in \cM \setminus [\cS_{0}^{+,1}]_{Q t}\) for all \(j \in \{0,...,k\}\), and by \eqref{eq:bound-simple} \(d(\cF_t\circ \cF_t^{k-j}(x), \cF_0 \circ \cF_t^{k-j}(x)) \le Qt\). Hence, by \eqref{eq:inductive-step-00} and the fact that \(x \in \cM \setminus [\cS_0^{+,k+2}]_{D^{k+1}(k+2)Qt}\), (assuming that \(D \ge 1\) without loss of generality),
\[
d\left(\cF_t^{k+1}(x), \cS_0^{+,1}\right) > D^{-k-1}D^{k+1}((k+2)Q)t - \sum_{j=0}^{k}D^{-j}Qt  \ge Qt,
\]
proving the statement.
\end{proof}

\begin{lemma}\label{lem:aux}
For each \(k \in \bN_0\), there exist \(Q'_k, C'_k \in \bR^+\) such that, for all \(t \in [0,\ve]\),
\[
    \sup_{\cM \setminus [\cS_0^{+,k+1}  ]_{Q'_k t} } \left|\cF^k_t \circ \cF_0 - \cF^{k+1}_t\right|\le C'_k t.
\]
\end{lemma}
\begin{proof}
We prove the statement by induction. For the case \(k = 0\), the statement follows by \eqref{eq:bound-simple} with \(Q_0' = C_0' = Q\). Let us assume that the statement is true for \(k \in \bN_0\). For any \(x \in \cM \setminus [\cS_0^{+,k+2}]_{ D^{k+1}(k+2) Q t}\), one has that \(\cF_0(x)\) and \(\cF_t(x)\) belong to the same path connected subset of \(\cM \setminus [\cS_0^{+,k+1}]_{ D^{k}(k+1) Q t}\). \mbox{Indeed, by \ref{ass:unif-extension},} 
\[
d\bigl(\cF_0(x), \cS_0^{+,k+1}\bigr) \ge d\bigl(\cF_0(x), \cF_0 \bigl(\cS_0^{+,k+2}\bigr)\bigr) \ge D^{-1}d\bigl(x, \cS_{0}^{+,k+2}\bigr) > D^{k}(k+2) Q t,
\]
and, by \eqref{eq:bound-simple}, \(d(\cF_t(x), \cF_0(x)) \le Qt\). This implies that \(\cF_0(x)\) and \(\cF_t(x)\) cannot be separated by \(\cS_0^{+, k+1}\) and that
\[
    d (\cF_t(x), \cS_0^{+,k+1}) \ge d(\cF_0(x), \cS_0^{+,k+1}) - d(\cF_t(x), \cF_0(x)) >  D^k(k+1)Qt.
\] 
According to what we just proved, by Lemma \ref{lem:aux-zero}, \(\cF^k_t \circ \cF_0 (x)\) and \(\cF_t^{k+1}(x) = \cF^k_t \circ \cF_t (x)\) belong to the same path connected subset of \(\cM \setminus [\cS_0^{+,1}]_{Q t}\).
Therefore, by Lemma \ref{lem:path-connected}, there exists \(j \in \cJ\) such that 
\[
    d\bigl( \cF_t^{k+1} \circ \cF_0(x), \cF_t^{k+2}(x) \bigr) =  d\bigl(F_{j,t}\circ\cF_t^{k} \circ \cF_0(x) , F_{j,t}\circ \cF_t^{k+1} (x) \bigr).
\]
Hence, for any \(x \in \cM \setminus [\cS_0^{+, k+2}]_{\max \{ D^{k+1}(k+2) Q t, Q'_k t\}} \subseteq \cM \setminus [\cS_0^{+, k+1}]_{Q_k' t}\),  by \ref{ass:unif-extension} and the inductive assumption, we have for the quantity above,
\[
\begin{split}
   d\bigl(F_{j,t}\circ\cF_t^{k} \circ \cF_0(x) , F_{j,t}\circ \cF_t^{k+1} (x) \bigr)  \le D d\bigl(\cF_t^{k} \circ \cF_0(x) , \cF_t^{k+1}(x)\bigr) \le D C'_{k}t,
\end{split}
\]
concluding the proof.
\end{proof}
The following Lemma shows that at any fixed iteration, for \(t\) close to zero, the perturbed dynamics is close to the original one on most of the phase space. In the proof, we will denote by \(Q'_k\) and \(C'_k\) \mbox{the constants given by Lemma \ref{lem:aux}.}
\begin{lemma}
  \label{lem:n-map-sing}
 For each \(k \in \bN\), there exist \(Q_k, C_k \in \bR^+\) such that, for all \(t \in [0,\ve]\),
\[
     \sup_{\cM \setminus [\cS_0^{+,k}  ]_{Q_k t}} \left|\cF^k_t - \cF^k_0\right|\le C_k t.
\]
\end{lemma}
\begin{proof}
We prove the statement by induction. The case \(k = 1\) follows by Lemma \ref{lem:aux} (there, it is the case \(k = 0\)). Let us assume that the statement holds for some \(k \in \bN\). First notice that by \ref{ass:unif-extension}, for any \(x \in \cM \setminus [\cS_0^{+,k+1}]_{ DQ_k t}\), one has that \(\cF_0(x) \in \cM \setminus [\cS_0^{+,k}]_{Q_k t}\). Hence, by Lemma \ref{lem:aux} and the inductive assumption, for any \(x \in \cM \setminus [\cS_0^{+,k+1}]_{\max\{Q'_{k+1}, DQ_k \}t}\),
\[
\begin{split}
    d \bigl(\cF_t^{k+1}(x), \cF_0^{k+1}&(x)\bigr) \le d\bigl(\cF_t^{k+1}(x), \cF_t^k \circ \cF_0(x)\bigr)\\
& + d\bigl(\cF_t^k\circ\cF_0(x), \cF_0^{k+1}(x)\bigr) \le (C'_k + C_k)t,
\end{split}
\]
concluding the proof of the Lemma.
\end{proof}
Notice that in the previous Lemma there is no control on the constants \(Q_k, C_k\). Hence, the statement is useful for fixed \(k\) when \(t\) is small enough. That is the case also for the following result.
\begin{lemma}
  \label{lem:trasversality}
For each \(k \in \bN\), there exists \(C_k\in \bR^+\) such that, for all \(t, t' \in [0,\ve]\),
\[
    m\left([\cS^{+,k}_0 ]_{t} \cap [\cS_0^{-}]_{t'}\right) \le C_k t t'.
\]
\end{lemma}
\begin{proof}
By assumption \(\cS_{0}^{+,k}  \cap \cS_0^-\) consists of at most finitely many points. Let \(\left\{\gamma_{-, p}\right\}_p\), \(\left\{\gamma_{+,q}\right\}_q\) be finitely many \(\cC^1\) curves of finite length such that \(\overline \gamma_{-,p} \cap \overline \gamma_{+,q}\) consists of at most one point and \(\cS_0^- \subseteq \bigcup_{p}\gamma_{-,p}\), \(\cS_0^{+,k}  \subseteq \bigcup_{q} \gamma_{+,q}\). Since we are considering compact sets, if \(\overline\gamma_{-,p} \cap \overline\gamma_{+,q} = \emptyset\), for some \(p\) and \(q\), then \(d(\gamma_{-,p}, \gamma_{+,q})>0\). Therefore, if \(t\) and \(t'\) are small enough, then \( [\gamma_{-,p}]_t \cap [\gamma_{+,q}]_{t'} = \emptyset\) and the estimate follows by choosing \(C_k\) sufficiently big. On the other hand, if \(\gamma_{-,p} \cap \gamma_{+,q} = \{x\}\), \(x \in \cM\), then, by the transversality assumption, the angle between \(\gamma_{-,p}'(x)\) and \(\gamma_{+,q}'(x)\) is bigger or equal than some \(\alpha_0 >0\) (since there are only finitely many angles and they are all bigger than zero). Since the above curves are \(\cC^1\), they are well approximated by lines in a neighborhood of \(x\). Hence,  for \(t\) and \(t'\) small enough, the above set is contained in a parallelogram centered in \(x\), whose sides are proportional to \(t\) and \(t'\) and form an angle not smaller than \(\alpha_{0}\). This yields the statement for any pair \(\gamma_{-,p}\) and \(\gamma_{+,q}\). The general statement follows by iterating this estimate for all the pairs, which are in a finite number.
\end{proof}

It turns out that the previous estimates are sufficient to identify the form of the distributions in \(\partial \mathscr B\). To this aim, we prove the convergence for \(t\) tending to \(0\) of \(\int_{\cM} (\vf \circ \cF_{t}^k) \Deriv_{t} dm\). Notice that this does not follow by \eqref{eq:weak-star-limit} since \(\vf\) is composed with the perturbed dynamics. We need a preliminary result before stating the intended convergence in Lemma \ref{lem:limit-fixed-k}.

\begin{lemma}
  \label{lem:deriv-k}
 For each \(k \in \bN_0\), there exists \(C_k\in \bR^+\) such that, for any \(\vf \in \cC^1(\cM)\), \(t \in [0, \ve]\),
\[
  \left|\int_{\cM} \left(\vf \circ \cF_{t}^k - \vf \circ \cF_0^k\right) \Deriv_{t} dm \right| \le \|\vf\|_{\cC^1} C_k t.
\] 
\end{lemma}
\begin{proof}
 For \(k = 0\) the statement is trivial. Let \(k \in \bN\), \(Q_k \in \bR^+\) be given by Lemma \ref{lem:n-map-sing} and \(H \in \bR^+\) be given by Lemma \ref{lem:meas-bad-set}. Consider the following partition of \(\cM\):
  \[
    \begin{split}
      &\cM_1 = \cM \setminus  [\cS_0^{+,k} ]_{Q_k t}, \text{ }\cM_2 = [\cS_0^{+,k}]_{Q_k t}  \setminus [\cS_0^-]_{H t}, \text{ }\cM_3 =  [\cS_0^{+,k}]_{Q_k t} \cap [\cS_0^-]_{H t} .
    \end{split}
  \]
In the next estimates, \(C_k \in \bR^+\) is a constant big enough, possibly dependent on \(k\). By Lemma \ref{lem:n-map-sing} and Lemma \ref{lem:total-variation}, 
\[
  \begin{split}
   \biggl|\int_{\cM_1} \left(\vf \circ \cF_t^k - \vf \circ \cF_0^k\right)\Deriv_t dm\biggr|&\le \sup_{\cM_1}\left|\vf \circ \cF_t^k - \vf\circ \cF_0^k \right| \|\Deriv_t\|_{L^1}\\
   & \le C \|\vf\|_{\cC^1} \sup_{\cM \setminus  [\cS_0^{+,k}]_{Q_k t} }\left|\cF_t^k - \cF_0^k\right|  \le \|\vf\|_{\cC^1}C_k t.
  \end{split} 
\]
By Lemma \ref{lem:meas-bad-set}, \(\cG_t \supseteq \cM \setminus [\cS_0^-]_{H t} \supseteq \cM_{2}\) and so by Lemmata \ref{lem:single-perturbed-preimage-invariant-measure} and \ref{lem:neigh},
\[
     \begin{split}
      \biggl|\int_{\cM_2}  \left(\vf \circ \cF_t^k - \vf \circ \cF_0^k\right)\Deriv_t dm\biggr| &\le 2\|\vf\|_{\cC^0}\sup_{ \cM_2}\left|\Deriv_t\right| m \left(\cM_2 \right) \\
&\le 2\|\vf\|_{\cC^0} \sup_{\cG_t}\left|\Deriv_t\right| m \bigl([\cS_0^{+,k}]_{Q_k t}\bigr) \le \|\vf\|_{\cC^1}C_k t.
     \end{split}
\]
Finally by Lemma \ref{lem:trasversality} and Lemma \ref{lem:regularity-norms}, 
\[
  \begin{split}
    \biggl|\int_{\cM_3} \left(\vf \circ \cF_t^k - \vf \circ \cF_0^k\right)\Deriv_t dm\biggr| &\le 2\|\vf\|_{\cC^0} \sup_{\cM}|\Deriv_t|m(\cM_3) \\
&\le  2\|\vf\|_{\cC^0} \frac{C}{t} m \bigl( [\cS_0^{+,k}]_{Q_k t} \cap [\cS_0^-]_{H t}\bigr) \le \|\vf\|_{\cC^1}C_k t. 
  \end{split}
\]
Collecting all the previous estimates, we have the statement of the Lemma.
\end{proof}
Recall from the beginning of this section that \(\{s_m\} \subset (0, \ve] \) is a sequence tending to \(0\) as \(m\) tends to infinity such that equation \eqref{eq:weak-star-limit} holds. 
\begin{lemma}\label{lem:limit-fixed-k}
For each \(k \in \bN_0\) and \(\vf \in \cC^1(\cM)\),
\[
   \lim_{m \to \infty} \int_{\cM} (\vf \circ \cF_{s_{m}}^k) \Deriv_{s_m} dm = \widetilde \nu (\vf \circ \cF_0^k).
\]
\end{lemma}
\begin{proof}
Fix \(\ve \in \bR^+\). The function \(\vf \circ \cF_0^k\) is discontinuous at most on \(\cS_0^{+,k}\). Hence, there exists \(\vf_{\ve} \in \cC^0(\cM)\) such that \(\vf_{\ve} = \vf \circ \cF_0^k\) on \(\cM \setminus [\cS_0^{+,k}]_{\ve}\) and \(\|\vf_{\ve}\|_{\cC^0} \le \|\vf \circ \cF_0^k\|_{\cC^0} \le \|\vf\|_{\cC^0}\). 
We need some considerations before the main estimate. By definition of \(\vf_{\ve}\) and partitioning \([\cS_0^{+,k}]_{\ve}\), for any \(m \in \bN_0\),
\begin{equation}\label{eq:support-lusin}
\begin{split}
& \int_{\cM} |(\vf \circ \cF_{0}^k - \vf_{\ve})\Deriv_{s_m}| dm \le \int_{[\cS_0^{+,k}]_{\ve} \setminus [\cS_0^{-}]_{Hs_m}} |(\vf \circ \cF_{0}^k - \vf_{\ve})\Deriv_{s_m} |dm \\
+& \int_{[\cS_0^{+,k}]_{\ve} \cap [\cS_0^{-}]_{Hs_m}} |(\vf \circ \cF_{0}^k - \vf_{\ve})\Deriv_{s_m}| dm \le C\|\vf\|_{\cC^0} m([\cS_0^{+,k}]_{\ve}) \\
&+ C\|\vf\|_{\cC^0} \frac{ m \bigl([\cS_0^{+,k}]_{\ve} \cap [\cS_0^{-}]_{Hs_m}\bigr)}{s_m} \le C_k \|\vf\|_{\cC^0}\ve, 
\end{split}
\end{equation}
where we have also used Lemmata \ref{lem:single-perturbed-preimage-invariant-measure} (and \(\cG_t \supseteq \cM \setminus [\cS_0^-]_{H t}\)), \ref{lem:regularity-norms} in the second inequality and Lemmata \ref{lem:neigh}, \ref{lem:trasversality} in the third inequality. Recall the notation \(|\widetilde \nu|\) for the variation of \(\widetilde \nu\) and that, as it can be deduced by Proposition \ref{prop:form-measure}, \(|\widetilde \nu|\) is a finite positive (Radon) measure. By Lusin theorem (see e.g. \cite{Folland99} Theorem 7.10), there exists a set \(L_{\ve}\) and a continuous function \(g_{\ve}: \cM \to \bR\) such that \(\vf_{\ve}-\vf \circ \cF_0^k = g_{\ve}\) on \(L_{\ve}\) and \(|\widetilde \nu| (\cM \setminus L_{\ve}) \le \ve\). It is also possible to require that \(|g_{\ve}| \le | \vf_{\ve} - \vf \circ \cF_0^k|\). By adding and subtracting and using the definition \eqref{eq:weak-star-limit} of \(\widetilde \nu\),
\[
\begin{split}
 \widetilde \nu ( \vf_{\ve} - \vf \circ \cF_0^k) &= \widetilde \nu (g_{\ve}) + \widetilde \nu \bigl(( \vf_{\ve} - \vf \circ \cF_0^k - g_{\ve})\mathbbm 1_{\cM \setminus L_{\ve}}\bigr) \\
&= \lim_{m \to \infty}\int_{\cM} g_{\ve}\Deriv_{s_m}dm + \widetilde \nu \bigl(( \vf_{\ve} - \vf \circ \cF_0^k - g_{\ve})\mathbbm 1_{\cM \setminus L_{\ve}}\bigr).
\end{split}
\]
Therefore, using the properties of \(g_{\ve}\) and \(L_{\ve}\) and \eqref{eq:support-lusin},
\begin{equation}\label{eq:lusin}
\begin{split}
|\widetilde \nu ( \vf_{\ve} \!-\! \vf \circ \cF_0^k)| &\le \!\!\sup_{m \in \bN_0}\!\biggl|\int_{\cM}\!\!\!g_{\ve}\Deriv_{s_m}dm\biggr| + |\widetilde \nu|(\cM \!\setminus \!L_{\ve})\sup_{\cM}|  \vf_{\ve} \! - \! \vf \circ \cF_0^k \! - \! g_{\ve}|\\
&\le \sup_{m \in \bN_0}\int_{\cM} \!\!| \vf_{\ve} - \vf \circ \cF_{0}^k| |\Deriv_{s_m}| dm + 4\ve\|\vf\|_{\cC^0} \le C_k\|\vf\|_{\cC^0}\ve.
\end{split}
\end{equation}
Moreover, by \eqref{eq:weak-star-limit} again, there exists \(m_{*} \in \bN\) such that, for all \(m \ge m_{*}\), 
\begin{equation}\label{eq:weaak-star-prelusin}
\left|\int_{\cM} \vf_{\ve} \Deriv_{s_m} dm - \widetilde \nu (\vf_{\ve})\right| \le \ve.
\end{equation}
Using \eqref{eq:support-lusin}, \eqref{eq:weaak-star-prelusin} and \eqref{eq:lusin}, for all \(m \ge m_{*}\),
\[
\begin{split}
&\Biggl|\int_{\cM} (\vf \circ \cF_{0}^k) \Deriv_{s_m} dm -  \widetilde \nu (\vf \circ \cF_0^k)\Biggr| \le \left|\int_{\cM} (\vf \circ \cF_{0}^k - \vf_{\ve}) \Deriv_{s_m} dm\right|\\
&+ \left|\int_{\cM} \vf_{\ve} \Deriv_{s_m} dm - \widetilde \nu (\vf_{\ve})\right| + \left|\widetilde \nu (\vf_{\ve} - \vf\circ \cF_0^k)\right| \le C_k \|\vf\|_{\cC^0}\ve + \ve.
\end{split}
\]
Therefore, by Lemma \ref{lem:deriv-k} and the estimate above, there exists \(m_{*} \in \bN\) such that, for all \(m \ge m_{*}\), we have
\[
\begin{split}
\Biggl| \int_{\cM} (\vf \circ \cF_{s_{m}}^k) \Deriv_{s_m} dm -  \widetilde \nu (\vf \circ \cF_0^k)\Biggr|&
\le \Biggl| \int_{\cM} (\vf \circ \cF_{s_{m}}^k) \Deriv_{s_m} dm\\- \int_{\cM} (\vf \circ \cF_{0}^k) \Deriv_{s_m} dm\Biggl|
& + \Biggl|\int_{\cM} (\vf \circ \cF_{0}^k) \Deriv_{s_m} dm -  \widetilde \nu (\vf \circ \cF_0^k)\Biggr|\\
& \le  C_k \|\vf\|_{\cC^1} \left(s_{m} + \ve\right) + \ve,
\end{split}
\]
for some \(C_k \in \bR^+\) depending only on \(k\). This concludes the proof of the Lemma.
\end{proof}

We conclude the section proving Proposition \ref{prop:linear-resp-accumulation}. Let \(\{t_n\}\) be the sequence in the definition of \(\mathscr A\) and \(\mathscr B\).

\begin{proof}[\textbf{Proof Proposition \ref{prop:linear-resp-accumulation}}]
Let \(\widetilde \nu \in \partial \mathscr A\). Then, for some subsequence \(\{s_j\} \subset \{t_n\}\), \(j \in \bN\), \(s_j \to 0\) as \(j \to \infty\), one has that \(\nu_{s_j}/s_j \to \widetilde \nu\) weakly. By Lemma \ref{lem:telescopic}, one has that, for each \(j \in \bN\), \(\vf \in \cC^1(\cM)\),
\[
    \frac{\mu_{s_j}(\vf) - \mu_0(\vf) }{s_j} = \sum_{k \in \bN_0} \int_{\cM} ( \vf \circ \cF_{s_j}^k) \Deriv_{s_j} dm.
\]
By Lemma \ref{lem:decay-tot} (or assumption \ref{ass:decay-n}), 
\[
\biggl|\int_{\cM} ( \vf \circ \cF_{s_j}^k) \Deriv_{s_j} dm \biggr| \le C\Theta_k \|\vf\|_{\cC^1},
\]
uniformly in \(j \in \bN\), and by Lemma \ref{lem:limit-fixed-k}, for each \(k \in \bN_0\), \(\lim_{j \to \infty}\int_{\cM}( \vf \circ \cF_{s_j}^k ) \Deriv_{s_j} dm\) exists. Hence, we can exchange limit and sum and by Lemma \ref{lem:limit-fixed-k} again,
\[
\begin{split}
\lim_{j\to \infty}\sum_{k \in \bN_0} \int_{\cM}\!\!( \vf \circ \cF_{s_j}^k) \Deriv_{s_j} dm =\!\! \sum_{k \in \bN_0} \lim_{j\to \infty}\int_{\cM}\!\!( \vf \circ \cF_{s_j}^k) \Deriv_{s_j} dm = \!\!\sum_{k \in \bN_0} \widetilde \nu(\vf \circ \cF_0^k).
\end{split}
\]
By the last two formulas, there exists an element of \(\partial \mathscr B\) of the form \(I(\widetilde \nu)\) as in \eqref{eq:linear-resp-formulas}. To show that \(I\) is surjective we argue as follows. For any \(\widetilde \mu \in \partial \mathscr B\), there exists \(\{s_j\} \subset \{t_n\}\), \(s_j \to 0\) as \(j \to \infty\), such that \( (\mu_{s_j}(\vf) - \mu_0(\vf))/s_j \to \widetilde \mu\). By compactness of \(\mathscr A\), there exists a subsequence \(s_{j_q}\) and \(\widetilde \nu \in \partial \mathscr A\) such that \(\nu_{s_{j_q}}/s_{j_q} \to \widetilde \nu\) weakly, for \(q \to \infty\). Since we can exchange the limit and the sum as above, for any \(\vf \in \cC^1(\cM)\),
\[
\begin{split}
    \widetilde \mu(\vf) &=\lim_{q \to \infty} \frac{\mu_{s_{j_q}}(\vf) - \mu_0(\vf) }{s_{j_q}} = \lim_{q \to \infty}\sum_{k\in \bN_0} \int_{\cM} ( \vf \circ \cF_{s_{j_q}}^k) \Deriv_{s_{j_q}} dm\\
&=\sum_{k\in \bN_0}\lim_{q \to \infty} \int_{\cM} ( \vf \circ \cF_{s_{j_q}}^k ) \Deriv_{s_{j_q}} dm = \sum_{k \in \bN_0}\widetilde \nu \left( \vf \circ \cF_{0}^k\right) = I(\widetilde \nu)(\vf).
\end{split}
\] 
This concludes the proof of the Proposition.
\end{proof}

\section{A perturbation of the Arnold cat map}
\label{sec:examples}

In this section we consider a concrete example in which the assumptions of Theorems \ref{thm:Lip} and \ref{thm:linear-resp} are satisfied, and we compute the linear response formula for this case. We consider the following perturbation \(\catt: \cM \setminus \cS_t^{+} \to \cM\) of the famous Arnold cat map: 
\begin{equation}\label{eq:arnold-cat-map}
    \begin{pmatrix}x\\y 
\end{pmatrix} \to  \begin{pmatrix} \{x + y\}\\
\{x+(2-t)y\}\end{pmatrix},
\end{equation}
where \(t \in [0, 1/8]\) and \(\{a\}\) is the fractional part of \(a\). The singularity \(\cS_t^{+}\) will be introduced shortly. These maps are piecewise endomorphisms as defined in Section \ref{sec:setting}. To show the agreement with the abstract setting, we define the following open polygons:
\[
\begin{split}
  &\cM_{1,t}^+ = \left\{ x>0,\text{ } y>0,\text{ } y < -\frac{1}{2-t}x + \frac{1}{2-t}\right\};\\
  &\cM_{2,t}^+ = \left\{x > 0, \text{ }y > -\frac{1}{2-t}x + \frac{1}{2-t}\text{ }, y < -x+1 \right\};\\
  & \cM_{3,t}^+ = \left\{x < 1, \text{ }y<1, \text{ }y > -x+1, \text{ }y< -\frac{1}{2-t}x + \frac{2}{2-t}\right\} \\
  & \cM_{4,t}^+ = \left\{x < 1, \text{ }y<1, \text{ } y> -\frac{1}{2-t}x + \frac{2}{2-t}\right\}.
\end{split}  
\]
These polygons are plotted in \ref{Fig:1}. We have \(\cS_t^{+} = \cS_t^{+,1} \cup \partial \cM\) and
\begin{equation}
  \label{eq:SingCat+}
  \begin{split}
  \cS^{+,1}_t = \biggl\{(x, y) \in \cM& \text{ } : \text{ } y = -\frac{1}{2-t}x + \frac{1}{2-t} \\
&\text{ or }y = -\frac{1}{2-t}x + \frac{2}{2-t} \text{ or }y = -x +1 \biggr\}.\\
  &.
  \end{split}
\end{equation}
For \(j \in \{1,2,3,4\}\), let \(\catjt: \cM_{j,t}^{+} \to \cM_{j,t}^- = \catjt \left(\cM_{j,t}^+\right)\),
\begin{equation}\label{eq:cat-def}
\begin{split}
  &\catjt (x,y) = A_t\begin{pmatrix}
    x\\
    y
  \end{pmatrix} - v_i,\\
  & A_t = \begin{pmatrix}1&1\\ 1& 2-t\end{pmatrix}, \quad v_1 = \begin{pmatrix}
    0\\
    0
  \end{pmatrix}, \quad v_2 = \begin{pmatrix}
    0\\
    1
  \end{pmatrix}, \quad v_3 = \begin{pmatrix}
    1\\
    1
  \end{pmatrix}, \quad v_4 = \begin{pmatrix}
    1\\
    2
  \end{pmatrix}.
\end{split}  
\end{equation}
 A computation shows that the map \(\catt\) in \eqref{eq:arnold-cat-map} is given by the maps \(F_{A,j,t}\) as described in Section \ref{sec:setting}. Using \eqref{eq:cat-def} and \eqref{eq:SingCat+}, one has that
\begin{equation}
  \label{eq:singCat-}
  \begin{split}
  \cS_t^{-} \setminus \partial \cM = \bigl\{(x,y) \in \cM \text{ : } y = x &\text{ or }y = x - t  \text{ or }y = x + (1-t) \\
&\text{ or }y = (2-t)x \text{ or }y = (2-t)x - 1\bigr\}.
  \end{split}
\end{equation}

We denote by \(\cL^A_t\) the transfer operator associated with \eqref{eq:cat-def}. For \(r_0 > 0\) that will be specified shortly, we define
\[
\begin{split}
&\mathscr C^u = \left\{(\xi,\eta) \in \bR^2: \xi  \eta \ge 0\right\}, \\
&\mathscr C^s = \left\{(\xi, \eta) \in \bR^2: -\frac{\pi}{2} + r_0 \le \arctan\left(\frac{\eta}{\xi}\right) \le  - r_0  \right\}.
\end{split}
\]
The above cones are transversal in the sense that any two non-zero vectors of the respective cones form an angle uniformly bounded away from zero. We choose \(r_0\) small enough so that the direction of the segments of \(\cS_t^{+,1}\) lie inside \(\mathscr C^s\) for each \(t \in [0,1/8]\). We collect here a couple of basic observations that will be important in the next section. First notice that cone contraction is uniform in \(t\),
\begin{equation}\label{eq:cone-invariance}
    A_t \mathscr C^u \subseteq \Int \mathscr C^u, \quad A_t^{-1}\mathscr C^s \subseteq \Int \mathscr C^s, \text{ }\forall t \in \left[0,\frac 1 8\right],
\end{equation}
and there exist \(\lambda, \overline \lambda \in (1, \infty)\) such that
\begin{equation}\label{eq:minimal-contraction}
\lambda = \inf_{t \in \left[0, \frac 1 8\right]}\inf_{v \in \mathscr C^u} \frac{|A_t v |}{|v|}, \quad   \overline \lambda = \sup_{t \in \left[0, \frac 1 8\right]}\sup_{v \in \mathscr C^u} \frac{|A_t v |}{|v|}.
\end{equation}
Since \(\cS_{t}^{+,1}\) is aligned with \(\mathscr C^s\), by the second of \eqref{eq:cone-invariance}, one has that \(\cS_{t}^{+,k}\) is aligned with \(\mathscr C^s\). This last observation together with the fact that \(\cS_0^{-}\) is aligned with \(\mathscr C^u\) implies that \(\cS_{0}^{+,k}\) are transversal to \(\cS_0^{-}\), as requested by the general setting in Section \ref{sec:setting}. We also denote by \(E^u_t\) and \(E^s_t\) the expanding and contracting eigenspaces of the matrices \(A_t\). Notice that \(E^u_t \in \Int \mathscr C^u\) and \(E^s_t \in \Int \mathscr C^s\). Denote by \(\mu^{s}_t\), \(\mu^{u}_t\) the stable and unstable eigenvalues of \(A_t\), respectively. For each \(t \in [0, 1/8]\),
\begin{equation}\label{eq:contraction-uniform-cat}
\mu^s_t = \frac{\det A_t}{\mu^u_t} = \frac{1-t}{\mu^u_t} \le \frac{1}{\lambda}.
\end{equation}\\
We now introduce standard pairs adapted to our problem. Since we are looking at linear maps, we will consider only densities supported on segments. This choice is intended to present only the theory needed to arrive at the desired conclusion. Nonetheless, to the author's knowledge, the result presented in Proposition \ref{prop:fundamental-example} does not follow from any known result in the literature. The reason is that, in coupling results, it is generally assumed that the map has one mixing SRB measure to show fast decay of measures supported on standard pairs (and from this exponential mixing). In our example, we don't have this information \textit{a priori}. However, we will show that, for \(t\) small enough, the maps \(\catt\) are mixing, using that the unperturbed map \(\cato\) has very good statistical properties (see Lemma \ref{lem:first-coupling}).\footnote{The idea of obtaining mixing informations on the perturbed system from a good knowledge of the statistical properties of the unperturbed one is not new in the world of standard pairs. A similar approach has been carried out for billiards in \cite{CD09} and more recently in \cite{DL23}.} Notice the first part of Proposition \ref{prop:fundamental-example} could be obtained by considering perturbations of the transfer operator on suitable Banach spaces. This approach has been carried up in \cite{DL08}, but we cannot apply directly the results there since in \cite{DL08} is assumed that the maps \(\cF_t\) are all piecewise diffeomorphisms (while in our case the maps \(\catt\) are not surjective).\footnote{Banach spaces for cone hyperbolic maps that are well suited to study the family \(\catt\) have been introduced in \cite{BG09, BG10}. However, to the author's knowledge, there are still no results to handle perturbations in these spaces.}

We call standard segment (or standard curve) \(W\) any segment in \(\cM\) which is aligned with \(\mathscr C^u\).  For \(\delta \in \bR^+\), we let \(\cW_{\delta}\) be the set of standard segments of length at least \(\delta\) and we set \(\cW = \bigcup_{\delta \in \bR^+}\cW_{\delta}\). Notice also that since each \(W \in \cW\) is inside \(\cM\), one has that \(|W| \le \sqrt{2}\).
\begin{proposition}\label{prop:fundamental-example}
There exists \(\ve_0 \in (0, 1/8]\) such that, for each \(t \in (0, \ve_0]\), there exists a measure \(\mu_t\) such that
\[
     \mu_{t}(\vf) = \lim_{n \to \infty}m\left( \vf \circ \catt^n \right), \quad \forall \vf \in \cC^0(\cM).
\]
Moreover, there exist \(C \in \bR^+\) and \(\gamma \in (0,1)\) such that, for all \(t \in [0,\ve_0]\), for all \(W \in \cW\), \(\rho \in \cC^1(W, \bR)\), and all \(n \in \bN_0\),
\[
       \left|\int_{W} \left(\vf \circ \catt^n\right) \rho - \mu_t(\vf)\int_{W} \rho \right| \le C \gamma^n \|\rho\|_{\cC^1}\|\vf\|_{\cC^1}\quad \forall \vf \in \cC^1(\cM).
\]
\end{proposition}

The proof of Proposition \ref{prop:fundamental-example} is in Section \ref{sec:coupling}. We are now ready to show linear response for the perturbed cat maps. Let \(\mu_t\) be the measure introduced in Proposition \ref{prop:fundamental-example} and \((x,y)\) the usual coordinates on \(\bR^2\).

\begin{theorem}
  \label{thm:cat-lin-resp}
For any \(\vf \in \cC^1(\cM)\),
  \[
    \lim_{t \to 0}\frac{\mu_{t}(\vf) - m(\vf)}{t} = \sum_{k \in \bN_0}\left[m(\vf) - \int_0^1 (\vf \circ \cato^k)(s,s)ds \right].
  \]
\end{theorem}
\begin{proof}
We show that \(\{\catt\}_{t \in [0,\ve_0]}\) is an admissible family for \(\ve_0\) given by Proposition \ref{prop:fundamental-example}. \ref{ass:unif-extension0}, \ref{ass:unif-extension} and \ref{ass:sing-near} are a direct consequence of the definition. \ref{ass:smooth-0} is verified since \(\cato\) projects to a linear automorphism of the torus and so preserves the Lebesgue measure. \ref{ass:srb} coincides with the first assertion in Proposition \ref{prop:fundamental-example}. We now discuss \ref{ass:foliation}. In \ref{Fig:2}, we show the set \(\cS_{0}^{-} \cup \cS_{t}^{-}\). We consider partitions \(\mathscr P_t\) of \(\cM \setminus (\cS_0^{-}\cup \cS_t^{-})\) in standard segments. Since we expect that the conditional probability measures on the elements \(W\) of \(\mathscr P_t\) are of the order \(p_W \sim 1/|W|\) (they are probability measures), we would like to have partitions \(\mathscr P_t\) of \(\cM \setminus (\cS_{0}^{-} \cup \cS_{t}^{-})\) made by standard curves of length of order one (uniformly in \(t\)). However, this is not possible because of the following two sets, for \((x,y) \in \cM\),
\[
   R_t = \left\{ x>0, \text{ }x+1-t<y < 1\right\}, \quad B_t = \left\{ x<1, \text{ }x-t <y<2x-1\right\}.
\]
These are the red and blue triangles in \ref{Fig:2}. 
\hspace{10cm}
\hspace{10cm}
\hspace{10cm}

\begin{tikzpicture}[scale=2.5] 
{\crtcrossreflabel{(Fig.2)}[Fig:2]}


\begin{scope}

      \coordinate (a3) at (0,0);
      \coordinate (b3) at (1,0);
      \coordinate (c3) at (0,1);
      \coordinate (d3) at (1,1);
      
      \coordinate (cd1) at (1/8, 1);
      \coordinate (ac) at (0, 1-1/8);
      \coordinate (ab) at (8/15, 0);
      \coordinate (cd2) at (8/15, 1);
      \coordinate (ab1) at (1/8, 0);
      \coordinate (bd1) at (1,1-1/8);

      \draw [very thick, dotted] (a3) -- (b3) -- (d3) --(c3)--(a3);
      \draw [very thick, dotted] (a3)--(d3);
      \draw [very thick, dotted] (0,0)--(1/2, 1);
      \draw [very thick, dotted] (1,1)--(1/2, 0);
     
\end{scope}
\begin{scope}[xshift = 1.8cm]

      \coordinate (a3) at (0,0);
      \coordinate (b3) at (1,0);
      \coordinate (c3) at (0,1);
      \coordinate (d3) at (1,1);
      
      \coordinate (cd1) at (1/8, 1);
      \coordinate (ac) at (0, 1-1/8);
      \coordinate (ab) at (8/15, 0);
      \coordinate (cd2) at (8/15, 1);
      \coordinate (ab1) at (1/8, 0);
      \coordinate (bd1) at (1,1-1/8);

      \draw [ultra thick] (a3) -- (b3) -- (d3) --(c3)--(a3);
      \draw [very thick] (a3)--(d3);
      \draw [very thick] (cd1)--(ac);
      \draw [very thick] (ab1)--(bd1);
      \draw [very thick] (a3)--(cd2);
      \draw [very thick] (bd1)--(ab);
     
\end{scope}
  
\begin{scope}[xshift = 3.6cm]

      \coordinate (a3) at (0,0);
      \coordinate (b3) at (1,0);
      \coordinate (c3) at (0,1);
      \coordinate (d3) at (1,1);
      
      \coordinate (cd1) at (1/8, 1);
      \coordinate (ac) at (0, 1-1/8);
      \coordinate (ab) at (8/15, 0);
      \coordinate (cd2) at (8/15, 1);
      \coordinate (ab1) at (1/8, 0);
      \coordinate (bd1) at (1,1-1/8);

      \fill[red] (c3) -- (cd1)  -- (ac) -- cycle;
      \fill[blue] (d3)--(7/8, 6/8)--(bd1);

      \draw [ultra thick] (a3) -- (b3) -- (d3) --(c3)--(a3);
      \draw [very thick] (a3)--(d3);
      \draw [very thick] (cd1)--(ac);
      \draw [very thick] (ab1)--(bd1);
      \draw [very thick] (a3)--(cd2);
      \draw [very thick] (bd1)--(ab);
      \draw [very thick, dotted] (0,0)--(1/2, 1);
      \draw [very thick, dotted] (1,1)--(1/2, 0);

\end{scope}
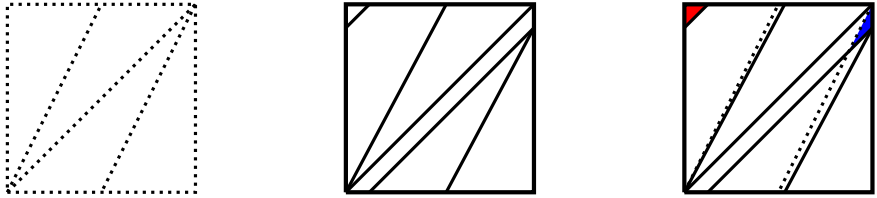
 \node at (2.5,-0.3) [align = left] {Figure 2: \(\cS_0^{-}\) (left), \(\cS_t^{-}\) (center), \(\cS_0^{-}\cup \cS_t^{-}\) (right), \(t = \frac 1 8\). In red and blue, \\
the sets \(R_t\) and \(B_t\) respectively.}; 
\end{tikzpicture}


\hspace{10cm}

Notice that all the other connected components of \(\cM \setminus (\cS_0^- \cup \cS_t^{-})\) do not `shrink' in all the directions of the unstable cone, as \(t\) tends to zero. Let \(\mathfrak R_t = \left\{R_s^t\right\}_{s \in [0,t]}\) and \(\mathfrak B_t = \left\{B_s^t\right\}_{s \in [0,t]}\) be the following partitions of \(R_t\) and \(B_t\) in standard segments:
\begin{equation}\label{eq:partition-triangles}
\begin{split}
 &R_{s}^t = \bigl\{ y =  x + 1-s: \text{ } 0 < x < s  \bigr\}, \text{ } B_s^t = \bigl\{ y =  x -t- s: \text{ } \frac{y+1}{2} < x < 1  \bigr\}.
\end{split}
\end{equation}
We define \(\mathscr P_{t}^2 = \mathfrak R_t \cup \mathfrak B_t\). To prove  \ref{ass:foliation}-\ref{ass:good-initial-cond2}, we take advantage of that the measure of these `bad' areas of the phase space is \(\cO(t^2)\). The conditional measures and the disintegration of \(m\) for the partitions \(\mathfrak R_t\) are
\begin{equation}\label{eq:disint-example}
   dm_{\mathfrak R_{t}}(s) = \frac{|R_s^t| ds}{\sqrt 2}, \quad p_{R_s^t} = \frac{1}{|R_s^t|}.
\end{equation}
For \(g \in \cC^1(\cM)\), \(t \in (0,\ve_0]\), by the second part of Proposition \ref{prop:fundamental-example},
\[
\begin{split}
    \cov(\vf, n, t; p_{R_s^t} g) &= \frac{1}{|R_s^t|}\left|\int_{R_s^t} (\vf \circ \catt^n) g  -  \mu_t(\vf) \int_{R_s^t} g\right|\\
& \le\frac{ C \gamma^n \|g\|_{\cC^1}\|\vf\|_{\cC^1}}{|R_s^t|}.
\end{split}
\]
Hence, by \eqref{eq:disint-example},
\begin{equation}\label{eq:decay-red}
\begin{split}
\frac{\int_{\mathfrak R_t} \cov(\vf, n, t; p_{R_s^t} g) dm_{\mathfrak R_t} }{t} &\le \frac{\int_{0}^t   \frac{C \gamma^n\|g\|_{\cC^1}\|\vf\|_{\cC^1}}{|R_s^t|}\frac{|R_s^t| ds}{\sqrt 2}}{t} \\
&\le C \gamma^n \|\vf\|_{\cC^1}\|g\|_{\cC^1}.
\end{split}
\end{equation}
Analogously, a similar computation for \(\mathfrak B_t\) yields, for \(t \in (0, \ve_0]\),
\begin{equation}\label{eq:decay-blu}
\frac{\int_{\mathfrak B_t} \cov(\vf, n, t; p_{B_s^t} g) dm_{\mathfrak B_t} }{t} \le C \gamma^n \|\vf\|_{\cC^1}\|g\|_{\cC^1}.
\end{equation}
By \eqref{eq:decay-red} and \eqref{eq:decay-blu}, we have that condition \ref{ass:foliation}-\ref{ass:good-initial-cond2} is satisfied with \(\Theta_n = C\gamma^n\). For \(\mathscr P_t^1\), we first notice that is possible to partition \(\cM \setminus (\cS_0^{-} \cup \cS_t^{-} \cup R_t \cup B_t)\) with a collection of standard segments of length not smaller than some given constant (the value of such constant is unimportant but \(1/2\) suffices). This implies that it is possible to define \(\mathscr P_t^1\) such that 
\begin{equation}\label{eq:C1-norm-conditional}
\|p_W\|_{\cC^1} \le C
\end{equation}
 uniformly in \(W \in \mathscr P_t^1\) and \(t \in (0,\ve_0]\). We give more details on the construction of \(\mathscr P_t^1\) and the proof of \eqref{eq:C1-norm-conditional} in Appendix \ref{sec:app}. By the second part of Proposition \ref{prop:fundamental-example}, for every \(t \in (0,\ve_0]\), \(W \in \mathscr P_t^1\) and  \(\rho_W \in \cC^1 (W , \bR)\),
\[
\begin{split}
    \cov(\vf, n, t; p_W \rho_W ) &= \biggl| \int_{W}(\vf \circ \catt^n) p_W \rho_W  -  \mu_t(\vf) \int_{W}p_W \rho_W \biggr)\biggr| \\
&\le C \|p_W \rho_W\|_{\cC^1}\|\vf\|_{\cC^1}\gamma^n \le C \gamma^n \|\rho_W\|_{\cC^1}\|\vf\|_{\cC^1},
\end{split}
\]
where in the last inequality we have used \eqref{eq:C1-norm-conditional}. Therefore, condition \ref{ass:foliation}-\ref{ass:good-initial-cond1} is satisfied. The statement of Theorem \ref{thm:cat-lin-resp} follows by Theorem \ref{thm:linear-resp} and Lemma \ref{lem:formula-lin-resp}, recalling that \(m\) is invariant for \(\cato\).
\end{proof}

\begin{lemma}\label{lem:formula-lin-resp} 
For any \(\vf \in \cC^0(\cM)\),
  \[
       \lim_{t \to 0} \frac{\int_{\cM}(\vf \circ \catt - \vf)dm}{t} =   m (\vf)-\int_{0}^1 \vf(s,s)ds.
  \]
\end{lemma}
\begin{proof}
By a direct computation (see also \ref{Fig:1}), for all \(t \in (0, 1/8]\),
\begin{equation}\label{eq:direct-com}
    \frac{\cL^A_{t}1 - 1}{t} = \frac{1}{1-t} \mathbbm 1_{\cM \setminus H_t} - \frac{1}{t}\mathbbm 1_{H_t} ,
\end{equation}
where \(H_t\) is the white portion of \(\cM\) in the right side of \ref{Fig:1}, i.e.,
\[
H_t = D_t \cup R_t, \quad D_t =  \bigcup_{s \in (0,t)} \left\{(x,y) \in \cM : \text{ }y = x - s\right\}, \quad R_t = \bigcup_{s \in (0,t)} R^t_s,
\]
where \(R_s^t\) were defined in \eqref{eq:partition-triangles}. Since \(\vf\) is continuous on a compact set, it is uniformly continuous, and for any \(\ve \!\in \!\bR^+\), there exists \(t_{\ve} \!\in \!\bR^+\) such that, \mbox{for all \(x \!\in \![0,1]\),}
\begin{equation}\label{eq:continuity}
    |\vf(x,y) - \vf(x,x)| \le \ve, \quad \forall y \in (x-t_{\ve}, x + t_{\ve}) \cap \cM.
\end{equation}
Since \(m \left(R_t\right) = t^2 /2\), \(m(D_t) = t - t^2/2\), and by \eqref{eq:continuity}, for any \(t \in (0, t_{\ve})\),
\begin{equation}\label{eq:initial-resp1}
\begin{split}
     &\biggl|\int_{H_t} \frac{\vf(x,y)}{t} dm - \int_0^1\vf(s,s)ds\biggr|\le \biggl|\int_{D_t} \frac{\vf(x,y)}{t} dm - \int_0^1 \vf(s,s)ds\biggr|\\
&+ \frac{m(R_t)}{t}\|\vf\|_{\cC^0} \le \biggl|\int_{D_t}\frac{\vf(x,x)}{t}dm - \int_0^1 \vf(s,s)ds \biggr|\\
&+ \int_{D_t}\frac{|\vf(x,y) - \vf(x,x)|}{t}dm + \frac{t}{2}\|\vf\|_{\cC^0}  \\
&\le \biggl(1 - \frac{m(D_t)}{t}\biggr)\|\vf\|_{\cC^0}+ \ve\frac{m(D_t)}{t} + \frac{t}{2} \|\vf\|_{\cC^0} \le \ve + t\|\vf\|_{\cC^0}.
\end{split}
\end{equation}
And, recalling that \(t \le 1/8\),
\begin{equation}\label{eq:initial-resp2}
\begin{split}
\biggl|\int_{\cM \setminus H_t} \frac{\vf}{1-t} dm - m(\vf)\biggr|&\le  \int_{\cM \setminus H_t}\left|\frac{\vf}{1-t} - \vf \right|dm +  \int_{H_t} |\vf| dm \\
& \le  \left(\frac{t}{1-t} + m\left(H_t\right)\right)\|\vf\|_{\cC^0}\le 3 t \|\vf\|_{\cC^0}.
\end{split}
\end{equation}
By \eqref{eq:direct-com}, \eqref{eq:initial-resp1} and \eqref{eq:initial-resp2},
\[
\begin{split}
   \lim_{t \to 0} \frac{\int_{\cM}(\vf \circ \catt - \vf)dm}{t} &= \lim_{t \to 0} \int_{\cM} \frac{\cL_{t}^A 1 - 1}{t}\vf dm = m (\vf)-\int_{0}^1 \vf(s,s)ds,
\end{split}
\]
concluding the proof of the Lemma.
\end{proof}

We conclude this section with a couple of remarks.
\begin{remark}[Transversality to the topological class]
The perturbation \(\catt\) is not tangent to the topological class of \(\cato\). By this we mean that the map \(\catt\) is not close  modulo errors of size \(o(t)\) to any map topologically conjugated to \(\cato\) (see e.g. \cite{BS08} for a discussion of perturbations tangent to the topological class in the unimodal setting). This happens because \(\cato\) is a bijection while the perturbation opens a hole of size \(O(t)\) in the image. In these regards, the family \(\catt\) is more similar to a 2D hyperbolic analog of a perturbation of the tent map in which we move down the tip of the tent near the full-branch map (losing surjectivity).
\end{remark}

\begin{remark}[Multiple accumulation points]
It is possible to craft examples for which the set \(\mathscr B\) in \eqref{eq:A-Bsets} has multiple accumulation points. Let \(\widetilde \catt\) be a perturbation of \(\cato\) such that \(\widetilde \catt = \catt\), for \(t \in [0,1/8]\setminus \bQ\), and \(\widetilde \catt = \cato\), for \(t \in [0,1/8]\cap \bQ\). Consider a sequence \(t_n \searrow 0\) that contains infinitely many rational and infinitely many irrational numbers. In this case, both the elements of \({\cC^1(\cM)}^{\star}\) \(\vf \to \sum_{k \in \bN_0} \left[m(\vf) - \int_0^1 (\vf \circ \cato^k)(s,s)ds \right]\) and the null element \(\vf \to 0\) are accumulation points of \(\mathscr B\). This yields a (trivial) application of the general statement of Proposition \ref{prop:linear-resp-accumulation}.
\end{remark}

\section{Coupling} \label{sec:coupling}
This section is entirely devoted to the proof of Proposition \ref{prop:fundamental-example} and, modulo the definition of the maps \(\catt\) at the beginning of Section \ref{sec:examples}, is self-contained. We start with some properties of the singularity set.
First we notice that for each \(k \in \bN\) the set \(\cS_t^{+,k}\) is a finite collection of straight lines. This is a consequence of definition \eqref{eq:future-disc}, the fact that \(\partial \cM^+_{j,t}\) are lines and \(\catjt^{-1}\) are affine, for all \(j \in \cJ\). Next Lemma is about the continuation of the singularity lines, a property that is observed also in billiards (see \cite{CM06}, chapter 4).
\begin{lemma}\label{lem:geom-disc}
 For any \(k \in \bN\), \(x \in \cS_t^{+,k}\), there is a decreasing path \(\gamma_x\) inside \(\cS_t^{+,k}\) such that the endpoints of \(\gamma_x\) belong to \(\partial \cM\).
\end{lemma}
\begin{proof}
We prove the statement by induction. In the case \(k=1\), the set \(\cS_t^{+,1} = \cS_{t}^+ \setminus \partial \cM\) is a collection of lines with slope in \(\mathscr C^s\), whose endpoints are in \(\partial \cM\) (see \ref{Fig:1}) and the statement is true. Suppose now the statement is true for some \(k \in \bN\) and let \(x \in \cS_t^{+,k+1}\setminus \cS_t^{+,1}\). Then \(\catt(x) \in \cS_t^{+,k}\) and let \(\gamma_{\catt(x)}\) be given by the statement. By definition \(\gamma_{\catt(x)} \cap \cS_t^{-} \neq \emptyset\) and let \(\widetilde \gamma_{\catt(x)} \subseteq \gamma_{\catt(x)}\) be a path containing \(\catt(x)\) such that \(\widetilde \gamma_{\catt(x)} \cap \cS_t^{-} = \{\text{both endpoints of }\widetilde \gamma_{\catt(x)}\}\). By \eqref{eq:cone-invariance}, the set \(\catt^{-1}(\widetilde \gamma_{\catt(x)})\) is a union of lines whose slope is in \(\mathscr C^s\) and so are decreasing. Since \(\catt^{-1}\) is discontinuous at most on \(\cS_t^{-}\), we have that \(\catt^{-1}(\widetilde \gamma_{\catt(x)})\) is a (continuous) decreasing path that contains \(x\) whose endpoints belong to \(\partial \cM \cup \cS_t^{+,1} = \text{“} \catt^{-1}(\cS_t^{-})"\). By the case \(k = 1\), we can connect both these endpoints to \(\partial \cM\) via a single decreasing path and we conclude.
\end{proof}

We say that a segment varies continuously with \(t\) if its endpoints are described by continuous functions of \(t\) (the segment can reduce to a point if the endpoints happen to coincide).

\begin{lemma}\label{lem:continuity-sing-set}
For each \(k \in \bN\), \(\cS_t^{+,k}\) is a finite union of segments which vary continuously with \(t\).
\end{lemma}
\begin{proof}
We prove the statement by induction. The case \(k = 1\) follows directly from the expression \eqref{eq:SingCat+} of \(\cS_t^{+,1}\). Assume that the statement holds for \(k \in \bN\). Recalling \eqref{eq:future-disc}, \(\cS_{t}^{+,k+1} = \catt^{-1}(\cS_t^{+,k}) \cup \cS_t^{+,1}\) and it is sufficient to prove the statement for \(\catt^{-1}(\cS_t^{+,k})\). By the inductive hypothesis \(\cS_t^{+,k}\) is a finite collection of segments that vary continuously with \(t\) and let \(S_t\) one of those segments. We have \(\catt^{-1}(S_t) = \bigcup_{j =1}^4 \catjt^{-1}(S_t \cap \cM_{j,t}^{-})\). Since \(\cM_{j,t}^{-}\) are convex, \(S_t \cap \cM_{j,t}^{-}\) is either empty or is a segment and, by transversality between \(S_t\) and \(\cS_t^{-}\), varies continuously with \(t\). Here, we are also using that each segment in \(\cS_t^{-}\) varies continuously with \(t\) as it is clear from their explicit formula \eqref{eq:singCat-}.
The statement follows from the fact that  \(\catjt^{-1}\) are affine and continuous in \(t\).
\end{proof}

In other words, the discontinuity segments either move continuously with \(t\) or gradually grow from a point (a corner) of already existing singularities. For \(t \in [0,1/8]\), \(k \in \bN_0\), let \(\cM^k_t\) be the collection of maximal open connected sets of \(\cM \setminus \cS^{+,k}_t\). We introduce a version of the \textit{complexity} uniform in the parameter defining the perturbation. Let \(K_k \in \bN_0\),
\[
  K_k = \sup_{t \in [0,\frac{1}{8}]}\sup_{x \in \cS_t^{+,k}}\left\{\# \{O \in \cM^k_t: x \in \partial O\}\right\}.
\]
The above is the maximal number of segments of \(\cS_t^{+,k}\) that meet at some point. Next Lemma is often summarized by saying that hyperbolicity beats cutting.
\begin{lemma}\label{lem:hyp-beats-cutting}
  There exists \(N_0 \in \bN\) such that,
  \[
       \frac{K_{N_0} +1}{\lambda^{N_0}} < 1.
  \]
\end{lemma}
\begin{proof}
The statement follows by the proof of Proposition 17 in \cite{BG09} that is based on \cite{Buzzi97}. Indeed, in \cite{BG09} it is shown that \(\sup_{x \in \cS_t^+}\left\{\# \{O \in \cM^k_t: x \in \partial O\}\right\}\) grows polynomially in \(k\), with parameters depending only on the dimension and the number of sides of the partition \(\{\cM_{j, t}^+\}_{j \in \cJ}\). Since these parameters are uniform in \(t\), the statement is proved.
\end{proof}
Recall the definition of \(\cW\) before Proposition \ref{prop:fundamental-example} in Section \ref{sec:examples}.
\begin{lemma}\label{lem:trasversality2}
  For any \(k \in \bN_0\), there exists \(P_k \in \bR^+\) such that, for any \(t \in [0,1/8]\), \(W \in \cW\),
  \[
  \begin{split}
    \#\{W \cap \cS_t^{k,+}\} \le K_{k} + P_{k}|W|,
  \end{split}
  \]
  or, equivalently,
  \[
  \begin{split}
\# \{W': W' \text{ is a maximal connected}& \text{ component} \\
&\text{ of }W \setminus \cS_t^{k,+}\} \le K_{k}+1 + P_k |W|.
  \end{split}
  \]
\end{lemma}
\begin{proof}
For each \(k \in \bN_0\), there exists \(\delta_0 \in \bR^+\), uniform in \(t \in [0,1/8]\), such that any \(W \in \bigcup_{\delta \le \delta_0}\cW_{\delta}\) intersects \(\cS_t^{+,k}\) at most \(K_k\) times. Indeed, by Lemma \ref{lem:continuity-sing-set} and because \(t\) varies in a compact set, there exists a lower bound on the distance between segments in \(\cS_t^{+,k}\) that for any \(t \in [0,1/8]\) do not intersect. For \(W \in \cW\), we divide \(W\) in \(\lfloor |W|/\delta_0 \rfloor\) components of size \(\delta_0\) plus a reminder. By the defining property of \(\delta_0\), each of these components can intersect \(\cS_t^{+,k}\) in at most \(K_k\) points. Hence, setting \(P_k = K_k/\delta_0\), the total number of intersection is bounded by
\[
\left \lfloor\frac{|W|}{\delta_0} \right \rfloor K_k + K_k \le P_k |W| + K_k,
\]
yielding the claim.
\end{proof}

We call standard family any finite collection \(\cG = \left\{(W_j, p_j)\right\}_{j \in \cI}\), where \(\cI \subset \bN\), \(W_j \in \cW\) and \(p_j \in (0,1]\), \(\sum_{j\in \cI}p_j = 1\). To any standard family \(\cG\), we associate a probability measure \(\mu_{\cG}\), defined by, for \(\vf \in \cC^0(\cM)\),
\[
    \mu_{\cG} (\vf) = \sum_{j \in \cI}p_j\frac{\int_{W_j}\vf}{|W_j|}.
\]
It is useful to think at \(p_j\) as `the proportion of mass in the segment \(W_j\)'. The segments \(W_j\) need not be disjoint. We say that two standard families \(\cG_1\) and \(\cG_2\) are \textit{equivalent} if their measures coincide, i.e., for any \(\vf \in \cC^0(\cM)\), \(\mu_{\cG_1}(\vf) = \mu_{\cG_2}(\vf)\), and we write \(\mu_{\cG_1} = \mu_{\cG_2}\). We also define the \textit{regularity} \(\cZ(\cG)\) of \(\cG\) as
\begin{equation}
  \cZ (\cG) = \sum_{j \in \cI}p_j \frac{1}{|W_j|}.
\end{equation} 
It is possible that \(\cZ(\cG) \neq \cZ(\cG')\) even if \(\mu_{\cG} = \mu_{\cG'}\). Sometimes it will be useful to consider the \textit{same} standard family but with modified weights. Let \(\cG\) as above and, for some finite indices \(\cI_j \subset \bN\), \(j \in \cI\), and \(p'_{j,k} >0\), \(\sum_{k \in \cI_j} p'_{j,k} = 1\), set
\begin{equation}\label{eq:indistinguishable}
\widetilde \cG = \{( W_{j,k}, p_{j,k})\}_{(j,k)}, \quad  W_{j,k} = W_j, \quad  p_{j,k} = p_j p'_{j,k}.
\end{equation}
In this case, \(\mu_{\cG} = \mu_{\widetilde \cG}\) and \(\cZ(\cG) = \cZ(\widetilde \cG)\). In general, if two standard families \(\cG\) and \(\widetilde \cG\) can be obtained by modifying the weights as in \eqref{eq:indistinguishable}, we say that \(\cG\) and \(\widetilde \cG\) are \textit{indistinguishable} and write \(\cG \simeq \widetilde \cG\). Given a standard family \(\cG = \left\{W_j, p_j\right\}_{j \in \cI}\), we define its evolution \(\catt \cG\) as 
\begin{equation}\label{eq:def-evolution-sf}
\catt \cG = \{(\catt (W_{j,k}), p_{j,k})\}_{(j,k)}, \quad p_{j,k} = p_j \frac{| W_{j,k}|}{| W_{j}|},
\end{equation}
where \(W_{j,k}\) are the maximal connected components of \(W_j \setminus \cS_t^{+}\), \(j \in \cI\). By \eqref{eq:cone-invariance}, \(\catt \cG\) is a standard family. By a change of variable, we have that the evolution of standard families describes the pushforward of the associated measure, i.e.,
\begin{equation}\label{eq:pushforward-cat}
\begin{split}
\mu_{\cG}\left(\vf \circ \catt \right)  &=  \sum_{j \in \cI}p_j\frac{\int_{W_j}\vf \circ \catt}{|W_j|} = \sum_{j \in \cI}\sum_{k} p_j\frac{\int_{W_{j,k}}\vf \circ \catt}{|W_j|}\\
& = \sum_{j \in \cI}\sum_{k} p_j \frac{\int_{\catt(W_{j,k})} \vf }{|W_j|}\frac{|W_{j,k}|}{|\catt(W_{j,k})|} = \mu_{\catt \cG}(\vf).
\end{split}
\end{equation}
Notice that if \(\cG \simeq \widetilde \cG\) then \(\catt^p \cG \simeq \catt^p \widetilde \cG \), for all \(p \in \bN\). Next Lemma gives more information about the weights. For any \(W \in \cW\), there is naturally associated the standard family \(\cG_{W} = \{(W,1)\}\).
\begin{lemma}\label{lem:weigths-single-curve}
 Let \(W \in \cW\), \(n \in \bN_0\) and denote (renaming the indices) \(\catt^n \cG_{W} = \left\{(W_q, p_q)\right\}\). Then \(p_q = |W_q|/|\catt^n (W)|\) and in particular \(p_q \ge  |W_q|/(\overline \lambda^n |W|)\).
\end{lemma} 
\begin{proof}
We prove the statement by induction. For \(n = 0\) the statement is trivially verified. Suppose the statement is true for some \(n \in \bN_0\). Let \((W_j, p_j) \in \cF_t^n \cG_W\) and let \(W_{j,k}\) the connected components of \(W_j \cap \cS_t^{+}\). 
Then, writing \(W_q = \catt (W_{j,k})\), using the inductive hypothesis and recalling \eqref{eq:def-evolution-sf}, the weight associated to \(W_q\) is
\[
p_{q} = p_j \frac{| W_{j,k}|}{|W_j|} =  \frac{|W_j|}{|\catt^n (W) |}  \frac{|W_{j,k}|}{|W_j|}  = \frac{| W_{j,k}|}{|\catt^n (W) |} = \frac{|\catt (W_{j,k})|}{|\catt^{n+1}(W)|} = \frac{|W_q|}{|\catt^{n+1}(W)|},
\] 
where in the fourth equality we have used that \(W_{j,k} \subseteq \catt^n (W)\) and thus \(W_{j,k}\) and \(\catt^n (W)\) have the same slope. This shows the first part of the Lemma. The second statement follows from what we just proved and the second of \eqref{eq:minimal-contraction}.
\end{proof}
In the next Lemma we show that small segments grow on average when iterated by the dynamics until they reach a certain length scale given by the properties of the map. While similar growth Lemmas are ubiquitous in the study of hyperbolic systems with discontinuities, the following is very close to the one in \cite{C99}. Let \(N_0\) be given by Lemma \ref{lem:hyp-beats-cutting}.
\begin{lemma}[Growth Lemma]\label{lem:growth-lemma}
  There exist \(z \in (0,1)\) and \(Z \in \bR^+\) such that, for any \(t \in [0, 1/8]\), \(k \in \bN_0\), and standard family \(\cG\),
  \[
      \cZ(\catt^{k N_0} \cG) \le z^k \cZ (\cG) + Z.
  \]
  Furthermore, there exist \(Z_{*}, Z_{**}, Q_1,Q_2 \in \bR^+\), such that, for all \(t \in [0, 1/8]\),
  \[
  \begin{split}
    &\cZ(\catt^{p}\cG) \le \max\{Z_{**}\cZ(\cG), Z_{*}\}, \quad \forall p \in \bN_0,\\
    &\cZ(\catt^{p}\cG) \le Z_{*}, \quad \forall p \ge Q_1\left|\ln \cZ(\cG)\right| + Q_2.
  \end{split}
  \]
\end{lemma}
\begin{proof}
 Let \(\cG = \{W_j, p_j\}_{j \in \cI}\). For each \(j \in \cI\), let \(\{W_{j,j'}\}\), \(j' \in \cI_j\), be a partition of \(\catt^{p}(W_j)\), \(p \in \bN_0\), in connected components. The cardinality of \(\cI_j\) is not bigger than the number of connected components of \(W_j \setminus \cS^{+,p}_t\). Therefore, by Lemma \ref{lem:trasversality2}, there exists a constant \(P_{p} \in \bR^+\), such that, for each \(j \in \cI\),
\begin{equation}\label{eq:cardinality-new-partition}
  \# \cI_j \le K_{p} + 1 + P_{p}|W_j|.
\end{equation}
Moreover, by Lemma \ref{lem:weigths-single-curve}, the weight associated to \(W_{j,j'}\) is
\begin{equation}\label{eq:iterated-evol}
p_{j,j'} = p_j \frac{|W_{j,j'}|}{|\catt^p(W_j)|}.
\end{equation}
Hence, by \eqref{eq:cardinality-new-partition} and \eqref{eq:iterated-evol},
\begin{equation}\label{eq:estimate-evol-Z}
\begin{split}
  \cZ(\catt^{p}\cG) &= \sum_{j \in \cI}\sum_{j' \in \cI_{j}} p_{j,j'} \frac{1}{|W_{j,j'}|}= \sum_{j \in \cI}\sum_{j' \in \cI_{j}}p_j\frac{|W_{j,j'}|}{|\catt^{p} W_j|}\frac{1}{|W_{j,j'}|} \\
  &\le \sum_{j \in \cI}p_j\frac{K_{p} +1 +P_{p}|W_j|}{\lambda^{p}|W_j|}\le \frac{K_{p} + 1}{\lambda^{p}}\cZ(\cG) + \frac{P_{p}}{\lambda^{p}}.
\end{split}
\end{equation}
We now prove the first statement of the Lemma. The case \(k = 1\) follows by \eqref{eq:estimate-evol-Z} and Lemma \ref{lem:hyp-beats-cutting}. For general \(k \in \bN\), it follows by induction, using \eqref{eq:estimate-evol-Z} with  \(\catt^{k N_0} \cG\) in place of \(\cG\) and \(p = N_0\). Moreover, equation \eqref{eq:estimate-evol-Z} with \(p = 1\) yields
\begin{equation}\label{eq:bad-estimate}
  \cZ \left(\catt \cG\right) \le\frac{K_{1} + 1}{\lambda}\cZ(\cG) + \frac{P_{1}}{\lambda} \le\frac{K_1 + 1+ \sqrt{2} P_1}{\lambda} \cZ(\cG),
\end{equation}
where we have used that \(\cZ(\cG) \ge 1/\sqrt{2}\) (recall the upper bound for the lengths of standard segments). Writing \(p = kN_0 + q\), \(k \in \bN_0\), \(q \in \{0,...,N_0-1\}\),
\begin{equation}\label{eq:growth-eq-proof}
\begin{split}
\cZ(\catt^p \cG) &\le \left( \frac{K_1 + 1+ \sqrt{2}P_1}{\lambda}\right)^q \cZ(\catt^{kN_0}\cG) \le Z_{**} \max\left\{\cZ(\cG), \frac{Z}{1-z}\right\} \\
&\le \max\{Z_{**}\cZ(\cG), Z_{*}\},
\end{split}
\end{equation}
where we have used \eqref{eq:bad-estimate} repeatedly in the first inequality, the first part of the Lemma in the second, and we have set
\begin{equation}\label{eq:star-2star}
Z_{**} = \left(\frac{K_1 + 1+ \sqrt{2}P_1}{\lambda} + 1\right)^{N_0}, \quad Z_{*} = \frac{2ZZ_{**}}{1-z}.
\end{equation}
The third claim can be obtained by the first part of the Lemma together with the first inequality in \eqref{eq:growth-eq-proof} and the definition of \(Z_{*}\) in the equation above.
\end{proof}
We stress the uniformity in the parameter \(t\) in Lemma \ref{lem:growth-lemma}. Keeping track of the regularity of standard families will be an important task. We say that \(\cG\) is a \textit{regular} standard family if\footnote{\(20000\) is just a very big number that comes from the far from optimal estimates in Lemmata \ref{lem:int-rect} and \ref{lem:first-coupling}.}
\begin{equation}\label{eq:regular-sf-def}
\cZ(\cG) \le \widetilde Z, \quad \widetilde Z = 20000Z_{*},
\end{equation}
where \(Z_{*}\) is defined in \eqref{eq:star-2star}. The following simple Lemma unveils the connection between \(\cZ\) and the proportion of mass in long segments in a standard family.
\begin{lemma}\label{lem:proportion-of-good-sp}
  For any standard family \(\cG\) and any \(\delta \in \bR^+\),
  \[
  \begin{split}
    &\sum_{j \in \cI: |W_j| \le \delta}p_j \le \delta \cZ (\cG), \\
  \end{split}
  \]
  or, equivalently,
  \[
    \sum_{j \in \cI: |W_j| > \delta}p_j \ge 1 - \delta \cZ (\cG). 
  \]
\end{lemma}
\begin{proof}
Indeed,
\[
\begin{split}
  \sum_{j \in \cI: |W_j| \le \delta}p_j \le \sum_{j \in \cI: |W_j| \le \delta}p_j \frac{\delta}{|W_j|} \le \delta \cZ(\cG).
\end{split}
\]
The second assertion can be deduced from the first and \(\sum_{j \in \cI} p_j= 1\).
\end{proof}
We introduce a number \(\delta_{*} \in \bR^+\) that will be useful in the future. Set
\begin{equation}\label{eq:deltastar}
\delta_{*} = \frac{1}{100 Z_{*}}.
\end{equation}
In view of Lemma \ref{lem:growth-lemma} and Lemma \ref{lem:proportion-of-good-sp}, \(\delta_{*}\) represents the minimal length of segments shared by most mass in a standard family (to be precise, not less than \(99\%\) of the mass) after the time required by the growth Lemma. Roughly speaking, we can view \(\delta_{*}\) as an intrinsic length scale given by the discontinuous dynamics at equilibrium (uniform in \(t\)). We now define coupling of standard families. Intuitively, two standard families are \((p,\eta)\)-coupled if at least a proportion \(p\) of their mass is at distance \(\eta\) along the stable direction.
\begin{definition}[\((p,\eta)\)-Coupling] \label{def:(p,eta)-coupling}
We say that \(W_1, W_2 \in \cW\) are \(\eta\)-\textit{coupled}, for \(\eta \in \bR^+\), if 
\[
\begin{split}
     &\forall x \in W_1, \quad \{x + \widetilde \eta v\}_{\widetilde \eta \in [-\eta, \eta]} \cap W_2 \neq \emptyset, \quad v \in E^s_t, \text{ }\|v\| =1,\\
    &\hspace{4cm}\text{and}\\
    &\forall x \in W_2, \quad \{x + \widetilde \eta v\}_{\widetilde \eta \in [-\eta, \eta]} \cap W_1 \neq \emptyset, \quad v \in E^s_t, \text{ }\|v\| =1.
\end{split}
\]
Given two standard families \(\cG_1 = \{W_{1,j}, p_{1,j}\}_{j \in \cI_1}\) and \(\cG_2 = \{W_{2,j}, p_{2,j}\}_{j \in \cI_2}\), we say that \(\cG_1\) and \(\cG_2\) are \((p, \eta)\)-coupled, \(p \in (0,1]\), \(\eta \in \bR^+\), if there exists, \(\cI_{1,c} \subseteq \cI_1\), \(\cI_{2,c} \subseteq \cI_2\), standard curves \(W_{1,j,c} \subseteq W_{1,j}\), \(j \in  \cI_{1,c}\), and \( W_{2,j,c} \subseteq W_{2,j}\), \(j \in  \cI_{2,c}\), and a bijection \(\Upsilon :  \cI_{1,c} \to  \cI_{2,c}\), such that 
\[
\begin{split}
    &W_{1,j,c} \text{ and } W_{2, \Upsilon(j),c}  \text{ are } \eta-\text{coupled};\\
  &\sum_{j \in  \cI_{1,c}} \min\left\{p_{1,j} \frac{|W_{1,j,c}|}{|W_{1,j}|}, p_{2,\Upsilon(j)} \frac{|W_{2,\Upsilon(j),c}|}{|W_{2,\Upsilon(j)}|}\right\} = p.
\end{split}
\]
 If \(\cG_1\) and \(\cG_2\) are \((1,\eta)\)-coupled, we say that they are \(\eta\)-coupled.
\end{definition}
Whenever we have two \((p,\eta)\)-coupled standard families, we can always reorder the indices to have \(\Upsilon(j) = j\), i.e., \(\cI_{1,c} = \cI_{2,c}\). For two \((p,\eta)\)-coupled standard family it is possible to define their coupled and uncoupled parts and renormalize the weights in a way that these are again standard families.
\begin{definition}[Decomposition] \label{def:decomposition-coupled-sf}
Let \(\cG_{\iota}\), \(\iota \in \{1,2\}\), be two \((p,\eta)\)-coupled standard families and \(\{W_{\iota,j, c}\}\), as above. For \(j \in \cI_\iota\), let \(\cD_{\iota,j}\) be some finite set indexing the connected components of \(W_{\iota,j} \setminus W_{\iota,j, c}\) and denote them by \(\overline W_{\iota,j, p}\), \(p \in \cD_{\iota,j}\) (if \(j \in \cI_\iota \setminus \cI_{\iota,c}\) we set \(W_{\iota,j, c} = \emptyset\)). We define,
\[
\begin{split}
&\underline{p}_{j} = \min\left\{p_{1,j} \frac{|W_{1,j,c}|}{|W_{1,j}|}, p_{2,j} \frac{|W_{2,j,c}|}{|W_{2,j}|}\right\}, \quad \forall j \in \cI_{1,c} = \cI_{2,c},  \quad \text{`the mass that} \\
&\text{we couple in the \(j^{th}\) element of the family';}\\\\
& p_{*,j} = \frac{\underline{p}_j}{\sum_{k \in  \cI_{1,c}} \underline p_k } =  \frac{\underline{p}_j}{p}, \quad \text{ `as above but normalized by the total coupled mass';}\\\\
& \overline p_{\iota,j} = p_{\iota,j} \frac{|W_{\iota,j,c}|}{|W_{\iota,j}|} - \underline{p}_j, \quad \forall j \in \cI_{\iota,c}, \quad \text{`mass that we do not couple in the \(j^{th}\)}\\
&\text{element of the family because of the difference between the two densities';}\\\\
& \overline p_{\iota,j,p} = p_{\iota,j}\frac{|\overline W_{\iota,j,p}|}{|W_{\iota,j}|}, \quad \forall p \in  \cD_{\iota,j}, \text{ }j \in \cI_{\iota}, \quad \text{`other mass that we do not couple}\\
&\text{in the \(j^{th}\) element of the family because it is supported on uncoupled segments';}\\\\
\end{split}
\]
\[
\begin{split}
&p_{\square, \iota, j} = \frac{\overline p_{\iota,j}}{1-p}, \quad \text{`non coupled mass normalized by all non coupled mass';}\\\\
& p_{\square, \iota, j, p} = \frac{\overline p_{\iota,j,p}}{1-p}\quad \text{`other non coupled mass normalized by all non coupled mass'}.\\
&
\end{split}
\]
Set \(\cD_\iota = \{(j,p) : j \in \cI_{\iota}, p \in \cD_{\iota,j}\}\). We define the coupled and uncoupled standard families as
\[
\begin{split}
    &\Coup_{p,\eta}(\cG_1, \cG_2) = \left\{(W_{1,j,c}, p_{*,j})\right\}_{j \in  \cI_{1,c}},\\
    &\Coup_{p,\eta}(\cG_2, \cG_1) = \left\{(W_{2,j,c}, p_{*,j})\right\}_{j \in  \cI_{2,c}},\\
    &\cG_1 \setminus \Coup_{p,\eta}(\cG_1, \cG_2) = \left\{(W_{1,j,c}, p_{\square, 1, j})\right\}_{j \in \cI_{1,c}} \cup \left\{(\overline W_{1,j,p}, p_{\square,1,j,p})\right\}_{(j,p) \in \cD_1},\\
    &\cG_2 \setminus \Coup_{p,\eta}(\cG_2, \cG_1) = \left\{(W_{2,j,c}, p_{\square, 2, j})\right\}_{j  \in \cI_{2,c}} \cup \left\{(\overline W_{2,j,p}, p_{\square,2,j,p})\right\}_{(j,p) \in \cD_2}.
\end{split}
\]
\end{definition}
Notice that \(\Coup_{p,\eta}(\cG_1, \cG_2)\) and \(\Coup_{p,\eta}(\cG_2, \cG_1)\) are \(\eta\)-coupled and all the above are standard families. We also observe that for each \(W_{\iota,j}\), \(\iota \in \{1,2\}\), we couple at most \textit{a single connected} segment \( W_{\iota ,j,c}\), so that
\begin{equation}\label{eq:number-connected-segments-coupling}
    \# \cD_{\iota, j} \le 2, \quad \forall j \in \cI_\iota.
\end{equation}
Since we will often deal with two standard families, we introduce the following notation for the `other' family. For \(\iota \in \{1,2\}\), we set
\[
    \iota^{*} = \iota +1 \mod 2.
\]
We will sometimes omit one or both subscripts in the \(\Coup\) notation, if clear from the context. Notice that if two standard families \(\cG_1\) and \(\cG_2\) are \((p,\eta)\)-coupled, there could be multiple choices for the standard families \(\Coup(\cG_{\iota}, \cG_{\iota*})\), but the difference is inessential for the following. The whole point of the decomposition above is that, for any \(\vf \in \cC^0(\cM)\),
  \begin{equation}\label{eq:basic-deco}
      \mu_{\cG_\iota}(\vf) = p \mu_{\Coup(\cG_\iota, \cG_{\iota^*})}(\vf) + (1-p)\mu_{\cG_\iota \setminus \Coup(\cG_\iota, \cG_{\iota^*})}(\vf).
  \end{equation}

The problem of coupling two standard families is that we are often able to couple only pieces of them. This fragmentation yields a loss of regularity of the resulting coupled and uncoupled families that we investigate in the next Lemma. 
\begin{lemma}\label{lem:Z-coupled}
  Let \(\cG_1\) and \(\cG_2\) two \((p,\eta)\)-coupled standard families for some \(p \in (0,1)\), \(\eta \in \bR^+\). Then, for \(\iota \in \{1,2\}\), 
  \[
  \begin{split}
      &\cZ\left(\Coup(\cG_\iota,\cG_{\iota^*})\right) \le \frac{\cZ(\cG_\iota)}{p}, \quad \cZ \left(\cG_\iota \setminus \Coup(\cG_\iota,\cG_{\iota^*})\right) \le \frac{3\cZ(\cG_\iota)}{1-p}. 
  \end{split}
  \]
\end{lemma}
\begin{proof}
  Setting \(\iota = 1\) for notational convenience and according to Definition \ref{def:decomposition-coupled-sf}, 
  \[
  \begin{split}
   \cZ\left(\Coup(\cG_1,\cG_2)\right) &= \sum_{j \in  \cI_{1,c}}\frac{ p_{*,j}}{| W_{1,j,c}|} =   \sum_{j \in \cI_{1,c}} \frac{\min\left\{p_{1,j} \frac{|W_{1,j,c}|}{|W_{1,j}|}, p_{2,j} \frac{|W_{2,j,c}|}{|W_{2,j}|}\right\}}{p|W_{1,j,c}|}\\
& \le \frac{1}{p} \sum_{j \in \cI_{1,c}}\frac{p_{1,j}}{|W_{1,j}|} \le \frac{1}{p} \sum_{j \in \cI_{1}}\frac{p_{1,j}}{|W_{1,j}|}\le \frac{\cZ(\cG_1)}{p}.
  \end{split}
  \]
For the second inequality, we have
\[
  \cZ \left(\cG_1 \setminus \Coup(\cG_1,\cG_2)\right) = \sum_{j \in \cI_{1,c}} p_{\square,1,j}\frac{1}{| W_{1,j,c}|} + \sum_{(j,p) \in \cD_1} p_{\square, 1, j, p} \frac{1}{|\overline W_{1,j,p}|},
\]
and we estimate each piece separately. For the first term,
\begin{equation}\label{eq:zeta-first}
\begin{split}
    \sum_{j \in \cI_{1,c}} &\frac{\left(p_{1,j}\frac{| W_{1,j,c}|}{|W_{1,j}|} - \underline p_{j} \right)}{ 1-p}\frac{1}{|W_{1,j,c}|}\le \frac{\sum_{j \in \cI_{1,c}} p_{1,j}\frac{1}{|W_{1,j}|}}{1-p} \le \frac{\cZ(\cG_1)}{1-p}.
\end{split}
\end{equation}
As for the second part, by \eqref{eq:number-connected-segments-coupling},
\begin{equation}\label{eq:zeta-second}
\begin{split}
   \sum_{(j,p) \in \cD_1} p_{1,j}\frac{|\overline W_{1,j,p}|}{|W_{1,j}|}&\frac{1}{(1-p)|\overline W_{1,j,p}|} \\
&= \sum_{j \in \cI_1}p_{1,j}\frac{ \# \cD_{1,j} }{(1-p)|W_{1,j}|}\le \frac{2}{1-p} \cZ(\cG_1).
\end{split}
\end{equation}
The second statement then follows by adding \eqref{eq:zeta-first} and \eqref{eq:zeta-second}.
\end{proof}

We conclude this section with some simple observations that follow directly from Definition \ref{def:decomposition-coupled-sf} and will be useful later. Let \(\cG_1\) and \(\cG_2\) be \((p,\eta)\)-coupled:
 \begin{enumerate}
  \item [{\crtcrossreflabel{(O1)}[obs:p-eta]}] \(\cG_1\) and \(\cG_2\) are \((p',\eta')\)-coupled for any \(\eta' \ge \eta\) and \(p' \le p\). For the second part, one defines coupling with smaller segments which are retractions of the original ones.
\end{enumerate}
By considering smaller segments as in \ref{obs:p-eta}, the regularity of the coupled families decrease (i.e., \(\cZ\) increases). There is a somehow better way to couple less mass.
\begin{enumerate}
  \item  [{\crtcrossreflabel{(O2)}[obs:pless]}]
For any \(p' \in (0,p]\), there exist \((p',\eta)\)-coupled standard families \(\widetilde \cG_\iota \simeq \cG_\iota\) such that \(\cZ (\Coup_{p', \eta}(\widetilde \cG_\iota, \widetilde \cG_{\iota^*})) = \cZ (\Coup_{p, \eta}(\cG_\iota, \cG_{\iota^{*}}))\). This follows from the identity, for \(\iota \in \{1,2\}\),
\[
\begin{split}
\hspace{1cm}\mu_{\cG_{\iota}} \!= & p'\!\mu_{\Coup_{p,\eta}(\cG_\iota, \cG_{\iota^*})} \! +\! (p\!-\!p')\mu_{\Coup_{p,\eta}(\cG_\iota, \cG_{\iota^*})} \!+\!(1\!-\!p)\mu_{ \cG_{\iota} \setminus \Coup_{p,\eta}(\cG_{\iota}, \cG_{\iota^*})}.
\end{split}
\]
Then, one modifies the weights as in \eqref{eq:indistinguishable} and considers only the mass in the first addend as coupled. The regularity of \(\widetilde \cG_{\iota} \setminus \Coup_{p',\eta}(\widetilde \cG_{\iota}, \widetilde \cG_{\iota^*})\) is a weighted average of \(\cZ (\Coup_{p, \eta}(\cG_\iota, \cG_{\iota^*}))\) and \(\cZ(\cG_{\iota} \setminus \Coup_{p,\eta}(\cG_{\iota}, \cG_{\iota^*}))\) that can be estimated by Lemma \ref{lem:Z-coupled}.
  \item  [{\crtcrossreflabel{(O3)}[item:dist]}] For any \(\vf \in \cC^1(\cM)\), by the decompositions \eqref{eq:basic-deco}, 
  \[
\begin{split}
    \bigl|\mu_{\cG_1}(\vf) - &\mu_{\cG_2}(\vf)\bigr| \le p\left|\mu_{\Coup(\cG_1, \cG_2)}(\vf) \!- \!\mu_{\Coup(\cG_2, \cG_1)}(\vf)\right| \\
&+ (1-p)\bigl|\mu_{\cG_1 \setminus \Coup(\cG_1, \cG_2)}(\vf)\bigr| + (1-p)\bigl|\mu_{\cG_2 \setminus \Coup(\cG_2, \cG_1)}(\vf)\bigr| \\
&\le p\eta \|\vf'\|_{\cC^0} + 2(1-p)\|\vf\|_{\cC^0}. 
\end{split}
\]
 \end{enumerate}

Starting from here, we address the essence of the coupling arguments.

\subsection{Coupling of two standard segments}\label{subsec:coup-segments} \textit{In this subsection we prove that it is possible to couple any two standard segments after a big enough iterate of the perturbed map. The core of this subsection is Lemma \ref{lem:first-coupling}}. We introduce partitions of \(\cM\) useful for coupling. For \(\eta \in \bR^+\), consider a partition of \(\cM\) into squares whose sides are parallel to \(E_u^t\) and \(E_s^t\) (note that \(A_t\) is symmetric) and have length \(\eta\). It is not possible to construct a similar partition near \(\partial \cM\). There, we consider polygons \(Q' = Q \cap \cM\), where \(Q\) is a square as above. Call any such partition \(\cQ_\eta\) and let \(\widehat \cQ_\eta \subseteq \cQ_\eta\) be the elements of the partition that are at least \(2\eta\) distant from \(\partial \cM\) (and so are squares). Notice that \(\lim_{\eta \to 0}\#\widehat \cQ_\eta /\#\cQ_\eta = 1\). We call size of a square in \(\widehat\cQ_\eta\) the size of its sides. We will couple mass in these squares. The next Lemma is the consequence of equidistribution of standard families for the Arnold cat map tailored to what we need. Recall that \(\cW_\delta\) are segments in the unstable cone of length at least \(\delta\).
\begin{lemma}\label{lem:unp-density}
  For any \(\eta\) small enough and  \(\delta \in \bR^+\), there exists \(N_{\eta,\delta}\in \bN\) such that for any \(W \in \cW_{\delta}\) and \(k \ge N_{\eta, \delta}\),
\[
   \sum_{Q \in \widehat \cQ_{\eta}} \left| \frac{1}{\# \widehat \cQ_{\eta}} - \frac{\left|Q \cap \cato^k(W)\right|}{\left|\cato^k(W) \right|}\right| \le \frac{1}{100}.
\]
\end{lemma}
\begin{proof}
Let \(\eta_0 \in \bR^+\) small enough such that for all \(\eta \in (0, \eta_0)\),
\begin{equation}\label{eq:smalleta}
\left|\eta^2 - \frac{1}{\# \widehat \cQ_{\eta}} \right| = \frac{m \left(\bigcup_{Q \in \cQ_{\eta} \setminus \widehat \cQ_{\eta}} Q\right)}{\#\widehat \cQ_{\eta}} \le \eta^2\frac{\# (\cQ_{\eta} \setminus \widehat \cQ_{\eta})}{\#\widehat \cQ_{\eta}} \le \frac{\eta^2}{200},
\end{equation}
where in the first inequality we have used that \(m(Q) \le \eta^2\) for any \(Q \in \cQ_{\eta}\) and in the second the limit for \(\eta\) tending to zero. Fix \(\eta \in (0, \eta_0)\) and \(\delta \in \bR^+\). By equidistribution of standard pairs for the Arnold cat map,\footnote{Equation \eqref{eq:equidistribution-unp} follows e.g. by the results in \cite{DKL21} chapter 3, by approximating any \(Q \in \widehat{Q}_{\eta}\) with \(\cC^1\) functions.} there exists \(N_{\eta, \delta} \in \bN\), such that, for any \(k \ge N_{\eta,\delta}\), \(Q \in \widehat \cQ_{\eta}\),
\begin{equation}\label{eq:equidistribution-unp}
\left|\frac{\left|Q \cap \cato^k (W)\right|}{\left|\cato^k (W)\right|} - \eta^2 \right| = \left|\frac{\left|Q \cap \cato^k (W)\right|}{\left|\cato^k (W)\right|} - m(Q) \right| \le \frac{\eta^2}{200}, \quad  \forall W \in \cW_{\delta}.
\end{equation}
Hence, by \eqref{eq:smalleta} and \eqref{eq:equidistribution-unp}, since \(\eta^2 \#\widehat \cQ_{\eta}\le 1\), for any \(W \in \cW_{\delta}\) and \(k \ge N_{\eta, \delta}\),
\[ 
\begin{split}
    \sum_{Q \in \widehat \cQ_{\eta}} \biggl| \frac{1}{\# \widehat \cQ_{\eta}} - &\frac{\left|Q \cap \cato^k(W)\right|}{\left|\cato^k (W) \right|}\biggr| \le \sum_{Q \in \widehat \cQ_{\eta}} \biggl| \frac{1}{\# \widehat \cQ_{\eta}} - \eta^2\biggr|\\
    &+ \sum_{Q \in \widehat \cQ_{\eta}} \biggl|\eta^2 - \frac{\left|Q \cap \cato^k(W)\right|}{\left|\cato^k (W) \right|}\biggr|\le \frac{2}{200}  \sum_{Q \in \widehat \cQ_{\eta}} \eta^2 \le \frac{1}{100}.
\end{split}
\]
\end{proof}
For \(W \in \cW\) , we denote by \(\Center(W)\) the set of points in \(W\) whose distance from \(\partial W\) is bigger than \(|W|/100\). In the next purely geometric Lemma we argue that if two long standard segments intersect a small square in their centers they can be coupled. Recall that \(\cG_{W} = \{(W,1)\}\).
\begin{lemma}\label{lem:int-rect}
  Let \(W_1, W_2 \in \cW\) and set 
\[
\zeta = \min\{|W_1|, |W_2|\}, \quad \alpha_{*} = \max \left\{\angle(W_1, E_t^u), \angle(W_2, E_t^u)\right\}.
\] 
Let \(\rho \in \left(0, \zeta/1000\right)\), \(Q \in \widehat Q_{\rho}\) and assume that \(\alpha_{*} \le \pi/8\). If \(\Center W_{\iota} \cap Q \neq \emptyset\), \(\iota \in \{1,2\}\), then there exist standard curves \(W_{\iota,c} \subseteq W_\iota\) such that \(\min\{|W_{1,c}|, |W_{2,c}|\} \ge \zeta/200\) and \(\cG_{W_{1,c}}\), \(\cG_{W_{2,c}}\) are \(\left(3\rho + \zeta \sin 2 \alpha_{\star}\right)\)-coupled.
\end{lemma}
\begin{proof}
Since \(\alpha_{*} \le \pi/8\), both \(W_1\) and \(W_2\) intersect a square \(Q_{3\rho}\) of size \(3\rho\), centered in the original square \(Q\), in both sides parallel to \(E^s_t\). Since \(\Center W_{\iota} \cap Q \neq \emptyset\), it is possible to consider \(W_{\iota,c} \subseteq W_{\iota}\) such that \(\zeta/200 \le |W_{\iota,c}| \le \zeta\) and \(W_{\iota,c}\) have an endpoint on the right side of \(Q_{3\rho}\), lie all at the right of such side and are \(\eta\)-coupled for some \(\eta\) that we now specify. Since both \(W_{\iota,c}\) intersect a square of size \(3\rho\), the distance along \(E^s_t\) between \(W_{1,c}\) and \(W_{2,c}\) is less than, 
\[
3 \rho + \max\{|W_{1,c}|, |W_{2,c}|\} \sin  \angle (W_1, W_2) \le 3 \rho + \zeta   \sin  \angle (W_1, W_2).
\]
The fact that the angle between \(W_1\) and \(W_2\) is less than \(2\alpha_{*}\) concludes the proof.
\end{proof}
We now show that for small \(t\), it is possible to couple a fixed amount of mass in a controlled time for the perturbed maps \(\catt\). Recall \(\widetilde Z\) from \eqref{eq:regular-sf-def}, \(\delta_{*}\) from \eqref{eq:deltastar} and that \(\cG_1 \simeq \cG_2\) is introduced after equation \eqref{eq:indistinguishable}.
\begin{lemma}\label{lem:first-coupling}
   For any \(\eta \in \bR^+\) and \(\delta \in \bR^+\), there exists \(N_{\eta, \delta}\in \bN\) such that for any \(k \ge N_{\eta,\delta}\) there exists \(\ve_0 \in (0, \frac{1}{8}]\) with the following property: For any \(t \in [0, \ve_0]\) and \(W_{\iota} \in \cW_{\delta}\), \(\iota \in \{1,2\}\), there exist standard families \(\cG_{c,k,\iota} \simeq \catt^{k}\cG_{W_\iota}\) that are
\[
 \left(\frac{\delta_{*}}{400 \overline \lambda^k}, \eta\right)-\text{coupled}.
\]
Moreover, \( \cZ\bigl(\Coup(\cG_{c,k,\iota}, \cG_{c,k,\iota^{*}})\bigr) \le \widetilde Z\) (i.e., the coupled family is regular).
\end{lemma}
\begin{proof}
Let \(\eta_0 \in \bR^+\) be such that for any \(\eta \le \eta_0\) the statement of Lemma \ref{lem:unp-density} holds. Without loss of generality, we consider only \(\eta\) small enough, while for general \(\eta\) the statement is true because of \ref{obs:p-eta}. Choose any \(\eta, \delta \in \bR^+\),
\begin{equation}\label{eq:fix-eta-small-first-enc}
0 < \eta \le \min\left\{\eta_0, \frac{\delta_{*}}{2000}\right\}.
\end{equation}
We now fix \(N_{\eta, \delta}\), \(k\) and \(\ve_0\) in this order. Since \(W_{\iota} \in \cW_{\delta}\), we have that \(\cZ(\cG_{W_\iota}) \le \delta^{-1}\). Hence, by the third statement of the Growth Lemma \ref{lem:growth-lemma}, there exists \(N_1 \in \bN\) such that for all \(p \ge N_1\), \(W_{\iota} \in \cW_{\delta}\),
  \begin{equation}\label{eq:time-long-pieces}
    \max\left\{\cZ(\catt^{p}\cG_{W_1}), \cZ(\catt^{p}\cG_{W_2}) \right\} \le Z_{*}.
  \end{equation} 
  Let also \(N_2 \in \bN\) be given by Lemma \ref{lem:unp-density} for the \(\eta\) and \(\delta\) given by the statement of the Lemma. Finally, let \(N_3 \in \bN\) be big enough such that, for all \(p \ge N_3\),
  \begin{equation}\label{eq:K3-angle}
      \delta_{*}\sin\left (2\angle(\catt^{p}W_\iota, E_t^u) \right) \le \eta, \text{ } \text{ } \angle(\catt^{p}W_\iota, E_t^u) \le \frac{\pi}{8}, \quad  \iota \in\{1,2\}.
  \end{equation}
 Set \(N_{\eta, \delta} = \max\{N_1, N_2, N_3\}\). Fix \(k \ge N_{\eta, \delta}\). Recall that, by Lemma \ref{lem:n-map-sing}, there exist \(C_k, Q_k \in \bR^+\) such that 
  \begin{equation}\label{eq:distance-pert-unpert}
    \sup_{\cM \setminus [\cS_0^{+,k} ]_{Q_k t}} \left|\catt^k - \cato^k\right|\le C_k t.
  \end{equation}
Let also \(P_k\) be given by Lemma \ref{lem:trasversality2}. We fix \(\ve_0\) small enough such that 
\begin{equation}\label{eq:fixt_0}
  C_k \ve_0 \le \frac{\eta}{100},  \quad \left(\frac{K_k}{\delta} + P_k\right)Q_{k}C_{cone}\ve_0 \le \frac{1}{100},
\end{equation}
where \(C_{cone} \in \bR^+\) depends only on the transversality between \(\mathscr C^u\) and \(\mathscr C^s\) and will be introduced in \eqref{eq:good-prop-2} and \(K_k\) is the complexity. For \(t \in [0,\ve_0]\), \(\iota \in \{ 1,2\}\), define
\[
\begin{split}
&N_t (\iota) = \biggl\{x \in W_\iota: \catt^k (x) \text{ belong to the Center of a connected}\\
&   \hspace{3cm}\text{component of }\catt^k (W_\iota) \text{ not shorter than }\delta_{*}\biggr\},\\
& G_t (\iota) =\left( \cM \setminus [\cS_0^{+,k} ]_{Q_k t} \right) \cap  N_t (\iota), \\
&\widehat\cQ_{\eta, t}^{good} (\iota) = \left\{Q \in \widehat \cQ_{\eta}: Q \cap \cato^k \left(G_t (\iota)\right)\neq \emptyset\right\}.
\end{split}
\]

To help the reader, before proceeding with the proof, we discuss briefly the definitions just introduced. The set \(G_t (\iota) \) are points in \(W_{\iota}\) that after \(k\) iterations of the perturbed map will belong to centers of long connected components of \(\catt^k(W_{\iota})\) and stay sufficiently far from the singularity set \(\cS_{0}^{+,k}\). These choices are made for two reasons. The first is that to couple mass we show that a component of both the two families \(\catt^k \cG_{W_{\iota}}\) intersect the same square. However, in order to apply Lemma \ref{lem:int-rect}, we need the components to be long enough. The second reason is that we use mixing properties of the unperturbed dynamics to show that the two families intersect the same square simultaneously. However, we need to be sufficiently far from the singularities to compare perturbed and unperturbed dynamics (see equation \eqref{eq:distance-pert-unpert}). The squares where we are able to couple mass are (enlargements of) those in \(\widehat\cQ_{\eta, t}^{good}\). Going back to the proof, we make the following:\\

\textit{Claim:} for each \(t \in [0,\ve_0]\),\\
\begin{equation}\label{eq:claim}
    \widehat\cQ_{\eta, t}^{good} (1) \cap \widehat\cQ_{\eta, t}^{good} (2) \neq \emptyset.
\end{equation}
For convenience, set \(\catt^k\cG_{W_\iota} = \left\{(W_j^\iota, p_{j}^\iota)\right\}_{j \in \cI_\iota}\). Notice that, for \(\iota \in \{ 1,2\}\),
\[
\begin{split}
  \frac{|N_t (\iota)|}{|W_\iota|} &= \frac{|\catt^k (N_{t}(\iota))|}{|\catt^k (W_{\iota})|} = \sum_{j \in \cI_\iota : |W_j^\iota| \ge \delta_{*}} \frac{|\Center W_{j}^\iota|}{|\catt^k (W_{\iota})|}  \\
&= \sum_{j \in \cI_\iota : |W_j^\iota| \ge \delta_{*}} \frac{|W_j^{\iota}|}{|\catt^k (W_{\iota})|}\frac{|\Center W_{j}^\iota|}{|W_j^{\iota}|} =  \sum_{j \in \cI_\iota : |W_j^\iota| \ge \delta_{*}}p_j^\iota \frac{\left|\Center W_{j}^\iota \right|}{|W_{j}^\iota|},
\end{split}
\]
where in the first equality we have used that \(\catt\) is linear and injective, in the second one the definition of \(N_{t}(\iota)\) and in the last one Lemma \ref{lem:weigths-single-curve}. Since the proportion of points in the centers is \(49/50\) of the total, for all \(t \in [0,\ve_0]\),
\begin{equation}\label{eq:good-prop-1}
 \frac{|N_t (\iota)|}{|W_\iota|} \ge \frac{49}{50}  \sum_{j \in \cI_\iota : |W_j^\iota| \ge \delta_{*}}p_j^\iota \ge \frac{49}{50}\left(1- \delta_{*}Z_{*}\right) \ge  \frac{49}{50} \frac{99}{100} \ge \frac {95}{100},
\end{equation}
where in the second inequality we have used Lemma \ref{lem:proportion-of-good-sp} and equation \eqref{eq:time-long-pieces} and in the third the choice of \(\delta_{*}\) in \eqref{eq:deltastar}. Moreover, by transversality between \(\cS_0^{+,k}\) and \(\mathscr C^u\), for some \(C_{cone} \in \bR^+\) depending only on \(\mathscr C^s\) and \(\mathscr C^u\),
\begin{equation}\label{eq:good-prop-2}
\begin{split}
    \frac{\left|W_\iota \cap \left( \cM \setminus  [\cS_0^{+,k} ]_{Q_k t}\right)\right|}{|W_\iota|} &=  1 -  \frac{\left|W_\iota \cap [\cS_0^{+,k} ]_{Q_k t}\right|}{|W_\iota|}\\
& \ge 1- C_{cone} \frac{ Q_k t \# \{W_\iota \cap \cS_0^{+,k}\}}{|W_\iota|}.
\end{split}
\end{equation}
By Lemma \ref{lem:trasversality2}, the right term in \eqref{eq:good-prop-2} is not smaller than
\[
     1- C_{cone} \frac{ Q_k t \left(K_k + P_k|W_\iota|\right)}{|W_\iota|}.
\]
Hence, using that \(|W_\iota| \ge \delta\) and the second of \eqref{eq:fixt_0} we have, for all \(t \in [0,\ve_0]\),
\begin{equation}\label{eq:good-prop-2-fin}
\begin{split}
\frac{\left|W_\iota \cap \left( \cM \setminus  [\cS_0^{+,k} ]_{Q_k t}\right)\right|}{|W_\iota|} \ge 1 - \left(\frac{K_k}{\delta} + P_k \right)Q_k C_{cone} \ve_0 \ge \frac{99}{100}.
\end{split}
\end{equation}
Combining \eqref{eq:good-prop-1} and \eqref{eq:good-prop-2-fin}, we have that, for all \(t \in [0,\ve_0]\),
\[
\begin{split}
  \frac{\left|G_t(\iota)\right|}{|W_\iota|} &\ge \frac{\left|W_\iota \cap \left( \cM \setminus [\cS_0^{+,k} ]_{Q_k t} \right) \right| + \left|N_t (\iota) \right| - |W_\iota|}{|W_\iota|} \\
&\ge \frac{99+95}{100}-1\ge \frac{9}{10}.
\end{split}
\]
Therefore, using the linearity and injectivity of \(\cato\), for \(t \in [0,\ve_0]\), \(\iota \in \{1,2\}\),
\begin{equation}\label{eq:good-proportion-final.1}
  \frac{\left|\cato^k (G_t(\iota))\right|}{|\cato^k (W_\iota)|} = \frac{\left|G_t(\iota)\right|}{|W_\iota|}  \ge \frac{9}{10}, \quad \text{i.e.} \quad  \frac{\left|\cato^k (W_\iota) \setminus \cato^k (G_t(\iota))\right|}{|\cato^k (W_\iota)|} \le \frac{1}{10}.
\end{equation}
Recall that \(k \ge N_{\eta,\delta} \ge N_2\) and so we can apply Lemma \ref{lem:unp-density}. Hence, for all \(t \in [0,\ve_0]\), \(\iota \in \{1,2\}\), it holds
\[
\begin{split}
&\frac{\# \left(\widehat\cQ_{\eta} \setminus  \widehat\cQ_{\eta,t}^{good} (\iota)\right) }{\# \widehat\cQ_{\eta}} = \sum_{Q \in \widehat Q_{\eta}: Q \cap \cato^k \left(G_t(\iota)\right) = \emptyset} \frac{1}{\# \widehat\cQ_{\eta}} \\
&\le  \sum_{Q \in \widehat Q_{\eta}} \left| \frac{1}{\# \widehat\cQ_{\eta}} - \frac{\left|Q \cap \cato^k(W_\iota)\right|}{\left|\cato^k (W_\iota)\right|}\right| +  \sum_{Q \in \widehat Q_{\eta}: Q \cap \cato^k \left(G_t(\iota)\right) = \emptyset} \frac{\left|Q \cap \cato^k(W_\iota)\right|}{\left|\cato^k (W_\iota)\right|} \\
&\le \frac{1}{100} + \frac{\left|\cato^k(W_\iota) \setminus \cato^k \left(G_t^k(\iota)\right)\right|}{\left| \cato^k(W_\iota) \right|} \le \frac{1}{100} + \frac{1}{10} \le \frac{1}{5},
\end{split}
\]
where in the last inequality we have used \eqref{eq:good-proportion-final.1}. In particular the equation above proves the claim \eqref{eq:claim}, since the proportion of good squares for both segments is more than \(1/2\). 

We now conclude the proof of the Lemma. For any \(t \in [0,\ve_0]\), let \(Q_c \in \widehat \cQ_{\eta,t}^{good}(1) \cap \widehat \cQ_{\eta,t}^{good}(2)\) and let \(Q_{c,2}\) be a square centered in \(Q_c\) with double size (such a square exists because \(Q \in \widehat\cQ_{\eta}\) and is at least \(2\eta\) distant from \(\partial \cM\)). By definition of good squares, for each \(\iota \in \{1,2\}\), there exists \(x_{\iota} \in G_t(\iota)\) such that \(\cato^k(x_{\iota}) \in Q_c\). However, by \eqref{eq:distance-pert-unpert} and the first of \eqref{eq:fixt_0}, recalling that \(G_t(\iota)  \subseteq \cM\setminus [\cS_0^{+,k}]_{Q_k t}\),
\[
 d\left(\catt^k(x_{\iota}), \cato^k(x_{\iota}) \right) \le C_k t \le  C_k \ve_0 \le \frac{\eta}{100}.
\]
Hence, since \(Q_{c,2}\) has size \(2\eta\), 
\begin{equation}\label{eq:coupling-square-found}
  \catt^k(x_{\iota}) \in Q_{c,2}, \quad \iota \in \{1,2\}.
\end{equation}
Since \(x_{\iota} \in  G_t(\iota) \subseteq N_{t}(\iota) \), by definition of \(N_t(\iota)\), one has that \(\catt^k(x_{\iota})\) belongs to the center of connected component not smaller than \(\delta_{*}\) of \(\catt^k \left( W_\iota \right)\). Therefore, equation \eqref{eq:coupling-square-found} implies that for both \(\iota \in \{1,2\}\) long components of \(\catt^k \left( W_\iota \right)\) intersect the same square \(Q_{c,2}\). By \eqref{eq:fix-eta-small-first-enc} the size of \(Q_{c,2}\) is \(2\eta \le \delta_{*}/1000\) and considering also the second of \eqref{eq:K3-angle}, we can apply Lemma \ref{lem:int-rect}. By the mentioned Lemma, there exist coupled standard segments \(W_{\iota,c} \subseteq \catt^{k}(W_{\iota}) \) such that 
\begin{equation}\label{eq:lowe-length-seg}
 \min\{|W_{1,c}|, |W_{2,c}| \}\ge \frac{\delta_{*}}{200},
\end{equation}
and the associated standard families \(\cG_{W_{1,c}}\) and \(\cG_{W_{2,c}}\)  are 
\[
6\eta + \delta_{*}\sin\left (2\angle(\catt^{k}W_\iota, E_t^u) \right) \le 7\eta \text{ }\text{ }\text{-coupled},
\]
where we have also used the first of \eqref{eq:K3-angle}. To compute the amount of coupled mass, we denote by \(\overline W_{\iota,c}\) the connected components that contain \(W_{\iota,c}\) and by \(\overline p_{\iota,c}\) the associated weights. Then, by Lemma \ref{lem:weigths-single-curve},
\[
\begin{split}
\min \biggl\{\overline p_{1,c}\frac{|W_{1,c}|}{|\overline W_{1,c}|}, \overline p_{2,c}\frac{|W_{2,c}|}{|\overline W_{2,c}|}\biggr\} &=
\min \biggl\{\frac{|\overline W_{1,c}|}{|\catt^k W_1|}\frac{|W_{1,c}|}{|\overline W_{1,c}|}, \frac{|\overline W_{2,c}|}{|\catt^k W_2|}\frac{|W_{2,c}|}{|\overline W_{2,c}|}\biggr\} \\
&=\min \biggl\{\frac{|W_{1,c}|}{|\catt^k W_1|}, \frac{|W_{2,c}|}{|\catt^k W_2|}\biggr\}.
\end{split}
\]
Therefore, according to Definition \ref{def:(p,eta)-coupling}, the families \(\catt^k\cG_{W_1}\) and \(\catt^k\cG_{ W_2}\) are 
\[
\biggl( \min \biggl\{\frac{|W_{1,c}|}{|\catt^k W_1|}, \frac{|W_{2,c}|}{|\catt^k W_2|},\biggr\} ,7\eta \biggr) \text{ }\text{ }\text{-coupled}.
\]
Since \(|\catt^k W_\iota| \le \overline \lambda^k 2 \) and by \eqref{eq:lowe-length-seg}, the coupled mass is not less than \(\delta_{*}/(400 \overline \lambda^k)\). Moreover, setting \(\Coup (\catt^k\cG_{ W_\iota}, \catt^k\cG_{ W_{\iota^*}}) = (W_{\iota,c}, 1)\), by \eqref{eq:lowe-length-seg} again,
\[
\begin{split}
\cZ\left(\Coup(\catt^k\cG_{ W_\iota}, \catt^k\cG_{ W_{\iota^*}})\right) &\le \max\biggl\{\frac{1}{|W_{1,c}|}, \frac{1}{|W_{2,c}|}\biggr\}\le \frac{200}{\delta_{*}} \le 20000 Z_{*} = \widetilde Z.
\end{split}
\]
Finally, by \ref{obs:pless}, there exist standard families \(\cG_{c,k,\iota} \simeq \catt^k\cG_{W_{\iota}}\) that are precisely \(\left(\frac{\delta_{*}}{400 \overline \lambda^k}, 7\eta \right)\)-coupled and such that
\[
\cZ\bigl(\Coup(\cG_{c,k,\iota}, \cG_{c,k,\iota^{*}})\bigr) = \cZ\left(\Coup (\catt^k\cG_{ W_\iota}, \catt^k\cG_{ W_{\iota^*}})\right) \le \widetilde Z.
\]
By the last statement and by renaming \(\eta\), we conclude the proof of the lemma.
\end{proof}
In other words, it is possible to couple arbitrarily small pieces and have them arbitrarily close. The price for this is waiting a possibly very long time \(k \ge N_{\eta, \delta}\) and couple very little mass \(\sim \overline \lambda^{-k}\). However, the estimates are uniform in \(t \in [0, \ve]\) and the regularity of the coupled standard families is under control (\(\le \widetilde Z\)). We upgrade Lemma \ref{lem:first-coupling} to regular standard families. 

\subsection{Coupling of two standard families}\label{subsec:coupling-sf} \textit{In this subsection we prove that it is possible to couple any two regular standard families after a big enough iterate of the perturbed map.} We would like to apply Lemma \ref{lem:first-coupling} to couple \textit{all} the segments of two standard families. Unfortunately, this is not possible since standard families may have arbitrarily small pieces. However, if the standard families are regular, most of the pieces are not too short. Recall that \(\cG\) is regular if \(\cZ(\cG) \le \widetilde Z\).

\begin{lemma}\label{lem:first-coupl-s-f}
 For any \(\eta \in \bR^+\), there exists \(N_{\eta} \in \bN\) such that, for all \(k \ge N_{\eta}\), there exists \(p_c \in (0,1]\) and \(\ve_0 \in \bR^+\) with the following property: For any \(t \in [0,\ve_0]\), for any two regular standard families \(\cG_\iota\), \(\iota \in \{1,2\}\), there exist two \((p_c, \eta)\)-coupled standard families \(\cG_{c,k,\iota} \simeq \catt^k \cG_\iota\). Moreover, 
 \[
  \max\left\{\cZ\left(\Coup (\cG_{c,k,\iota}, \cG_{c,k,\iota^*})\right) ,  \cZ\bigl(\cG_{c,k,\iota} \setminus  \Coup(\cG_{c,k,\iota}, \cG_{c,k,\iota^*})\bigr) \right\} \le \widetilde Z.
 \]

\end{lemma}
\begin{proof}
Let \(\cG_\iota = \{W_j^\iota, p_j^\iota\}_{j \in \cI_\iota}\). Consider any \(\eta \in \bR^+\) and let \(N_{\eta}\) be given by \(N_{\eta, \delta}\) in Lemma \ref{lem:first-coupling} where \(\eta\) is given by the statement and \(\delta = (2 \widetilde Z)^{-1}\). Fix \(k \ge N_{\eta}\) and let \(\ve_0\) be given by Lemma \ref{lem:first-coupling} for the selected \(k\). By the choice of \(\delta\),
\[
\sum_{\left\{j \in \cI_\iota: |W_{j}^\iota| \ge \delta \right\} }p_{j}^\iota \ge 1 - \delta \cZ(\cG_\iota) \ge 1 - \delta \widetilde Z \ge \frac 1 2,
\]
where we have also used Lemma \ref{lem:proportion-of-good-sp} and that \(\cG_\iota\) are regular. By the equation above, and modifying the weights according to \eqref{eq:indistinguishable}, we may assume without loss of generality that there exist two subsets \(\widetilde \cI_{\iota} \subseteq \left\{j \in \cI_\iota: |W_{j}^\iota| \ge \delta \right\}\), such that
\[
 \sum_{j \in \widetilde \cI_\iota}p^\iota_{j} = \frac 1 2.
\]
To couple, we need to consider yet other two standard families \(\widetilde \cG_{\iota} \simeq \cG_{\iota}\). By the last equation, we have that \(p^{\iota}_{j} =\sum_{l \in \widetilde\cI_{\iota^*}}2p^\iota_{j}p^{\iota^*}_{l}\). Hence, the standard families \(\widetilde \cG_\iota\) defined by replacing each pair \((W^\iota_{j}, p^\iota_j) \), \(j \in \widetilde \cI_{\iota}\), with the pairs \(\{(W^\iota_{j}, 2p^\iota_{j}p^{\iota^*}_{l})\}_{l \in \widetilde \cI_{\iota^*}}\) are indistinguishable to the original ones. The resulting standard families are
\[
\begin{split}
    \widetilde \cG_{1} &= \{(W^{1}_{j}, 2p^1_{j}p^{2}_{l})\}_{(j,l) \in \widetilde \cI_{1} \times \widetilde \cI_{2}} \cup \{(W_j^{1}, p_j)\}_{j \in \cI_{1} \setminus \widetilde\cI_{1}}\\
  \widetilde \cG_{2} &= \{(W^{2}_{l}, 2p^1_{j}p^{2}_{l})\}_{(j,l) \in \widetilde \cI_{1} \times \widetilde \cI_{2}} \cup \{(W_l^{2}, p_l)\}_{l \in \cI_{2} \setminus \widetilde\cI_{2}}.
\end{split}
\]
In particular, they have same weights for the common indices \(\widetilde\cI_{1} \times \widetilde \cI_{2}\) and, by definition, \(\min\{|W^1_{j}|, |W^2_{l}|\} \ge \delta\), for all \((j,l) \in \widetilde\cI_{1}\times \widetilde \cI_2\). Therefore, Lemma \ref{lem:first-coupling} applies, allowing to couple \((W^{1}_{j}, 2p^1_{j}p^{2}_{l}) \in \widetilde \cG_1\) with \((W^{2}_{l}, 2p^1_{j}p^{2}_{l}) \in \widetilde \cG_2\), for any \((j,l) \in \widetilde\cI_{1}\times \widetilde\cI_2\). According to the last sentence, for any \(t \in [0,\ve_0]\), there are standard families \(\cG_{c,k,\iota} \simeq \catt^{k}\widetilde \cG_{\iota} \simeq \catt^{k}\cG_{\iota} \) which are \((p_{c}, \eta)\)-coupled, where 
\[
    p_{c} = \sum_{(j,l) \in \widetilde \cI_{1}\times \widetilde\cI_2}2p^1_{j}p^2_{l} \frac{\delta_{*}}{400 \overline \lambda^{k}} = \frac{\delta_{*}}{800 \overline \lambda^{k}} \in \left(0,\frac 1 2\right).
\]
This proves the first part of the Lemma. The estimate on the regularity of the coupled part follows from the second part of Lemma \ref{lem:first-coupling}. As for the regularity of the uncoupled part, \mbox{by Lemma \ref{lem:Z-coupled},}
\[
\begin{split}
    \cZ\bigl(\cG_{c,k,\iota} \setminus  \Coup_{p_c, \eta} (\cG_{c,k,\iota}, \cG_{c,k,\iota^*})\bigr) &\le \frac{3\cZ (\cG_{c,k,\iota})}{1-p_c} \\
    &= \frac{3\cZ(\catt^{k} \cG_\iota)}{1-p_c} \le  \frac{\widetilde Z}{2(1-p_c)} \le \widetilde Z,
\end{split}
\]
where the last inequality follows from \(p_c \le 1/2\) and in the second inequality we have used the last statement in Lemma \ref{lem:growth-lemma} (possibly choosing \(N_{\eta}\) big enough so that \(\cZ(\catt^{k} \cG_\iota) \le \widetilde Z /6\)). This last estimate concludes the proof of the Lemma.
\end{proof}

\subsection{Iteration of two coupled standard families} \label{subsec:iteration} \textit{In this subsection we consider two \(\eta\)-coupled standard families and we compute how much mass won't be coupled in the evolved standard families because of the discontinuities.} The next Lemma says that the amount of lost mass is proportional to \(\eta\) and so it is less and less as we iterate the two families (because the surviving mass gets closer and closer). Recall that \(N_0\) is the power of the map in the \mbox{Growth Lemma \ref{lem:growth-lemma}.}
\begin{lemma}\label{lem:push-sf}
 There exists \(L \in \bR^+\) such that for all \(\eta \in \bR^+\) small enough and any two \(\cG_{\iota}\), \(\iota \in \{1,2\}\), \(\eta\)-coupled regular standard families we have the following: For every \(k \in \bN\), \(t \in [0,1/8]\), \(\catt^{k N_0}\cG_{\iota}\) are 
\[
(1-L\eta, \lambda^{-kN_0}\eta)-\text{coupled}.
\]
Moreover, we have \(\cZ\bigl(\Coup (\catt^{k N_0}\cG_\iota, \catt^{k N_0} \cG_{\iota^*})\bigr) \le \widetilde Z\).
\end{lemma}
\begin{proof}
Recalling Definition \ref{def:(p,eta)-coupling}, since \(\cG_1\) and \(\cG_2\) are \(\eta\)-coupled, there exist \(\eta\)-coupled segments \(W^{\iota}_j\) and common weights and indices \(p_j >0\), \(j \in \cI\), such that, for \(\iota \in \{1,2\}\), \(\cG_{\iota} = \{W^{\iota}_j, p_j\}_{j \in \cI}\). By Lemma \ref{lem:proportion-of-good-sp} and since \(\cG_{\iota}\) are regular, 
\[
\sum_{\{j \in \cI: \text{ } W^{\iota}_j \in \cG_{\iota} \text{ }|W^{\iota}_j| \le \eta \}}p_j \le \eta \widetilde Z.
\]
Therefore, the amount of mass supported on segments shorter than \(\eta\) is proportional to \(\eta\) and we do not couple it. Thanks to the equation above and possibly choosing a bigger \(L\) in the statement, \mbox{we may assume without loss of generality that}
\begin{equation}\label{eq:miminal-length-for-coupling}
  |W^{\iota}_j| \ge \eta, \quad \forall j \in \cI, \text{ }\iota \in \{1,2\}.
\end{equation}
This observation will be useful later. By definition of \(\eta\)-coupling, for each \(x \in W^\iota_j\), there exists a segment \(\ell_x\) in the \(E^s_t\) direction joining \(W^\iota_j\) with \(W^{\iota^*}_j\). We define
\[
     W_{j,sing}^\iota = \left\{ x \in W^\iota_j: \ell_x \cap \cS_t^{+,N_0} \neq \emptyset \right\},
\]
and denote by \(\{\widetilde K^{\iota}_{j,p}\}_p\) a partition of \(W^{\iota}_{j} \setminus \cS_t^{+,N_0}\) into maximal connected components. We also consider the following partition of \(W^{\iota}_j\) in connected segments
\[
W^{\iota}_j = W^{\iota}_{j,sing} \cup \{K_{j,p}^{\iota}\}_p, \quad \cup_{p} K_{j,p}^{\iota} = (W^{\iota}_j \setminus W^{\iota}_{j,sing}),
\]
where \(K_{j,p}^{\iota}\subseteq \widetilde K_{j,p}^{\iota}\) are maximal intervals (possibly empty) such that their stable segments \(\ell_x\) do not intersect any singularity. We reorder the indices \(p\) such that every non-empty \(K^{\iota}_{j,p}\) is connected via stable segments to \(K^{\iota^{*}}_{j,p}\). Since \(\ell_x\) are parallel, the map that connects points through the stable direction is linear and it preserves proportions. I.e., for every \(j\) and \(p\),
\begin{equation}\label{eq:talete}
    \frac{|W^{1}_{j,sing}|}{|W^1_j|} =  \frac{|W^{2}_{j,sing}|}{|W^2_j|}, \quad \frac{|K_{j,p}^1|}{|W^{1}_j|} = \frac{|K_{j,p}^2|}{|W^{2}_j|}.
\end{equation}
Moreover, because \(\catt^{N_0}\) is continuous on \(\cM \setminus \cS_t^{+,N_0}\), all the points in \( K_{j,p}^{\iota} \subseteq  W_{j}^\iota \setminus  W_{j,sing}^\iota\) have their images in a \(\lambda^{-N_0}\eta\) stable neighborhood of a point of \(\catt^{N_0}K^{\iota^*}_j\). Therefore, according to Definition \ref{def:(p,eta)-coupling} (see also Definition \ref{eq:def-evolution-sf} with \(\catt^{N_0}\) in place of \(\catt\) for the evolution of the weights), the proportion of coupled mass for the families \(\catt^{N_0}\cG_{\iota}\) is not smaller than
\[
\begin{split}
&\sum_{j \in \cI}\sum_{p} \min \biggl\{p_j\frac{|\widetilde K^{1}_{j,p}|}{|W^{1}_j|} \frac{|\catt^{N_0} (K_{j,p}^1)|}{|\catt^{N_0}(\widetilde K_{j,p}^1)|}, p_j \frac{|\widetilde K^{2}_{j,p}|}{|W^{2}_j|} \frac{|\catt^{N_0}(K_{j,p}^2)|}{|\catt^{N_0}(\widetilde K_{j,p}^2)|}\biggr\}.
\end{split}
\]
By linearity, \(\catt\) disappears from the equation and \(|\widetilde K_{j,p}^{\iota}|\) cancels and, by the relations \eqref{eq:talete}, the quantity above is equal to
\[
\begin{split}
 \sum_{j \in \cI} p_j \sum_{p}\frac{ |K_{j,p}^{\iota}|}{|W^{\iota}_j|} =\sum_{j \in \cI} p_{j}\frac{| W^{\iota}_j \setminus W^\iota_{j,sing}|}{|W^{\iota}_j|}
 = 1- \sum_{j \in \cI} p_{j}\frac{| W^\iota_{j,sing}|}{|W^{\iota}_j|}. 
 \end{split}
\]
 This expression for the coupled mass holds for both \(\iota \in\{1,2\}\). Not surprisingly, the mass that will not be coupled is the one supported on points whose stable direction meets a discontinuity. Let \(E \in \bR^+\), depending only on \(\mathscr C^s\) and \(\mathscr C^u\), to be introduced shortly. We denote by \(\widetilde W^{\iota}_{j}\) an \(E\)-enlargement  of \(W^{\iota}_{j}\) (i.e., a segment such that \(|\widetilde W^{\iota}_{j}| = E|W^{\iota}_{j}|\) and \(W^{\iota}_{j}\) is in the center of \(\widetilde W^{\iota}_{j}\)). The following statement holds. There exists an \(E\)-enlargement of \(W^{\iota}_{j}\) with the following property: If \(\ell_x \cap \cS_t^{+, N_0} \supseteq \{z\}\) for some \(x \in W^{\iota}_{j,sing}\), \(z \in \cM\), then there exists a path \(\zeta_x \subseteq  \cS_t^{+, N_0}\) connecting \(z\) to \(\widetilde W^{\iota}_{j}\). In other words, if a singularity path intersects the stable line joining two coupled standard segments, then it must intersect both segments as well. The enlargement is needed because the singularity could actually `slightly miss' the original standard segments. We first prove the statement above. By Lemma \ref{lem:geom-disc}, there is a path inside \(\cS_t^{+, N_0}\) connecting \(z\) to \(\partial \cM\). But, by transversality, the mentioned path forms a positive angle with \(W^{\iota}_j\) and so it must intersect the line containing \(W^{\iota}_j\). However, by \eqref{eq:miminal-length-for-coupling}, the distance between \(W^1_j\) and \(W^2_j\) (not bigger than \(\eta\)) is in the worst case comparable with the size of the coupled segments (not smaller than \(\eta\)) and so it is sufficient to consider enlargements that are big enough to contain the intersection.

Continuing with the proof of the main statement, for any \(x \in W^{\iota}_{j,sing}\), by transversality again, the path \(\zeta_x\) introduced above and \( W^{\iota}_{j}\) describe everywhere a positive angle. Hence, since \(|\ell_x| \le \eta\), there exists \(C_{cone} \in \bR^+\), such that the distance between \(x\) and \(\zeta_x \cap \widetilde W^{\iota}_{j}\) is less than \(C_{cone}\eta\). It follows that any \(x \in W^{\iota}_{j,sing}\) belongs to a \(C_{cone}\eta\) neighborhood of an intersection between \(\widetilde W^{\iota}_{j}\) and \(\cS_t^{+, N_0}\). Therefore, 
\[
\begin{split}
    | W_{j,sing}^\iota|  \le C_{cone} \eta \#\{\cS_t^{+, N_0} \cap \widetilde W^{\iota}_{j}\} \le  C_{cone}\eta \left(K_{N_0} + P_{N_0}E|W^\iota_j|\right),
\end{split}
\]
where we have used Lemma \ref{lem:trasversality2} in the second inequality. By the above estimate, the amount of lost mass does not exceed
\begin{equation}\label{eq:amount-of-uncoupled}
\begin{split}
       \sum_{j \in \cI} p_{j}\frac{| W^\iota_{j,sing}|}{|W^{\iota}_j|}  &\le C_{cone}\eta(K_{N_0}\max_{\iota \in \{1,2\}}\cZ(\cG_{\iota}) + E P_{N_0}) \\
&\le  C_{cone}\eta(K_{N_0} \widetilde Z + E P_{N_0}).
\end{split}
\end{equation}
In the second inequality above we have used the regularity of the two families. We are about to conclude. Define \(L \in \bR^+\) as
\[
L = \sum_{j =0}^{\infty} \lambda^{-jN_0} C_{cone}(K_{N_0} \widetilde Z + E P_{N_0}).
\] 
By \eqref{eq:contraction-uniform-cat}, the mass \(m_j \in [0,1]\) that is still coupled after \(j N_0\) iterations, \(j \in \bN\), is \(\lambda^{-j N_0}\eta\)-coupled. Hence, by repeating the argument above with \(\lambda^{-j N_0}\eta\) in place of \(\eta\) and by equation \eqref{eq:amount-of-uncoupled}, we have that the proportion of initial mass that we fail to couple in the iterations between \(jN_0\) and \((j+1)N_0\) is not bigger than
\[
    C_{cone}\lambda^{-j N_0}\eta(K_{N_0} \widetilde Z + E P_{N_0}) m_j \le  C_{cone}\lambda^{-j N_0}\eta(K_{N_0} \widetilde Z + E P_{N_0}).
\]
Therefore, by summing over the iterations of \(\catt^{N_0}\), the proportion of initial mass that at time \(k N_0\) is \(\lambda^{-kN_0}\)-coupled is not smaller than, for any \(k \in \bN\),
\[
 1-\sum_{j=0}^{k-1} \lambda^{-j N_0}\eta C_{cone}(K_{N_0} \widetilde Z + E P_{N_0}) \ge 1- L\eta.
\]
Recalling \ref{obs:p-eta}, this proves the first part of the Lemma. As for the second part, by Lemma \ref{lem:Z-coupled} and the first part of the statement
\[
 \cZ\left(\Coup_{1-L\eta, \lambda^{- kN_0}\eta}(\catt^{k N_0}\cG_{\iota}, \catt^{k N_0} \cG_{\iota^*})\right) \le \frac{\cZ\left(\catt^{k N_0}\cG_{\iota}\right)}{1- L \eta}.
\]
By Lemma \ref{lem:growth-lemma} and since \(\cG_{\iota}\) are regular, the quantity above is not bigger than
\[
\frac{z^k \cZ(\cG_{\iota}) + Z}{1 - L\eta} \le \frac{z^k \widetilde Z + Z}{1- L \eta} \le \frac{z \widetilde Z + Z}{1- L \eta} \le \widetilde Z,
\]
where in the last inequality we have used that \(\widetilde Z \ge 2Z/(1-z)\) and we are considering \(\eta\) small enough. This concludes the proof of the Lemma.
\end{proof}
According to the previous result, the quantity of mass that `decouples' is proportional to the distance \(\eta\) of the coupled mass, which we expect to tend eventually to zero. However, it has a very bad regularity, since it is supported on small segments. These segments are created near the singularities and are order \(\eta\) as well. In other words, the price for losing very small mass is that we are forced to deal with mass supported on standard families with arbitrary bad regularity. This phenomenon is described by Lemma \ref{lem:Z-coupled}. To aid us, the Growth Lemma asserts that the mass becomes regular at an exponential rate. As it will be clear after the following two sections, the net effect is exponential mixing. 
\subsection{A linear scheme} \label{subsec:linear-scheme}\textit{In this subsection we use the previous results to show that the evolution of coupled and uncoupled mass obeys a particular linear scheme.} Let \(L\) be given by Lemma \ref{lem:push-sf}. We introduce \(\eta_0 \in \bR^+\) small enough such that Lemma \(\ref{lem:push-sf}\) holds and 
\begin{equation}\label{eq:initial-coup-dist}
\eta_{0} \le \frac{1}{2L}.
\end{equation} 
The number \(\eta_0\) is the maximal distance at which we couple mass and it is fixed now once and for all. Let \(N_{\eta_0} \in \bN\) be given by Lemma \ref{lem:first-coupl-s-f} for \(\eta_0\), and set
\begin{equation}\label{eq:time-window}
\begin{split}
&N_{c} = \min \left\{m N_0 : m \in \bN, \text{ }mN_0 \ge N_{\eta_0}, \text{ }\lambda^{-mN_0} < \frac 1 2\right\}, \\
&\hspace{1.7cm} \rho = \lambda^{-N_c} \in \left (0,\frac 1 2 \right).
\end{split}
\end{equation}
Notice that \(N_c \in \bN\) is a multiple of \(N_0\), as requested by Lemma \ref{lem:push-sf}, and is not smaller than \(N_{\eta_0}\), as requested by Lemma \ref{lem:first-coupl-s-f}. We will measure time in units of \(N_c\). To deal with standard families with very small segments we introduce \textit{regularity classes}. Recall \(z\) and \(Z\) from Lemma \ref{lem:growth-lemma}. For \(r \in \bN_0\), we define inductively \(Q_r \in \bR^+\), by
\[
  \begin{split}
Q_0 = \widetilde Z, \quad Q_{r+1} = \left(Q_{r} - Z\right) z^{- \frac{N_c}{N_0}}.
 \end{split}
\]
We have that \(Q_{r}\) grows exponentially in \(r\). Indeed, by writing explicitly \(Q_r\) according to the equation above, for any \(r \in \bN\),
\begin{equation}\label{eq:lower-bound-Q}
\begin{split}
Q_{r} &= \widetilde Z z^{-r \frac{N_c}{N_0} } - Z \sum_{j=1}^r z^{-j\frac{N_c}{N_0} } \ge \widetilde Z z^{-r \frac{N_c}{N_0}}- Z\frac{ z^{-(r+1)\frac{N_c}{N_0}}-1}{z^{-\frac{N_c}{N_0}}-1} \\
&\ge  \left(\widetilde Z - \frac{Z}{1-z^{\frac{N_c}{N_0}}}\right)z^{-r \frac{N_c}{N_0}} \ge \left(\widetilde Z - \frac{Z}{1-z}\right)z^{-r \frac{N_c}{N_0}} \ge\frac{\widetilde Z}{2} z^{-r\frac{N_c}{N_0}},
\end{split}
\end{equation}
where in the last inequality we have used that \(\widetilde Z = 20000Z_{*} \ge 2Z/(1-z)\) as it is clear from \eqref{eq:star-2star}. We say that a standard family \(\cG\) is in the \(r^{th}\)-regularity class if \(\cZ(\cG) \le Q_{r}\) and we write that \(\cG \in H_{r}\). By the first part of Lemma \ref{lem:growth-lemma} and the definition of \(Q_r\), for any \(r \in \bN_0\), \(\cG \in H_{r+1}\),
\[
\cZ (  \catt^{N_c} \cG ) = \cZ( \catt^{N_0\frac{N_c}{N_0}} \cG) \le z^{\frac{N_c}{N_0}} \cZ (\cG) + Z \le z^{\frac{N_c}{N_0}} Q_{r+1} + Z = Q_{r}.
\]
In particular, the last equation implies that
\begin{equation} \label{eq:regularity-one-time}
   \catt^{N_c} \cG \in H_{r} \quad \forall \text{ }\cG \in H_{r+1}.
\end{equation}
In other words, the mass in \(H_{r}\) takes \(r\) iterates of \(\catt^{N_c}\) before becoming mass supported on a regular standard family that we can couple (via Lemma \ref{lem:first-coupl-s-f}). Notice that \(\cG \in H_0\) is equivalent to saying that \(\cG\) is regular. We now introduce some book-keeping tools. Recall that \(\rho \in (0,1/2)\) is defined in \eqref{eq:time-window}.
\begin{definition}\label{def:ckuk}[Coupled-uncoupled decomposition]
For any two standard families \(\cG_1\), \(\cG_2\) and any \(n \in \bN_0\), we write 
\[
  c_{r}^{\cG_1, \cG_2} (n) = c_r, \quad  u_{r}^{\cG_1, \cG_2}(n) = u_r,
\]
for \(r \in \bN_0\), \(c_r, u_r, \in [0,1]\), if there exist \(\rho^{r}\eta_0\)-coupled regular standard families \(\cG_{1,r}^c\), \(\cG_{2,r}^c\) and standard families \(\cG_{1,r}^u\), \(\cG_{2,r}^u\) in \(H_r\) such that, for \(\iota \in \{1,2\}\),
\begin{equation}\label{eq:decomposition-sf}
\begin{split}
   &\mu_{\catt^{nN_c}\cG_\iota} = \sum_{r \in \bN_0} c_r \mu_{\cG_{\iota,r}^c} + \sum_{r \in \bN_0}u_r\mu_{\cG_{\iota,r}^u}.
\end{split}
\end{equation}
\end{definition}
We may drop the superscript and write simply \(c_r(n), u_r(n)\), whenever there is no ambiguity on the standard families. Notice that the decomposition in \eqref{eq:decomposition-sf} may be non unique in the sense that the value of \(c_{r}(n)\) and \(u_{r}(n)\) are not uniquely determined by \(\cG_1\) and \(\cG_2\). Whenever we write that \(c_{r}(n)\) and \(u_{r}(n)\) have some value, we mean that there exist decompositions as in \eqref{eq:decomposition-sf} such that \(c_{r}\) and \(u_{r}\) have that particular value. 
The utility of Definition \eqref{def:ckuk} is that it records only the essential data about the evolution of two standard families: how much mass is coupled and at which distance, and the regularity of the uncoupled part. From now on, let \(p_c, \in (0,1]\) and \(\ve_0 \in \bR^+\) be fixed according to the following: 
\begin{equation}\label{eq:fix-eps0}
\begin{split}
\text{Let } p_c \in (0,1]& \text{ and } \ve_{0} \in \bR^+ \text{ be given by Lemma \ref{lem:first-coupl-s-f}}\\
&\text{  for }\eta = \eta_0 \text{ and } k = N_c,
\end{split}
\end{equation}
where we recall that \(N_c\) was introduced in \eqref{eq:time-window}. The following linear scheme encodes essentially all the results from the last three sections.
\begin{lemma}\label{lem:infinite-matrix-scheme}
  Let  \(t \in [0,\ve_0]\), \(r_0 \in \bN_0\), and \(\cG_1, \cG_2 \in H_{r_0}\) be two standard families. There exist \(\beta_1, \beta_2 \in \bN\) and decompositions as in Definition \ref{def:ckuk} of \(\catt^{nN_c} \cG_{\iota}\), \(\iota \in \{1,2\}\), \(n \in \bN_0\), such that
\[
   c_r(0) = 0 \text{ }\forall r\in \bN_0, \quad  u_{r_0}(0) =  1, \quad u_r(0) = 0 \text{ } \forall r \in \bN_0 \setminus \{r_0\},
\]
and \(c_r(n), u_r(n)\), \(n \in \bN\), are defined inductively by
\[
  \begin{split}
    &c_{r} (n+1) = \left(1- \frac 1 2\rho^{r-1}\right) c_{r-1}(n) \quad \text{if} \text{ }\text{ }  r \in \bN \\
    & c_{0}(n+1) = p_{c}u_0(n)\\
    & u_{0}(n+1) =  u_{1}(n) +(1-p_{c}) u_{0}(n)\\
    &u_{r}(n+1) =  \begin{cases}
      u_{r+1}(n)  +  \frac 1 2\rho^{ \frac{r-\beta_2}{\beta_1}}c_{\frac{r-\beta_2}{\beta_1}}(n), & \text{ if}  \text{ }\text{ }r  \in \mathfrak N,\\
      u_{r+1}(n),  & r \in \bN \setminus \mathfrak N,
    \end{cases} 
  \end{split}
\]
where we have set \(\mathfrak N = \{r \in \bN: (r-\beta_2)/\beta_1 \in \bN_0\}\).
\end{lemma}
\begin{proof}
We prove the statement by induction. For \(n = 0\) we can always consider \(\cG_1\) and \(\cG_2\) totally uncoupled and the statement follows from the fact that \(\cG_1\) and \(\cG_2\) belong to \(H_{r_0}\) by hypothesis. Let us assume that the statement is true for some \(n \in \bN_0\) and let \(\cG_{\iota,r}^{c/u}\) be given by the decompositions \eqref{eq:decomposition-sf}. Then,
\begin{equation}\label{eq:decomposition-sf-proof}
\begin{split}
   \mu_{\catt^{N_c (n+1)}\cG_\iota} &= \sum_{r \in \bN_0} c_r(n) \mu_{\catt^{N_c}\cG_{\iota,r}^c} + \sum_{r \in \bN_0}u_r(n)\mu_{\catt^{N_c}\cG_{\iota,r}^u}\\
&=\sum_{r \in \bN_0} c_r(n) \mu_{\catt^{N_c}\cG_{\iota,r}^c} + u_0(n)\mu_{\catt^{N_c}\cG_{\iota,0}^u} + \sum_{r \in \bN}u_r(n)\mu_{\catt^{N_c}\cG_{\iota,r}^u}.
\end{split}
\end{equation}
We study all the terms in \eqref{eq:decomposition-sf-proof} one by one. We start with the rightmost sum. By \eqref{eq:regularity-one-time}, the standard families \(\widetilde \cG_{\iota,r}^u := \catt^{N_c}\cG_{\iota,r+1}^u\) belong to \(H_r\), \(r \in \bN_0\), and we have
\begin{equation}\label{eq:matrix-u-u}
\sum_{r \in \bN}u_r(n)\mu_{\catt^{N_c}\cG_{\iota,r}^u} = \sum_{r \in \bN_0}u_{r+1}(n)\mu_{\widetilde \cG_{\iota,r}^u}.
\end{equation}
We now consider the term with \(u_0(n)\). Recalling \eqref{eq:fix-eps0} and Lemma \ref{lem:first-coupl-s-f}, there exist regular \((p_c, \eta_0)\)-coupled standard families \(\widetilde \cG_{\iota,0}^c\) and a regular uncoupled standard family \(\widetilde \cG_{\iota,0}^u\) such that
\begin{equation}\label{eq:matrix-u-c}
\begin{split}
&u_0(n)\mu_{\catt^{N_c}\cG_{\iota,0}^u} =  u_0(n)p_c\mu_{\widetilde \cG_{\iota,0}^c} + u_0(n)(1-p_c)\mu_{\widetilde \cG_{\iota,0}^u}.
\end{split}
\end{equation}
Finally, we consider the leftmost sum in \eqref{eq:decomposition-sf-proof}. By Lemma \ref{lem:push-sf} and since \(\cG_{\iota, r}^c\) and \(\cG_{\iota^*,r}^c\) are \(\eta_0 \rho^{r}\)-coupled, one has that \(\catt^{N_c}\cG_{1,r}^c\) and \(\catt^{N_c}\cG_{2,r}^c\) are 
\[
\begin{split}
   (1 -L\eta_0 \rho^r, \eta_0 \rho^{r+1})-\text{coupled},
\end{split}
\]
where we have also used the definition of \(\rho\). Recalling that \(\eta_0 \le 1/(2L)\), as specified in \eqref{eq:initial-coup-dist}, and possibly by coupling less mass according to \ref{obs:pless}, there exist two standard families \(\cH_{\iota,r} \simeq \catt^{N_c}\cG_{\iota,r}^{c}\) which are \((1 - \rho^r/2, \eta_0 \rho^{r+1})-\)coupled and 
\[
\cZ(\Coup(\cH_{\iota,r}, \cH_{\iota^{*},r})) = \cZ(\Coup(\catt^{N_c}\cG_{\iota,r}^{c},\catt^{N_c}\cG_{\iota^{*},r}^{c})) \le \widetilde Z.
\]
In the equation above, the equality is a consequence of \ref{obs:pless} and the inequality follows by the second part of Lemma \ref{lem:push-sf}. Moreover, by Lemma \ref{lem:Z-coupled},
\[
\begin{split}
&\cZ \left(\cH_{\iota,r} \! \setminus \!\Coup (\cH_{\iota,r}, \cH_{\iota^{*},r}) \right) \!\le \!\frac{3 \cZ(\cH_{\iota,r})}{\frac{1}{2} \rho^{r}} \!=\! \frac{6\cZ(\catt^{N_c}\cG_{\iota,r}^{c})}{\rho^r}\!\le\! \frac{6\bigl(z^{\frac{N_c}{N_0}}\widetilde Z + Z\bigr) }{\rho^{r}} \!\le\!  \frac{6\widetilde Z}{\rho^r},
\end{split}
\]
where in the second inequality we have used the Growth Lemma \ref{lem:growth-lemma} and the fact that \(\cG_{\iota,r}^{c}\) is regular and the equality follows from \(\cH_{\iota,r} \simeq \catt^{N_c}\cG_{\iota,r}^{c}\). To have a bound in terms of regularity classes, notice that there exist \(\beta_1, \beta_2 \in \bN\) big enough such that, for any \(r \in \bN_0\), the quantity above is not bigger than
\[
\frac{\widetilde Z}{2}z^{-\frac{N_c}{N_0}(\beta_1 r + \beta_2)} \le Q_{ \beta_1 r  + \beta_2 },
\]
where we have used \eqref{eq:lower-bound-Q}. Hence, as \(r\) ranges over \(\bN_0\), we have the following \((1-\rho^{r/2}, \eta_0 \rho^{r+1})\)-coupled regular standard families and uncoupled standard families in \(H_{\beta_1 r + \beta_2} \), 
\[
\widetilde \cG_{\iota,r+1}^c :=  \Coup (\cH_{\iota,r}, \cH_{\iota^*,r}), \text{ } \text{ } \widetilde \cG_{\iota, \beta_1 r + \beta_2}^{u} := \cH_{\iota,r} \setminus  \Coup (\cH_{\iota,r}, \cH_{\iota^*,r}),
\]
whose sum, weighted by the coupled/uncoupled mass, is equivalent to \(\catt^{N_c}\cG_{\iota,r}^{c}\). Therefore, we have that
\begin{equation}\label{eq:matrix-u-c2}
\sum_{r \in \bN_0} c_r(n) \mu_{\catt^{N_c}\cG_{\iota,r}^c} = \sum_{r \in \bN_0} c_r(n) \left[ \left(1- \frac 1 2 \rho^r\right) \mu_{\widetilde \cG_{\iota,r+1}^c} + \rho^r\frac 1 2 \mu_{\widetilde \cG_{\iota, \beta_1 r + \beta_2}^{u}}\right].
\end{equation}
By substituing \eqref{eq:matrix-u-u}, \eqref{eq:matrix-u-c} and \eqref{eq:matrix-u-c2} into \eqref{eq:decomposition-sf-proof} we prove the statement.

\end{proof}

\subsection{Exponential decay of correlations for standard families} \textit{In this subsection we use the linear scheme developed in Subsection \ref{subsec:linear-scheme} to show exponential decay of correlations for standard families and prove Proposition \ref{prop:fundamental-example}.} The relations in Lemma \ref{lem:infinite-matrix-scheme} imply that all the mass of two standard families gets closer and closer exponentially. A way to prove this is to introduce a quantity that controls the distance between two standard families and is contracted by the scheme described above. Let \(\beta_1, \beta_2\) be given by Lemma \ref{lem:infinite-matrix-scheme} and \(\psi_- \in (0,1)\), \(\psi_+ \in (1, \infty)\) be such that
\begin{equation}\label{eq:defpsi}
    \psi_- \in \left(\rho, \frac 1 2\right), \quad \psi_+ \in \left(1, \min \left\{\left(\frac{\psi_-}{\rho}\right)^{\frac 1 \beta_1}, [2(1-\psi_-)]^{\frac{1}{\beta_2+1}}, z^{-\frac{N_c}{N_0}} \right\}\right).
\end{equation}
Notice that we can guarantee the above relations by choosing \(\psi_-\) very close to \(1/2\) and \(\psi_{+}\) very close to \(1\) depending on \(\psi_{-}\). The parameter \(z\in(0,1)\) is given by the Growth Lemma. We isolate here a consequence of equation \eqref{eq:defpsi} that will be useful later. For any \(r \in \bN_0\),
\begin{equation}\label{eq:uncunny-algebra}
\begin{split}
 \psi_- \left(1-\frac{\rho^r}{2}\right) + \frac{\rho^r}{2}\left(\frac{\psi_{+}^{\beta_1}}{\psi_{-}} \right)^r \psi_{+}^{\beta_2 + 1} < \psi_- + \frac{\psi_+^{\beta_2 +1}}{2}   <1,
\end{split}
\end{equation}
where in the first inequality we have used the first condition on \(\psi_{+}\) and in the second inequality the second condition. We also introduce the following norm on the set of pairs \(\{(c_r, u_r)\}_r\) of \(\bN_0\)-sequences of nonnegative real numbers,
\begin{equation}\label{eq:norm-def}
       \left\|(c_r,u_r)\right\|_{\star} = \sum_{r \in \bN_0} \psi_-^{r} c_r + \sum_{r \in \bN_0}\psi_+^{r+1}u_r.
\end{equation}
Let \((c_r(n),u_r(n))\), \(n \in \bN_0\), be as in the statement of Lemma \ref{lem:infinite-matrix-scheme} (for any \(r_0 \in \bN_0\)).
\begin{lemma}\label{lem:contraction-strange-norm} There exists \(\tau \in (0,1)\) such that, for all \(n \in \bN_0\),
  \[
  \|(c_r(n+1), u_r(n+1))\|_{\star} \le \tau \|(c_r(n), u_r (n))\|_{\star}.
  \]
\end{lemma}
\begin{proof}
 We find convenient to split the sums in the \(\|\cdot\|_{\star}\) norm in the following way,
  \begin{equation}\label{eq:norm0}
\begin{split}
 \|(c_r(n+1),& u_r(n+1))\|_{\star} = \sum_{r \in \bN_0}\psi_-^{r+1} c_{r+1}(n+1) + c_0(n+1) \\
&+ \psi_+ u_0(n+1) + \sum_{r \in \bN \setminus \mathfrak N}\psi_+^{r+1}u_r(n+1) + \sum_{r \in \mathfrak N}\psi_+^{r+1}u_{r}(n+1).
\end{split}
\end{equation}
Recalling Lemma \ref{lem:infinite-matrix-scheme} and the definition of \(\mathfrak N\) there, equation \eqref{eq:norm0} is equal to
\begin{equation}\label{eq:norm1}
\begin{split}
&\sum_{r \in \bN_0}\psi_-^{r+1} \left(1- \frac 1 2\rho^{r}\right) c_r(n)  + p_c u_0(n) + \psi_+ \left( u_{1}(n) +(1-p_{c}) u_{0}(n) \right) \\
&+ \sum_{r \in \bN \setminus \mathfrak N}\psi_+^{r+1}u_{r+1}(n) + \sum_{r \in \mathfrak N}\psi_+^{r+1} \left(u_{r+1}(n) +  \frac 1 2\rho^{ \frac{r-\beta_2}{\beta_1}}c_{\frac{r-\beta_2}{\beta_1}}(n)\right).
\end{split}
\end{equation}
Grouping similar terms together and after some algebra, \eqref{eq:norm1} becomes
\begin{equation}
\begin{split}
\sum_{r \in \bN_0}\biggl[\psi_-^{r+1} &\left(1- \frac 1 2\rho^{r}\right)   + \psi_+^{r\beta_1 + \beta_2+1}  \frac 1 2\rho^{r} \biggr]c_r(n) \\
&+\left(p_c + \psi_+ (1-p_c)\right)u_0(n) + \sum_{r \in \bN_0}\psi_+^{r+1} u_{r+1}(n)\\
&\!\!\!\!\!\!\!\!\!\!\!\!\!\!\!\!\!\!\!\!\!\!\!\!\!\!\!\!\!\!\!\!\!=  \sum_{r \in \bN_0}\left[\psi_- \left(1-\frac{\rho^r}{2}\right) + \frac{\rho^r}{2}\biggl(\frac{\psi_{+}^{\beta_1}}{\psi_{-}} \biggr)^r \psi_{+}^{\beta_2 + 1}\right]\psi_{-}^r c_r(n)\\
& +\left(\psi_{+}^{-1}p_c +  (1-p_c)\right) \psi_+ u_0(n) + \psi_+^{-1}\sum_{r \in \bN}\psi_+^{r+1} u_{r}(n).
\end{split}
\end{equation}
Recalling \eqref{eq:uncunny-algebra}, this is less that \(\tau \|c_r(n), u_r(n)\|_{\star}\) where,
\[
\begin{split}
   \tau = \max \Biggl\{\psi_{+}^{-1}p_c &+  (1-p_c), \\
 &\sup_{r \in \bN_{0}}\biggl[\psi_- \left(1-\frac{\rho^r}{2}\right) + \frac{\rho^r}{2}\biggl(\frac{\psi_{+}^{\beta_1}}{\psi_{-}} \biggr)^r \psi_{+}^{\beta_2 + 1}\biggr]\Biggr\}<1.
\end{split}
\]
This concludes the proof of the Lemma.
\end{proof}

At this point, we illustrate how the previous results lead to decay of correlations.

\begin{lemma}\label{lem:decay-pre-standard}
  There exist \(C \in \bR^+\) and \(\gamma \in (0,1)\) such that, for each \(t \in [0, \ve_0]\) and any two standard families \(\cG_1\) and \(\cG_2\), we have
  \[
\begin{split}
      \bigl|\mu_{\cG_1}\left(\vf \circ \catt^n\right) &- \mu_{\cG_2}\left(\vf \circ \catt^n\right)\bigr| \le C  \max\left\{\cZ(\cG_1), \cZ(\cG_2) \right\} \|\vf\|_{\cC^1(\cM)}\gamma^n,
\end{split}
  \]
for any \(\vf \in \cC^1(\cM)\) and \(n \in \bN_0\).
\end{lemma}
\begin{proof}
Recall \(Z_{**}\in \bR^+\), \(z \in (0,1)\) and \(N_0 \in \bN\) from the Growth Lemma and define
\begin{equation}\label{eq:initial-reg-class0}
  r_0 = \max\left\{\left\lfloor\frac{N_0}{N_c} \ln_{z^{-1}}\biggr(Z_{**}\max\left\{\cZ(\cG_1), \cZ(\cG_2)\right\}\biggl)\right\rfloor + 1, 0\right\}.
\end{equation}
By the Growth Lemma and the definition above, for any \(n_0 \in \bN_0\),
\[
\begin{split}
   \max\left\{\cZ(\catt^{n_0}\cG_1), \cZ(\catt^{n_0}\cG_2)\right\} &\le Z_{**}\max\left\{\cZ(\cG_1), \cZ(\cG_2)\right\}\\
   &\le  z^{-r_0 N_c / N_0} \le z^{-r_0 N_c / N_0} \widetilde Z/2 \le Q_{r_0},
\end{split}
\]
where we have used \eqref{eq:lower-bound-Q} in the last inequality. Hence, for any given \(n_0 \in \bN_0\), we have that \(\catt^{n_0}\cG_{\iota}\) belongs to the regularity class \(H_{r_0}\), and we have the following estimate that will be useful later on
\begin{equation}\label{eq:initial-norm}
\begin{split}
&\biggl\|\bigl(c^{\catt^{n_0} \cG_1, \catt^{n_0} \cG_2}_r (0), u^{\catt^{n_0} \cG_1, \catt^{n_0} \cG_2}_r (0)\bigr) \biggr\|_{\star} \le  \psi_{+}^{r_0+1} \\
&\le C \psi_{+}^{\frac{N_0}{N_c}\ln_{z^{-1}} \max\{\cZ(\cG_1), \cZ(\cG_2)\}}  \le C\max\left\{\cZ(\cG_1), \cZ(\cG_2) \right\},
\end{split}
\end{equation}
where the first inequality follows from the definition of the norm and the regularity of \(\catt^{n_0} \cG_{\iota}\), in the second we have used the definition \eqref{eq:initial-reg-class0} of \(r_0\), and the rightmost inequality is a consequence of the last condition on \(\psi_{+}\) in \eqref{eq:defpsi}. 

For any \(n \in \bN_0\), we write \(n = \left\lfloor n/N_c\right\rfloor N_c +n_0\), \(n_0 \in\{0,...,N_c -1\}\). According to this decomposition,
\begin{equation}\label{eq:corr-two-sp}
\begin{split}
    \bigl|\mu_{\cG_1}(\vf \circ \catt^n)& - \mu_{\cG_2} (\vf \circ \catt^n ) \bigr| \\
&= \left|\mu_{\catt^{n_0} \cG_1}(\vf \circ \cF_t^{\left\lfloor \frac{n}{N_c}\right \rfloor N_c}) - \mu_{\catt^{n_0} \cG_2} (\vf \circ \catt^{\left\lfloor \frac{n}{N_c}\right\rfloor N_c})\right|, 
\end{split}
\end{equation}
Applying Lemma \ref{lem:infinite-matrix-scheme} to the two families \(\catt^{n_0}\cG_{\iota} \in H_{r_0}\) and recalling the decomposition \eqref{eq:decomposition-sf} and the corresponding notation, we \mbox{re-write the expression above as}
\[
\begin{split}
    \biggl|\sum_{r \in \bN_0} c_r^{\catt^{n_0} \cG_1, \catt^{n_0} \cG_2} &\left(\left\lfloor n/N_c \right\rfloor\right) (\mu_{\cG_{1,r}^c}(\vf) - \mu_{\cG_{2,r}^c}(\vf)) \\
&+ \sum_{r \in \bN_0} u_r^{\catt^{n_0} \cG_1, \catt^{n_0} \cG_2} \left(\left\lfloor  n/N_c\right\rfloor\right) (\mu_{\cG_{1,r}^u}(\vf) - \mu_{\cG_{2,r}^u}(\vf)) \biggr|.
\end{split}
\]
Since by definition \(\cG_{1,r}^c\) and \(\cG_{2,r}^c\) are \(\eta_0 \rho^r\)-coupled, by \ref{item:dist}, \(|\mu_{\cG_{1,r}^c}(\vf) - \mu_{\cG_{2,r}^c}(\vf)| \le \eta_0 \rho^{r} \|\vf'\|_{\cC^0}\). Hence, by the triangular inequality, the previous quantity is not bigger than
\[
\begin{split}
\sum_{r \in \bN_0} c_r^{\catt^{n_0} \cG_1, \catt^{n_0} \cG_2} &\left(\left\lfloor  n/N_c\right\rfloor\right)\rho^{r}\eta_0\|\vf'\|_{\cC^0} \\
&+ 2\sum_{r \in \bN_0} u_r^{\catt^{n_0} \cG_1, \catt^{n_0} \cG_2} \left(\left\lfloor  n/N_c\right\rfloor\right)\|\vf\|_{\cC^0}.
\end{split}
\]
By \eqref{eq:defpsi}, we have that \(\rho < \psi_-\) and \(\psi_+ > 1\). Hence, the quantity above can be bounded using the \(\|\cdot\|_{\star}\)-norm by
\[
\begin{split}
      &\max\{2,\eta_0\}\biggl\|\bigl(c_r^{\catt^{n_0} \cG_1, \catt^{n_0} \cG_2}\left(\left\lfloor  n/N_c \right\rfloor \right), u_r^{\catt^{n_0} \cG_1, \catt^{n_0} \cG_2}  \left (\left\lfloor  n/N_c \right \rfloor \right)\bigr)\biggr\|_{\star}\|\vf\|_{\cC^1}.
\end{split}
\]
Therefore, by Lemma \ref{lem:contraction-strange-norm}, there exists \(\tau \in (0,1)\) such that the previous expression is not bigger than
\[
\begin{split}
&\max\{2,\eta_0\}\tau^{\left\lfloor \frac{n}{N_c}\right\rfloor }\biggl\|\bigl(c_r^{\catt^{n_0} \cG_1, \catt^{n_0} \cG_2}\left(0\right), u_r^{\catt^{n_0} \cG_1, \catt^{n_0} \cG_2}\left ( 0\right)\bigr)\biggr\|_{\star} \|\vf\|_{\cC^1}.
\end{split}
\]
By \eqref{eq:initial-norm} and setting \(\gamma = \tau^{1/N_c} \in (0,1)\), the quantity above, that is a bound for the difference \eqref{eq:corr-two-sp} between correlations of two standard families, is estimated with
\[
      C\max\left\{\cZ(\cG_1), \cZ(\cG_2) \right\}\tau^{\left\lfloor\frac{n}{N_c}\right\rfloor}\|\vf\|_{\cC^1} \le  C \max\left\{\cZ(\cG_1), \cZ(\cG_2) \right\} \gamma^n \|\vf\|_{\cC^1},
\]
concluding the proof of the Lemma.
\end{proof}

\begin{lemma}\label{lem:phys-measure}
There exist \(C \in \bR^+\) and \(\gamma \in (0,1)\) such that, for every \(t \in [0, \ve_0]\), there exists a measure \(\mu_t\) such that, for any standard family \(\cG\) and \(n \in \bN_0\),
\[
\begin{split}
&\lim_{m \to \infty} \mu_{\cG}\left(\vf \circ \catt^m\right) = \mu_t(\vf), \quad \forall \vf \in \cC^0(\cM),\\
&\left |\mu_{\cG}\left(\vf \circ \catt^n\right) - \mu_t(\vf)\right| \le C \cZ(\cG)\|\vf\|_{\cC^1}\gamma^n, \quad \forall \vf \in \cC^1(\cM).
\end{split}
\]
\end{lemma}
\begin{proof}
Let \(\cG\) be any standard family. By the Growth Lemma \ref{lem:growth-lemma}, for any \(m \in \bN_0\), 
\begin{equation}\label{eq:Z-nonincrease}
\cZ(\catt^m \cG) \le \max\{Z_{**}\cZ(\cG), Z_{*}\}.
\end{equation}
Let \(\vf \in \cC^0(\cM)\) and fix any \(\ve \in \bR^+\). Take \(\vf_{\ve} \in \cC^1(\cM)\) such that \(\|\vf_{\ve} - \vf\|_{\cC^0} = \ve\). By Lemma \ref{lem:decay-pre-standard}, there exists \(C \in \bR^+\) and \(\overline n \in \bN\) such that for any \(n \ge \overline n\), \(m \in \bN_0\),
\[
\begin{split}
\bigl| \mu_{\catt^n \cG}(\vf) -& \mu_{\catt^{n+m} \cG}(\vf)\bigr| \le \left| \mu_{\catt^n \cG}(\vf) - \mu_{\catt^{n} \cG}(\vf_j)\right| \\
&+ \left| \mu_{\cG}(\vf_j \circ \catt^n) - \mu_{\catt^{m} \cG}(\vf_j \circ \catt^n)\right| + \left| \mu_{\catt^{n+m} \cG}(\vf_j) - \mu_{\catt^{n+m} \cG}(\vf)\right|\\
&\le 2\ve + C \gamma^n \max \{\cZ(\cG), \cZ(\catt^m \cG)\}\|\vf_{\ve}\|_{\cC^1}  \le 3 \ve,
\end{split}
\]
where in the last inequality we have used \eqref{eq:Z-nonincrease}. Therefore \(\{\mu_{\catt^n \cG}\}_{n \in \bN}\) is a Cauchy sequence (in the weak topology) and it converges to a measure \(\mu_t\). Furthermore, also any measure associated with any other standard family \(\widetilde \cG\) has the same limit. Indeed, for any \(\ve \in \bR^+\), approximating \(\vf\) with \(\cC^1\) functions \(\vf_{\ve}\) as before and by using that \(\mu_{\catt^m\cG}\) tends to \(\mu_t\), there exists an \(\overline n \in \bN\) such that for any \(n \ge \overline n\),
\[
\begin{split}
|\mu_{\widetilde \cG}(\vf\circ \catt^n) - \mu_t(\vf)| & = \lim_{m \to \infty}|\mu_{\catt^n\widetilde \cG}(\vf) - \mu_{\catt^{n+m} \cG}(\vf)| \\
&\le 2\ve + C\gamma^n \max \left\{\cZ(\widetilde \cG), \cZ(\cG), Z_{*}\right\}\|\vf_{\ve}\|_{\cC^1} \le 3 \ve,
\end{split}
\]
where the first inequality can be deduced as in the equation above using \eqref{eq:Z-nonincrease} and Lemma \ref{lem:decay-pre-standard}. This shows the first part of the Lemma. Using the result we just proved and Lemma \ref{lem:decay-pre-standard}, for any \(\vf \in \cC^1(\cM)\) and standard family \(\cG\),
\[
\begin{split}
\bigl|\mu_{\cG}(\vf \circ \catt^n ) \!-\! \mu_t(\vf) \bigr| & = \!\!\lim_{m \to \infty}\!|\mu_{\cG}(\vf \circ \catt^n ) \!- \!\mu_{\catt^m \cG}(\vf \circ \catt^n) |\!\le\! C \gamma^n\! \cZ(\cG)\|\vf\|_{\cC^1},
\end{split}
\]
concluding the proof of the Lemma.
\end{proof}
We have just obtained that constant probability densities supported on segments converge to the physical measure exponentially fast. Since these constant densities may be supported on arbitrarily short segments, we can upgrade the result to \(\cC^1\) densities on standard segments by approximation.
\begin{lemma}\label{lem:spc1}
There exist \(C \in \bR^+\) and \(\gamma \in (0,1)\) such that, for all \(t \in [0,\ve_0]\), \(W \in \cW\), \(\rho \in \cC^1(W, \bR)\) and \(n \in \bN_0\),
\[
\left|\int_{W} (\vf \circ \catt^n) \rho - \mu_t(\vf)\int_{W} \rho \right| \le C \gamma^n \|\rho\|_{\cC^1}\|\vf\|_{\cC^1}.
\]
\end{lemma}
\begin{proof}
Consider a partition \(\{W_i\}\) of \(W\) in \(N \!\in \!\bN\) \mbox{segments of size \(|W|/N\). One has}
\begin{equation}\label{eq:apploximation}
\biggl|\rho - \sum_{i=1}^N \frac{\int_{W_i}\rho}{|W_i|} \mathbbm{1}_{W_i} \biggr| \le \|\rho'\|_{\cC^0}\max_{i \in\{1,...N\}}|W_i| \le \frac{\|\rho\|_{\cC^1}|W|}{N}.
\end{equation}
Recall that \(\cZ(\cG_i) = 1/|W_i|\) for the standard families \(\cG_i = (W_i, 1)\). Hence, by the second part of Lemma \ref{lem:phys-measure}, there exist \(C \in \bR^+\) and \(\gamma \in (0,1)\) such that, for any \(\vf \in \cC^1\) and \(i \in \{1,...,N\}\), we have the following bound for later,
\begin{equation}\label{eq:decay-interval-approx}
\begin{split}
\biggl|\int_{W_i} \!\!(\vf \circ \catt^n) \frac{\int_{W_i}\rho}{|W_i|} - \mu_t(\vf)\int_{W_i}\!\!\rho \biggr| & = \left|\int_{W_i}\rho\right| \biggl|\int_{W_i} \frac{\vf \circ \catt^n}{|W_i|}  - \mu_t(\vf)\biggr| \\
&\le  C\frac{\int_{W_i}|\rho|}{|W_i|}  \gamma^n \|\vf\|_{\cC^1}.
\end{split}
\end{equation}
Returning to the main estimate, by adding and subtracting and a triangular inequality,
\[
\begin{split}
\biggl|\int_{W} (\vf \circ \catt^n)\rho &- \mu_t(\vf)\int_{W} \rho \biggr| \le \biggl|\int_{W} \vf \circ \catt^n \biggl(\rho - \sum_{i=1}^N \frac{\int_{W_i}\rho}{|W_i|} \mathbbm{1}_{W_i}\biggr)\biggr|\\
&  +\sum_{i=1}^N \biggl| \int_{W_i} \vf \circ \catt^n \frac{\int_{W_i}\rho}{|W_i|}  -  \mu_t(\vf)\int_{W_i}\rho\biggr|.
\end{split}
\]
By \eqref{eq:apploximation} and \eqref{eq:decay-interval-approx} respectively, the previous expression is bounded by
\[
\begin{split}
&\frac{\|\rho\|_{\cC^1}\|\vf\|_{\cC^0}|W|^2}{N}+C \sum_{i=1}^N \frac{\int_{W_i}|\rho|}{|W_i|} \gamma^n \|\vf\|_{\cC^1}. 
\end{split}
\]
Recalling that \(|W_i| = |W|/N\), the above quantity is less or equal to, for any \(n \in \bN_0\),
\[
  \frac{\|\rho\|_{\cC^1}\|\vf\|_{\cC^0}|W|^2}{N} + C \frac{\int_{W}|\rho|}{|W|} N \gamma^n\|\vf\|_{\cC^1} \le C \|\vf\|_{\cC^1}\|\rho\|_{\cC^1} \left(  N^{-1}+ N \gamma^n \right).
\]
We have also used that \(|W| \!\le \!\sqrt{2}\) for \(W \!\in \!\cW\). Setting \(N \! =\! \gamma^{-n/2}\) we obtain the sought bound \(C \|\vf\|_{\cC^1}\|\rho\|_{\cC^1} \gamma^{n/2}\). \!\!This yields \mbox{the result by renaming \(\sqrt \gamma\) with \(\gamma\).} 
\end{proof}

\begin{proof}[\textbf{Proof of Proposition \ref{prop:fundamental-example}}]
Let \(\mu_t\) be the measure given by Lemma \ref{lem:phys-measure} and \(\vf \in \cC^0(\cM)\). Denoting by \(\{V_x\}_{x \in [0,1]}\) the partition of the unit square into vertical lines and by \((x,y)\) the usual Cartesian coordinates in \(\cM\), by Fubini theorem, 
\begin{equation}\label{eq:Fubini}
\lim_{n \to \infty} m \left( \vf \circ \catt^n\right) = \lim_{n \to \infty}\int_{0}^1 \int_{V_x} \vf \circ \catt^n(x, y) dy dx.
\end{equation}
Since for each \(x \in [0,1]\) \(V_x\) are standard segments (they are in the cone) and \(|V_x| =1\), by the first part of Lemma \ref{lem:phys-measure},
\begin{equation}\label{eq:limit-vertical}
\lim_{n \to \infty} \int_{V_x} \vf \circ \catt^n(x,y) dy = \lim_{n \to \infty} \mu_{\cG_{V_x}} (\vf \circ \catt^n) =  \mu_t(\vf).
\end{equation}
Finally, by \eqref{eq:limit-vertical} and since \(\| \vf \circ \catt^n\|_{L^{\infty}} \le \|\vf\|_{\cC^0}\), we can apply Lebesgue dominated convergence theorem, and by \eqref{eq:Fubini} we have, for any \(\vf \in \cC^0(\cM)\),
\[
 \lim_{n \to \infty} m \left( \vf \circ \catt^n\right)= \int_{0}^1  \lim_{n \to \infty} \int_{V_x} \vf \circ \catt^n(x, y) dy dx = \mu_t (\vf).
\]
The equation above proves the first part of Proposition \ref{prop:fundamental-example}, while the second part is precisely Lemma \ref{lem:spc1}.
\end{proof}

\newpage

\appendix

\section{The partition \(\mathscr P_t^1\) for the perturbed cat map} \label{sec:app}
Here we give more details on the construction of the partition \(\mathscr P_t^1\) in the proof of Theorem \ref{thm:cat-lin-resp} and we prove \eqref{eq:C1-norm-conditional}. First, we have that 
\[
\cM \setminus (R_t \cup B_t \cup \cS_0^{-} \cup \cS_t^{-}) = A(t) \cup B(t) \cup C(t) \cup D(t) \cup E(t) \cup F(t) \cup G(t),
\]
where (see the rightmost square in \ref{Fig:2}),
\begin{equation}\label{eq:pedant}
\begin{split}
&A(t) = \{x>0, \text{ } y<1, \text{ } y<x+1-t, \text{ } y>2x\},\\
&B(t) = \{y<1,\text{ } y<2x, \text{ }y > (2-t)x\},\\
&C(t) = \{y<1,\text{ } y<(2-t)x,\text{ } y > x\},\\
&D(t) = \{y>0, \text{ }y< x, \text{ }y > x-t, \text{ }y> 2x-1\},\\
& E(t) = \{y>0, \text{ }y<x-t, \text{ }y>2x-1\},\\
&F(t) = \{y>0, \text{ }y<2x-1, \text{ }y < x-t, \text{ }y >(2-t)x -1\},\\
& G(t) = \{y>0,\text{ }x<1,\text{ } y<(2-t)x-1\}.
\end{split}
\end{equation}
We consider the set \(A(t)\) only, since the computations for the others are analogous as explained at the end of this appendix. Notice that the following is a \mbox{partition of \(A(t)\)},
\[
   \cA_t = \left\{\{y=\tan(\theta) x\}\cap A(t): \theta \in (\arctan(2), \pi/2)\right\}.
\]
We now compute the conditional and factor measure. Let \(g: \cM \to \bR\) be a bounded measurable function and set \(r(\theta, t) = |\{y = \tan(\theta)x\} \cap A(t)|\). Considering polar coordinates,
\begin{equation}\label{eq:polar}
\begin{split}
    \int_{A(t)}gdm &= \int_{\arctan(2)}^{\pi/2} \int_{0}^{r(\theta,t)} g(r,\theta)rdrd\theta \\
    &= \int_{\arctan(2)}^{\pi/2} \frac{\int_{0}^{r(\theta,t)} g(r,\theta)rdr}{\int_0^{r(\theta,t)} r dr } \left(\int_0^{r(\theta,t)} r dr \right)d\theta.
    \end{split}
\end{equation}
From \eqref{eq:polar}, we identify the disintegration of \(m\) and the corresponding conditional measures for the partition \(\cA_t\), 
\begin{equation}\label{eq:conditionalA}
   dm_{\cA_t}(\theta) = \left(\int_0^{r(\theta,t)} r dr \right)d\theta, \quad p_{\theta,t} (r) = \frac{r}{\int_0^{r(\theta,t)} r dr}.
\end{equation}
Since \(r(\theta,t) \ge 1/2\), for all \(t \in (0, 1/8)\), one has that 
\[
     \int_{0}^{r(\theta,t)}rdr = \frac{r(\theta,t)^2}{2} \ge \frac{1}{8}.
\]
This last estimate together with \eqref{eq:conditionalA} shows that \(\|p_{\theta,t}\|_{\cC^1}\) is bounded, uniformly in \(\theta\) and \(t\) as claimed in the proof of Theorem \ref{thm:cat-lin-resp}. It is possible to foliate all the sets in \eqref{eq:pedant} with lines that are either parallel or forming sectors of cones and we define \(\mathscr P_t^1\) in this way. By considering disintegrations either in Cartesian or polar coordinates, these partitions give rise to affine densities normalized over the length of the corresponding elements of the partition. Since there is a lower bound uniform in \(t\) of the above lengths, equation \eqref{eq:C1-norm-conditional} holds.


\end{document}